\documentclass[10pt,reqno,a4paper,twoside]{extarticle}

\usepackage{charter}
\usepackage[left=0.7in,right=0.7in,top=0.7in,bottom=0.7in]{geometry}

\usepackage{graphicx, varioref, amsfonts}
\usepackage{amsthm,amssymb,amsmath,mathrsfs,mathtools,stmaryrd} 
\usepackage{dsfont}
\usepackage{upgreek} 
\usepackage{esint} 
\usepackage{breakurl}
\usepackage{empheq}
\usepackage{enumerate,enumitem}
\usepackage{hyperref}
\usepackage{cleveref}
\usepackage{autonum}
\usepackage{soul}
\usepackage{cancel}
\usepackage{amscd}
\usepackage{srcltx}
\usepackage{ifthen}
\usepackage{float}
\usepackage{color}
\usepackage{comment}
\usepackage{chngcntr}
\usepackage{etoolbox}
\usepackage{fancyhdr}
\usepackage{breqn}
\usepackage{cite}
\usepackage[T1]{fontenc}
\usepackage{ulem}

\usepackage{titlefoot}

\usepackage{authblk} 
\theoremstyle{plain}
\newtheorem{lemma}{Lemma}[section]

\newtheorem{proposition}[lemma]{Proposition}

\newtheorem{theorem}[lemma]{Theorem}

\newtheorem{assumption}[lemma]{Assumption}

\theoremstyle{definition}

\newtheorem{definition}[lemma]{Definition}

\newtheorem{remark}[lemma]{Remark}

\newcommand{\R}{\ensuremath{\mathbb R}} 

\newcommand{\E}{\mathbb{E}}

\newlist{todolist}{itemize}{2}
\setlist[todolist]{label=$\square$}
\usepackage{pifont}

\numberwithin{equation}{section}
\begin{document}

\title{Peng's Maximum Principle for Stochastic Delay Differential Equations of Mean-Field Type}
\newcommand\shorttitle{Peng's Maximum Principle for SDDEs of Mean-Field Type}

\date{October 30, 2026}

\author{Giuseppina Guatteri\footnote{Dipartimento di Matematica, Politecnico di Milano, p.zza Leonardo da Vinci 32, 20133 Milano, Italy; Email: giuseppina.guatteri@gmail.com}}
\author{Federica Masiero\footnote{
Dipartimento di Matematica e Applicazioni, Universit\`a di Milano Bicocca, 
via Cozzi 55, 20125 Milano, Italy;
Email: federica.masiero@unimib.it}}
\author{Lukas Wessels\footnote{School of Mathematics, Georgia Institute of Technology, 686 Cherry Street, Atlanta, GA 30332, USA; Current address: Laboratoire J.A. Dieudonn\'e, Universit\'e C\^ote d'Azur, 06108 Nice, France; Email: lukas.wessels@univ-cotedazur.fr}}
\newcommand\authors{G. Guatteri, F. Masiero, and L. Wessels}

\affil{}

\maketitle

\unmarkedfntext{\textit{Mathematics Subject Classification (2020) ---} Primary 49N80; secondary 49K45, 93C23, 93E20.}


\unmarkedfntext{\textit{Keywords and phrases ---} Mean-field control, stochastic maximum principle, stochastic delay equations, McKean--Vlasov equations}


\begin{abstract}
    We extend Peng's maximum principle to the case of stochastic delay differential equations of mean-field type. More precisely, the coefficients of our control problem depend on the state, on the past trajectory and on its expected value. Moreover, the control enters the noise coefficient and the control domain may be non-convex. Our approach is based on a lifting of the state equation to an infinite dimensional Hilbert space that removes the explicit delay in the state equation. The main ingredient in the proof of the maximum principle is a precise asymptotics for the expectation of the first order variational process, which allows us to neglect the corresponding second order terms in the expansion of the cost functional.
\end{abstract}


\section{Introduction}


In this paper, we study an optimal control problem governed by stochastic delay differential equations (SDDEs) of mean-field type: the coefficients of the equation as well as the cost functional depend on the state, on the past trajectory of the state, and on its expected value. More precisely, we fix a finite time horizon $T>0$, a finite delay $d>0$, an initial condition $x_0\in \mathbb{R}^n$, $\mathbf{x}_0 : [-d,0]\to \mathbb{R}^n$, and consider the following optimization problem: Minimize the cost functional
\begin{equation}\label{cost_functional}
	J(u(\cdot)) = \mathbb{E} \left [ \int_0^T \tilde{l} \left ( t,x(t),\mathbf{x}(t),\mathbb{E}[x(t)], \mathbb{E}[\mathbf{x}(t)],u(t) \right ) \mathrm{d}t + \tilde{m} \left (x(T),\mathbf{x}(T),\mathbb{E}[x(T)],\mathbb{E}[\mathbf{x}(T)] \right ) \right ]
\end{equation}
subject to the controlled stochastic delay differential equation of mean-field type
\begin{equation}\label{state_equation}
\begin{cases}
	\mathrm{d}x(t) = \tilde{b}(t,x(t),\mathbf{x}(t),\mathbb{E}[x(t)],\mathbb{E}[\mathbf{x}(t)],u(t)) \mathrm{d}t + \tilde{\sigma}(t,x(t),\mathbf{x}(t),\mathbb{E}[x(t)],\mathbb{E}[\mathbf{x}(t)],u(t)) \mathrm{d}W(t) \\
    x(0) = x_0, \quad x(\theta)= \mathbf{x}_0(\theta),\quad \theta \in [-d,0].
\end{cases}
\end{equation}
Here, $\mathbf{x}(t):[-d,0]\to \mathbb{R}^n$ is given by $\mathbf{x}(t)(\theta) = x(t+\theta)$, $\theta \in [-d,0]$, and $(W(t))_{t\in [0,T]}$ is a Brownian motion with values in $\mathbb{R}^w$, $w\in \mathbb{N}$, on some probability space $(\Omega,\mathcal{F},(\mathcal{F}_t)_{t\in[0,T]}, \mathbb{P})$, where $(\mathcal{F}_t)$ is its natural filtration augmented by all $\mathbb{P}$-null sets. Moreover, the controls $u(\cdot)$ are progressively measurable, square-integrable processes taking values in some (not necessarily convex) set $U_{\text{ad}} \subset U$, for some real, separable Banach space $(U,\|\cdot\|_U)$. Precise assumptions on the coefficients $\tilde{l}$, $\tilde{m}$, $\tilde{b}$ and $\tilde{\sigma}$ will be given below in Section \ref{section_preliminaries}. We would like to emphasize, however, that the delay terms will be given by integration against square-integrable densities, see \eqref{b_sigma_tilde}--\eqref{l_m_tilde}.


Our objective in the present paper is the derivation of a necessary optimality condition for the aforementioned optimization problem. In \cite{pontryagin_1962}, Pontryagin and his coauthors derived a necessary optimality condition for controlled ordinary differential equations (ODEs) by employing a Taylor expansion of the cost functional up to first order, and then rewriting the resulting linear terms using duality. The Pontryagin maximum principle is then stated in terms of an adjoint state which is characterized as the solution of a backward ODE. This optimality condition can be generalized to a large class of controlled stochastic differential equations (SDEs) by characterizing the adjoint state as the solution of a backward stochastic differential equation (BSDE). However, the presence of the control in the noise coefficient $\tilde{\sigma}$ combined with the fact that the admissible controls take values in a non-convex control domain introduces another layer of difficulty. The main novelty stems from the fact that in this case one has to use Taylor expansions for the cost functional up to second order. This raises the question of how the additional quadratic terms in this expansion can be handled. In the case of controlled SDEs, in his seminal work \cite{peng_1990}, Peng solved this problem by linearizing the quadratic terms using the tensor product, which can be identified with a matrix, and then deriving a matrix-valued second order adjoint equation. Subsequently, Peng's work has been generalized in numerous directions, in particular to controlled mean-field equations, controlled SDDEs, and controlled infinite-dimensional stochastic systems.


The case of controlled mean-field equations was, to the best of our knowledge, first discussed in \cite{andersson_djehiche_2011}. In this paper, the authors prove a stochastic maximum principle in terms of a first order adjoint state. This result was extended in \cite{carmona_delarue_2015} to include more general dependence on the law of the solution of the state equation, see also \cite[Chapter 6]{carmona_delarue_2018}. In \cite{buckdahn_djehiche_li_2011,buckdahn_li_ma_2016}, the authors prove versions of Peng's maximum principle for controlled mean-field SDEs. The main difficulty in the proof of Peng's maximum principle for controlled mean-field SDEs is that the second order Taylor expansion leads to quadratic terms involving the square of the expectation. For these terms, it is not clear how they can be linearized via the tensor product. To resolve this issue, in \cite{buckdahn_djehiche_li_2011,buckdahn_li_ma_2016} the authors made the observation that these quadratic terms are actually of lower order and therefore can be neglected.


For control problems governed by SDDEs, there is an extensive body of literature on necessary optimality conditions that involve only the first order adjoint state, see e.g. \cite{chen_wu_2010,guatteri_masiero_2021,guo_xiong_zheng_2024,huang_shi_2012,oksendal_sulem_zhang_2011,peng_yang_2009,yu_2012,zhang_2021}. In the aforementioned papers the stochastic maximum principle for problems with delay is formulated by means of anticipated BSDEs which were introduced in the seminal paper \cite{peng_yang_2009}. We emphasize that in those works the analysis is confined to convex control spaces or to control-independent noise, so that only a first-order adjoint equation is needed to formulate the maximum principle. In the recent papers \cite{guatteri_masiero_2023,guatteri_masiero_2024,meng_shi_2021,meng_shi_wang_zhang_2025}, the authors consider the case of non-convex control domains and control-dependent noise. In \cite{guatteri_masiero_2024}, the authors exploit the specific structure of the model under consideration. In \cite{guatteri_masiero_2023}, the delay equation is reformulated in an infinite dimensional space, which is roughly speaking the product space of the present and of the past trajectory, see Section \ref{sec-inf-ref} for more details on this lifting. This reformulation restores Markovianity of the state equation. However, as discussed below, the infinite dimensionality of the state space introduces new challenges. In \cite{meng_shi_2021}, the authors introduce an additional second order adjoint equation, the solution of which is assumed to be zero. In \cite{meng_shi_wang_zhang_2025}, the stochastic maximum principle is formulated by means of Volterra equations.


The generalization of Peng's maximum principle to infinite-dimensional stochastic systems has been investigated by many authors, see \cite{du_meng_2013,fuhrman_hu_tessitore_2012,fuhrman_hu_tessitore_2013,frankowska_zhang_2020,lu_zhang_2014,lu_zhang_2015,lu_zhang_2018,stannat_wessels_2021,tang_li_1994}. One of the main difficulties in this case is the characterization of the second order adjoint state. While in finite dimensions, the second order adjoint state is matrix-valued, in infinite dimensions it becomes operator-valued. The natural state space for this process is the space of bounded linear operators. However, the fact that this is a non-reflexive Banach space introduces various technical difficulties for the duality theory as well as the stochastic integration theory. One way to circumvent these difficulties is to work in the space of Hilbert--Schmidt operators instead, which on the other hand limits the scope of applications. For a more detailed discussion of the literature regarding Peng's maximum principle for controlled stochastic partial differential equations, see also \cite{wessels_2022}.


In recent years, the study of optimal control problems governed by infinite dimensional systems of mean-field type has gained significant attention, both using the dynamic programming approach (see \cite{cosso_gozzi_kharroubi_pham_rosestolato_2023,defeo_gozzi_swiech_wessels_2025,djehiche_gozzi_zanco_zanella_2022,gozzi_masiero_rosestolato_2024}) as well as via the stochastic maximum principle (see \cite{agram_oksendal_2019,buckdahn_li_li_xing_2025,dumitrescu_oksendal_sulem_2018,shi_wang_yong_2013,spille_stannat_2025,tang_meng_wang_2019}). These works study various infinite dimensional models including systems with delays or path-dependence and systems governed by stochastic partial differential equations. However, all the papers studying the stochastic maximum principle work in the framework of convex control domains, thus only requiring a first order adjoint state.


In the present paper, we study the control problem \eqref{cost_functional}-\eqref{state_equation} governed by an SDDE with control in the noise coefficient and non-convex control domains. We use the lifting approach for SDDEs as in \cite{guatteri_masiero_2023} to reformulate the problem in an infinite dimensional space. This turns the problem into a mean-field control problem in infinite dimensions. Then, as in the classical approach, we use Taylor expansions to expand the cost functional. We employ second order expansions in the state (and the delay), however, we only use first order expansions in the expectation (and the expectation of the delay). In this approach the crucial point is the estimate in Proposition \ref{proposition_asymptotic_expectation_new} on the mean of the first order variation process, given by equation \eqref{first_variational_equation}. The method of Buckdahn et al., see \cite[Lemma 3.2]{buckdahn_djehiche_li_2011}, cannot be directly extended to the infinite-dimensional setting, as the two auxiliary equations they introduce are not easily generalizable. To overcome this limitation, we construct a family of operator-valued BSDEs \eqref{bsde_dual} and apply a duality argument to estimate the expected value of the first variation of the state. This method not only allows us to generalize the asymptotics to our infinite dimensional setting, but also to sharpen it. These refined asymptotics enable us to relax the assumptions on the coefficients of the control problem: we do not require a second order derivative in the argument that takes the expectation as an input.

Once these improved asymptotics are established, the proof follows the same route as the proof in the classical case: we use Taylor expansions of the cost functional, and use the adjoint states to handle the resulting terms. Along the way, we derive the well-posedness of general McKean--Vlasov BSDEs in infinite dimensional spaces, see Theorem \ref{BSDE_existence_of_solution}, which is needed for the first order adjoint equation \eqref{first_order_adjoint} as well as the family of operator-valued BSDEs \eqref{bsde_dual}. Moreover, we derive an equation for the tensor product of the first order variational equation, see Theorem \ref{theorem_tensor_product}, which seems to be new in this generality. Both of these results are of independent interest.

Finally, let us mention that our method exploits the special structure of the lifted infinite-dimensional control problem: The drift and the noise coefficients take values in the finite dimensional component of the state space, which leads to first order Fr\'echet derivatives that take values in the space of Hilbert--Schmidt operators, see Lemma \ref{lemma_B_Sigma_Lipschitz}. Moreover, the special structure of the running and terminal cost results in second order Fr\'echet derivatives which are Hilbert--Schmidt, see Lemma \ref{lemma_properties_L_M}. This allows us to circumvent some of the previously discussed issues that arise in the study of control problems governed by other infinite dimensional stochastic systems. The extension of our methods to more general cases is left for future work.


The paper is organized as follows. Section \ref{section_preliminaries} states the main assumptions, introduces the infinite-dimensional reformulation of the delay problem, and establishes basic well-posedness results. Section \ref{section_variational_equations} is devoted to the variational equations and to the asymptotic expansion of the cost functional.
In Section \ref{section_adjoint_equations} we define the adjoint equations and develop the associated duality relations, while Section \ref{section_stochastic_maximum_principle} presents the stochastic maximum principle. Finally, Appendices \ref{appendix_McKean_Vlasov_BSDE} and \ref{appendix_Operator_Valued_BSDE} contain auxiliary results on McKean--Vlasov backward equations and on operator-valued stochastic equations. These results are formulated in a more general framework and may also be applied beyond the specific setting of the delay reformulation considered here.

\section{Preliminaries}\label{section_preliminaries}

\subsection{Notation}

Throughout the paper, we are going to use the following notation.

\begin{itemize}
    \item Let $H = \mathbb{R}^n \oplus L^2([-d,0];\mathbb{R}^n)$. We denote the inner product and norm on $H$ by $\langle \cdot,\cdot \rangle$ and $\|\cdot\|$, respectively. We denote the components of a generic element $X\in H$ by $X=(x,\mathbf{x})$, i.e., $x\in \mathbb{R}^n$ and $\mathbf{x}\in L^2([-d,0];\mathbb{R}^n)$.
    \item Let $W^{1,2}([-d,0];\mathbb{R}^n)$ be the Sobolev space of order $1$, i.e., the subset of functions in $L^2([-d,0];\mathbb{R}^n)$ that have one weak derivative in $L^2([-d,0];\mathbb{R}^n)$.
    \item For two Hilbert spaces $H_1$ and $H_2$, let $L(H_1,H_2)$ denote the space of all bounded linear operators from $H_1$ to $H_2$ endowed with the norm $\| T \|_{L(H_1,H_2)} = \sup_{h\in H_1, \|h\|\leq 1} \|Th\|_{H_2}$, $T\in L(H_1,H_2)$. Moreover, let $L_2(H_1,H_2)$ denote the space of all Hilbert--Schmidt operators from $H_1$ to $H_2$. The inner product on $L_2(H_1,H_2)$ is given as follows: Let $(e_j)_{j\in J}$ be an orthonormal basis of $H_1$, and let $S,T \in L_2(H_1,H_2)$. Then $\langle S,T\rangle_{L_2(H_1,H_2)} := \sum_{j\in J} \langle S e_j, T e_j \rangle_{H_2}$. Let $\| \cdot \|_{L_2(H_1,H_2)}$ denote the induced norm. Note that if $H_1$ is finite dimensional, we have $\| \cdot \|_{L_2(H_1,H_2)} \leq C \| \cdot \|_{L(H_1,H_2)}$ for some constant $C$ depending only on the dimension of $H_1$. For $H_1=H_2$, we write $L(H_1)= L(H_1,H_1)$ and $L_2(H_1) = L_2(H_1,H_1)$.
    \item For some Banach space $(E,\|\cdot\|_E)$, let $L^2_{\mathcal{F}}(\Omega\times [0,T];E)$ denote the set of all $(\mathcal{F}_t)$-progressively measurable processes with values in $E$ such that $\| X\|_{L_{\mathcal{F}}^2(\Omega\times [0,T];E)} := \mathbb{E} [ \int_0^T \|X(t)\|_E^2 \mathrm{d}t ]^{1/2} < \infty$. Let $L^2_{\mathcal{F}}(\Omega; C([0,T];E))$ denote the set of all $(\mathcal{F}_t)$-progressively measurable processes with values in $E$ such that $\| X\|_{L_{\mathcal{F}}^2(\Omega; C( [0,T];E))} := \mathbb{E} [ \sup_{t \in [0,T]} \|X(t)\|_E^2]^{1/2} < \infty$.
    \item For Hilbert spaces $H_1$ and $H_2$ and a function $\phi : H_1\to H_2$, we denote by $\phi_X : H_1 \to L(H_1,H_2)$ and $\phi_{XX}: H_1\to L(H_1\times H_1,H_2)$ the first and second Fr\'echet derivatives, respectively. For $X,Y,Z\in H_1$, we denote by $\phi_{XX}(X)YZ = \phi_{XX}(X)(Y,Z)$, and if $Y=Z$, we denote $\phi_{XX}(X)Y^2 = \phi_{XX}(X)YY$.
\end{itemize}

\subsection{Assumptions}

In order to reformulate our problem on a space of square-integrable functions, we assume that the delay terms are given in terms of an integration against square integrable densities. More precisely, let us assume that $\tilde{b}: [0,T]\times \mathbb{R}^n \times L^2([-d,0];\mathbb{R}^n) \times \mathbb{R}^n \times L^2([-d,0];\mathbb{R}^n) \times U_{\text{ad}} \to \mathbb{R}^n$ and $\tilde{\sigma}: [0,T]\times \mathbb{R}^n \times L^2([-d,0];\mathbb{R}^n) \times \mathbb{R}^n \times L^2([-d,0];\mathbb{R}^n) \times U_{\text{ad}} \to L(\mathbb{R}^w,\mathbb{R}^n)$ are of the form
\begin{equation}\label{b_sigma_tilde}
\begin{split}
    \tilde{b}(t,x,\mathbf{x},y,\mathbf{y},u) &= b \left ( t,x,\int_{-d}^0 \mathbf{x}(\theta) f_{bx}(\theta) \mathrm{d}\theta,y, \int_{-d}^0 \mathbf{y}(\theta) f_{by}(\theta) \mathrm{d}\theta ,u\right ) \\
    \tilde{\sigma}(t,x,\mathbf{x},y,\mathbf{y},u) &= \sigma \left ( t,x,\int_{-d}^0 \mathbf{x}(\theta) f_{\sigma x}(\theta) \mathrm{d}\theta,y, \int_{-d}^0 \mathbf{y}(\theta) f_{\sigma y}(\theta) \mathrm{d}\theta ,u\right )
\end{split}
\end{equation}
for some $b: [0,T]\times \mathbb{R}^n \times \mathbb{R}^n \times \mathbb{R}^n\times \mathbb{R}^n \times U_{\text{ad}} \to \mathbb{R}^n$ and $\sigma : [0,T]\times \mathbb{R}^n \times \mathbb{R}^n \times \mathbb{R}^n\times \mathbb{R}^n \times U_{\text{ad}} \to L(\mathbb{R}^w,\mathbb{R}^n)$, and $f_{bx}, f_{by},f_{\sigma x},f_{\sigma y}\in L^2([-d,0];\mathbb{R})$. Analogously, we assume that $\tilde{l}: [0,T]\times \mathbb{R}^n \times L^2([-d,0];\mathbb{R}^n) \times \mathbb{R}^n \times L^2([-d,0];\mathbb{R}^n) \times U_{\text{ad}} \to \mathbb{R}$ and $\tilde{m}: \mathbb{R}^n \times L^2([-d,0];\mathbb{R}^n) \times \mathbb{R}^n \times L^2([-d,0];\mathbb{R}^n) \to \mathbb{R}$ are of the form
\begin{equation}\label{l_m_tilde}
\begin{split}
    \tilde{l}(t,x,\mathbf{x},y,\mathbf{y},u) &= l \left ( t,x,\int_{-d}^0 \mathbf{x}(\theta) f_{lx}(\theta) \mathrm{d}\theta,y, \int_{-d}^0 \mathbf{y}(\theta) f_{ly}(\theta) \mathrm{d}\theta ,u\right )\\
    \tilde{m}(x,\mathbf{x},y,\mathbf{y}) &= m \left ( x,\int_{-d}^0 \mathbf{x}(\theta) f_{mx}(\theta) \mathrm{d}\theta,y, \int_{-d}^0 \mathbf{y}(\theta) f_{my}(\theta) \mathrm{d}\theta \right )
\end{split}
\end{equation}
for some $l: [0,T]\times \mathbb{R}^n \times \mathbb{R}^n\times \mathbb{R}^n \times \mathbb{R}^n \times U_{\text{ad}} \to \mathbb{R}$ and $m : \mathbb{R}^n \times \mathbb{R}^n\times \mathbb{R}^n \times \mathbb{R}^n \to \mathbb{R}$, and $f_{lx},f_{ly},f_{mx},f_{my}\in L^2([-d,0];\mathbb{R})$.

\begin{remark}
    Our approach does not explicitly require the coefficients to depend on the past via integration against square-integrable densities. However, to introduce the infinite-dimensional reformulation described in Subsection \ref{sec-inf-ref}, the coefficients must depend on the past trajectory viewed as an element of $L^2([-d,0];\mathbb{R}^n)$, which specifically excludes the case of pointwise delay.\footnote{For a way to include linear pointwise delay in the infinite dimensional reformulation, we refer to \cite[Section 2.6.8.1]{fabbri_gozzi_swiech_2017}.} Moreover, we require certain Hilbert--Schmidt conditions on the derivatives of the coefficients (see Lemmas \ref{lemma_B_Sigma_Lipschitz} and \ref{lemma_properties_L_M}); in our setting, these conditions follow directly from the specific structure of the coefficients.
\end{remark}

Now, let us impose our assumptions on the coefficients of the control problem.

\begin{assumption}\label{assumption_b_sigma}
    The coefficients $b:[0,T]\times \mathbb{R}^n \times \mathbb{R}^n \times \mathbb{R}^n \times \mathbb{R}^n \times U_{\text{ad}} \to \mathbb{R}^n$ and $\sigma: [0,T]\times \mathbb{R}^n\times \mathbb{R}^n\times \mathbb{R}^n \times \mathbb{R}^n \times U_{\text{ad}} \to L(\mathbb{R}^w,\mathbb{R}^n)$ admit derivatives $b_x$, $\sigma_x$, $b_{\tilde{x}}$, $\sigma_{\tilde{x}}$, $b_y$, $\sigma_y$, $b_{\tilde{y}}$, $\sigma_{\tilde{y}}$, $b_{xx}$, $\sigma_{xx}$, $b_{\tilde{x}\tilde{x}}$, $\sigma_{\tilde{x}\tilde{x}}$, $b_{x\tilde{x}}$ and $\sigma_{x\tilde{x}}$. The functions $b$ and $\sigma$ as well as all these derivatives are continuous in the variables $(t,x,\tilde{x},y,\tilde{y},u)$. The functions $b_x$, $\sigma_x$, $b_{\tilde{x}}$, $\sigma_{\tilde{x}}$, $b_y$, $\sigma_y$, $b_{\tilde{y}}$, $\sigma_{\tilde{y}}$, $b_{xx}$, $\sigma_{xx}$, $b_{\tilde{x}\tilde{x}}$, $\sigma_{\tilde{x}\tilde{x}}$, $b_{x\tilde{x}}$ and $\sigma_{x\tilde{x}}$ are bounded, and $b$, $\sigma$ are bounded by $C(1+|x|+|\tilde{x}|+|y|+|\tilde{y}|)$.
\end{assumption}

\begin{assumption}\label{assumption_l_m}
    The coefficients $l:[0,T]\times \mathbb{R}^n\times \mathbb{R}^n\times \mathbb{R}^n \times \mathbb{R}^n \times U_{\text{ad}} \to \mathbb{R}$ and $m:\mathbb{R}^n \times\mathbb{R}^n \times\mathbb{R}^n \times \mathbb{R}^n \to \mathbb{R}$ admit derivatives $l_x$, $m_x$, $l_{\tilde{x}}$, $m_{\tilde{x}}$, $l_y$, $m_y$, $l_{\tilde{y}}$, $m_{\tilde{y}}$, $l_{xx}$, $m_{xx}$, $l_{\tilde{x}\tilde{x}}$, $m_{\tilde{x}\tilde{x}}$, $l_{x\tilde{x}}$ and $m_{x\tilde{x}}$. The functions $l$ and $m$ as well as all these derivatives are continuous in the variables $(t,x,\tilde{x},y,\tilde{y},u)$ and $(x,\tilde{x},y,\tilde{y})$, respectively. The functions $l_{xx}$, $m_{xx}$, $l_{\tilde{x}\tilde{x}}$, $m_{\tilde{x}\tilde{x}}$, $l_{x\tilde{x}}$ and $m_{x\tilde{x}}$ are bounded, $l_x$, $m_x$, $l_{\tilde{x}}$, $m_{\tilde{x}}$, $l_y$, $m_y$, $l_{\tilde{y}}$, $m_{\tilde{y}}$ are bounded by $C(1+|x| + |\tilde{x}| +|y| + |\tilde{y}| + \|u\|_U)$, and $|l(t,0,0,0,0,u)| \leq C(1+\|u\|_U^2)$.
\end{assumption}

\subsection{Infinite Dimensional Reformulation}\label{sec-inf-ref}

Now, we are in a position to introduce the infinite dimensional formulation. Let $H=\mathbb{R}^n \oplus L^2([-d,0];\mathbb{R}^n)$, where the inner product between $X=(x,\mathbf{x}) \in H$ and $Y=(y,\mathbf{y}) \in H$ is given by $\langle X,Y \rangle = x \cdot y + \langle \mathbf{x},\mathbf{y} \rangle_{L^2([-d,0];\mathbb{R}^n)}$ and the norm is denoted by $\| X \| = \sqrt{\langle X,X \rangle}$. By \cite[Part II, Chapter 4, Theorem 4.2]{bensoussan_daprato_delfour_mitter_2007}, the operator $A : \mathcal{D}(A) \subset H \to H$
\begin{equation}\label{operator_A}
    \mathcal{D}(A) = \left \{ X=\begin{pmatrix} x \\ \mathbf{x} \end{pmatrix} \in H,\;\mathbf{x} \in W^{1,2}([-d,0];\mathbb{R}^n),\;\mathbf{x}(0) = x \right \},\quad AX = A\begin{pmatrix} x\\\mathbf{x} \end{pmatrix} = \begin{pmatrix} 0\\\mathrm{d}\mathbf{x}/\mathrm{d}\theta \end{pmatrix},
\end{equation}
is the infinitesimal generator of a $C_0$-semigroup which we denote by $(e^{tA})_{t\geq 0}$. Using this operator, the state equation \eqref{state_equation} can be lifted to an equation for $X(t) = \begin{pmatrix} x(t)\\ \mathbf{x}(t) \end{pmatrix} = \begin{pmatrix} x(t)\\ (x(t+\theta))_{\theta \in [-d,0]} \end{pmatrix} \in H$ of the form
\begin{equation}\label{lifted_state_equation}
\begin{cases}
    \mathrm{d}X(t) = [ AX(t) + B(t,X(t), \mathbb{E}[X(t)],u(t)) ] \mathrm{d}t + \Sigma(t,X(t),\mathbb{E}[X(t)],u(t)) \mathrm{d}W(t), \quad t\in [0,T]\\
    X(0) = X_0 = \begin{pmatrix} x_0 \\ (\mathbf{x}_0(\theta))_{\theta\in [-d,0]} \end{pmatrix} \in H,
\end{cases}
\end{equation}
where $B:[0,T]\times H \times H \times U_{\text{ad}} \to H$ and $\Sigma:[0,T]\times H \times H \times U_{\text{ad}} \to L(\mathbb{R}^w,H)$ are given by
\begin{equation}\label{B_and_Sigma}
\begin{split}
    B(t,X,Y,u) &:= G b\left (t,x,\int_{-d}^0 \mathbf{x}(\theta) f_{bx}(\theta) \mathrm{d}\theta,y, \int_{-d}^0 \mathbf{y}(\theta) f_{by}(\theta) \mathrm{d}\theta , u \right )\\
    \Sigma(t,X,Y,u) &:= G \sigma \left (t,x,\int_{-d}^0 \mathbf{x}(\theta) f_{\sigma x}(\theta) \mathrm{d}\theta,y, \int_{-d}^0 \mathbf{y}(\theta) f_{\sigma y}(\theta) \mathrm{d}\theta , u \right )
\end{split}
\end{equation}
for $t\in [0,T]$, $X= ( x,\mathbf{x} )\in H$, $Y= ( y,\mathbf{y} ) \in H$, $u\in U_{\text{ad}}$, and
\begin{equation}
    G : \mathbb{R}^n \to H,\quad G = \begin{pmatrix} I\\ 0 \end{pmatrix}.
\end{equation}

\begin{remark}
    We will show in Proposition \ref{well_posedness_lifted_state_equation} below that, under Assumption \ref{assumption_b_sigma}, the lifted state equation \eqref{lifted_state_equation} admits a unique mild solution.
\end{remark}

We have the following equivalence between the finite dimensional SDDE and the lifted equation, see \cite[Theorem 3.4]{federico_tankov_2015}.
\begin{proposition}
    Let Assumption \ref{assumption_b_sigma} be satisfied. For an initial condition $(x_0,\mathbf{x}_0) \in H$ and an admissible control $u(\cdot) \in L^2_{\mathcal{F}}([0,T]\times \Omega; U_{\text{ad}})$, let $x(\cdot)$ be the solution of the SDDE \eqref{state_equation} and let $X(\cdot)$ be the mild solution of the lifted equation \eqref{lifted_state_equation}, i.e., for every $t\in [0,T]$,
    \begin{equation}
        X(t)= e^{ t A}X_0 + \int_0^t e^{ (t-s) A} B(s,X(s),\mathbb{E}[X(s)],u(s)) \, \mathrm{d}s +  \int_0^t e^{ (t-s) A} \Sigma(s,X(s),\mathbb{E}[X(s)],u(s)) \mathrm{d}W(s)
    \end{equation}
    $\mathbb{P}$-almost surely. Then for all $t\in [0,T]$, $X(t) = (x(t),\mathbf{x}(t))$, where we recall that $\mathbf{x}(t)(\theta)=x(t+\theta)$, $\theta \in [-d,0]$.
\end{proposition}

The cost functional can be rewritten as
\begin{equation}\label{lifted-cost}
    J(u(\cdot)) := \mathbb{E} \left [ \int_0^T L(t,X(t),\mathbb{E}[X(t)],u(t)) \mathrm{d}t + M(X(T),\mathbb{E}[X(T)]) \right ]
\end{equation}
where $L:[0,T]\times H \times H \times U_{\text{ad}} \to \mathbb{R}$ and $M : H \times H \to \mathbb{R}$ are given by
\begin{equation}\label{L_and_M}
\begin{split}
    L(t,X,Y,u) &:= l\left (t,x,\int_{-d}^0 \mathbf{x}(\theta) f_{lx}(\theta) \mathrm{d}\theta, y, \int_{-d}^0 \mathbf{y}(\theta) f_{ly}(\theta) \mathrm{d}\theta , u \right )\\
    M(X,Y) &:= m \left (x,\int_{-d}^0 \mathbf{x}(\theta) f_{mx}(\theta) \mathrm{d}\theta, y , \int_{-d}^0 \mathbf{y}(\theta) f_{my}(\theta) \mathrm{d}\theta \right )
\end{split}
\end{equation}
for $t\in [0,T]$, $X= ( x,\mathbf{x} ) \in H$, $Y= ( y,\mathbf{y} ) \in H$, $u\in U_{\text{ad}}$.

Thus, the control problem \eqref{cost_functional}-\eqref{state_equation} is equivalent with the following control problem: Minimize
\begin{equation}
    J(u(\cdot)) := \mathbb{E} \left [ \int_0^T L(t,X(t),\mathbb{E}[X(t)],u(t)) \mathrm{d}t + M(X(T),\mathbb{E}[X(T)]) \right ]
\end{equation}
subject to
\begin{equation}
\begin{cases}
    \mathrm{d}X(t) = [ AX(t) + B(t,X(t), \mathbb{E}[X(t)],u(t)) ] \mathrm{d}t + \Sigma(t,X(t),\mathbb{E}[X(t)],u(t)) \mathrm{d}W(t), \quad t\in [0,T]\\
    X(0) = X_0 = \begin{pmatrix} x_0 \\ (\mathbf{x}_0(\theta))_{\theta\in [-d,0]} \end{pmatrix} \in H,
\end{cases}
\end{equation}
over the set of admissible controls $L^2_{\mathcal{F}}([0,T]\times\Omega; U_{\text{ad}})$.

\subsection{Preliminary Results for the Infinite Dimensional Problem}

In this section, we establish the regularity properties of the coefficients of our reformulated problem in infinite dimensions and derive the well-posedness of the lifted state equation \eqref{lifted_state_equation}. Let us start with the unbounded operator $A$.

\begin{lemma}\label{lemma_A_pseudo_contraction}
    The operator $A: \mathcal{D}(A) \subset H \to H$ introduced in \eqref{operator_A} generates a $C_0$-semigroup of pseudo-contractions.
\end{lemma}

\begin{proof}
    By \cite[Part II, Chapter 4, Theorem 4.2]{bensoussan_daprato_delfour_mitter_2007}, we know that $A$ generates a $C_0$-semigroup. We are going to show that $A-\text{Id}$ generates a $C_0$-semigroup of contractions, which then implies that $A$ generates a $C_0$-semigroup of pseudo-contractions. To this end, all we need to show is that $A-\text{Id} : \mathcal{D}(A) \subset H \to H$ is dissipative. For $X\in \mathcal{D}(A)$, we have
    \begin{equation}
    \begin{split}
        \langle (A-\text{Id})X,X \rangle &= - | x|^2 + \int_{-d}^0 \frac{\mathrm{d} \mathbf{x}}{\mathrm{d}\theta}(\theta) \cdot \mathbf{x}(\theta) \mathrm{d}\theta - \| \mathbf{x} \|_{L^2([-d,0];\mathbb{R}^n)}^2\\
        &= - |x|^2 + \left [ \frac{|\mathbf{x}(\theta) |^2}{2} \right ]_{-d}^0 - \| \mathbf{x} \|_{L^2([-d,0];\mathbb{R}^n)}^2\\
        &= - \frac{|x|^2}{2} - \frac{| \mathbf{x}(-d) |^2}{2} - \| \mathbf{x} \|_{L^2([-d,0];\mathbb{R}^n)}^2 \leq 0,
    \end{split}
    \end{equation}
    where we used the fact that $\mathbf{x}(0) = x$ for all $X= (x,\mathbf{x}) \in \mathcal{D}(A)$.
\end{proof}

Next, we deduce properties of the coefficients $B$ and $\Sigma$ of the lifted equation \eqref{lifted_state_equation} that follow from our Assumption \ref{assumption_b_sigma} on the coefficients $\tilde{b}$ and $\tilde{\sigma}$ of the delay state equation \eqref{state_equation}.

\begin{lemma}\label{lemma_B_Sigma_Lipschitz}
    Let Assumption \ref{assumption_b_sigma} be satisfied. Let $B:[0,T]\times H \times H \times U_{\text{ad}} \to H$ and $\Sigma:[0,T]\times H\times H \times U_{\text{ad}} \to L(\mathbb{R}^w,H)$ be the coefficients introduced in \eqref{B_and_Sigma}. Then, there is a constant $C>0$ such that
    \begin{align}
        \label{lip_B} \| B(t,X,Y,u) - B(t,X',Y',u) \| &\leq C (\|X-X'\| + \|Y-Y'\| )\\
        \label{lip_Sigma}\| \Sigma(t,X,Y,u) - \Sigma(t,X',Y',u) \|_{L(\mathbb{R}^w,H)} &\leq C (\|X-X'\| + \|Y-Y'\|)\\
        \label{bound_in_control} \|B(t,0,0,u)\| + \| \Sigma(t,0,0,u) \|_{L(\mathbb{R}^w,H)} &\leq C
    \end{align}
    for all $t\in [0,T]$, $X,X',Y,Y' \in H$ and $u\in U_{\text{ad}}$. Moreover, $B$ and $\Sigma$ admit Fr\'echet derivatives $B_X$, $\Sigma_X$, $B_Y$, $\Sigma_Y$, $B_{XX}$ and $\Sigma_{XX}$. The functions $B$ and $\Sigma$ as well as these derivatives are continuous in the variables $(t,X,Y,u)$, and $B_X$, $\Sigma_X$, $B_Y$, $\Sigma_Y$, $B_{XX}$, $\Sigma_{XX}$ are bounded. Finally, $B_X(t,X,Y,u)$ is a Hilbert--Schmidt operator for all $t\in [0,T]$, $X,Y\in H$ and $u\in U_{\text{ad}}$, and $\| B_X(t,X,Y,u) \|_{L_2(H)}$ is uniformly bounded.
\end{lemma}

\begin{proof}
The Lipschitz properties \eqref{lip_B} and $\eqref{lip_Sigma}$, and the uniform bound in the control \eqref{bound_in_control} are straightforward consequences of the definitions of $B$ and $\Sigma$ in \eqref{B_and_Sigma} as well as Assumption \ref{assumption_b_sigma}. Regarding the differentiability, let us only treat the diffusion coefficient $\Sigma$; the drift $B$ can be treated in a similar way. In this proof, to shorten the notation, let us drop the arguments of the derivatives of $\sigma$, i.e., we write
\begin{equation}
    \sigma^{ij}_x := \sigma^{ij}_x \left ( t,x,\int_{-d}^0 \mathbf{x}(\theta) f_{\sigma x}(\theta) \mathrm{d}\theta,y, \int_{-d}^0 \mathbf{y}(\theta) f_{\sigma y}(\theta) \mathrm{d}\theta ,u\right ),
\end{equation}
for $i=1,\dots,n$ and $j=1,\dots, w$, and similarly for the other derivatives $\sigma_{\tilde{x}}$, $\sigma_y$, $\sigma_{\tilde{y}}$, $\sigma_{xx}$, $\sigma_{x\tilde{x}}$, $\sigma_{\tilde{x}x}$, $\sigma_{\tilde{x}\tilde{x}}$. We have the following explicit formulae for the derivatives: For every $Z=(z,\mathbf{z}) \in H$ and $j=1,\dots, w$,
\begin{align}
    \Sigma^j_X(t,X,Y,u) Z & = G\left ( \sigma^{ij}_x \cdot z + \sigma^{ij}_{\tilde{x}} \cdot \int_{-d}^0  f_{\sigma x}(\theta)  \mathbf{z}(\theta) \mathrm{d} \theta \right )_{i=1,\dots,n } \label{first_derivative_sigma}\\
    \Sigma^j_Y(t,X,Y,u) Z & = G\left ( \sigma^{ij}_y \cdot z + \sigma^{ij}_{\tilde{y}} \cdot \int_{-d}^0 f_{\sigma y}(\theta) \mathbf{z}(\theta) \mathrm{d} \theta \right )_{i=1,\dots,n }. 
\end{align} 
Moreover, regarding the second order derivatives, we have for $Z_1 = (z_1,\mathbf{z}_1), Z_2= (z_2,\mathbf{z}_2) \in H$, and $j=1,\dots,w$,
\begin{equation}
\begin{split}
    \Sigma_{XX}^j(t,X,Y,u) Z_1 Z_2 & = G \Bigg ( \sigma^{ij}_{x x} z_1 \cdot z_2 + \sigma^{ij}_{x\tilde{x}} \int_{-d}^0 f_{\sigma x}(\theta) \mathbf{z}_1(\theta) \mathrm{d} \theta \cdot z_2 + \sigma^{ij}_{x\tilde{x}} z_1 \cdot \int_{-d}^0 f_{\sigma x}(\theta) \mathbf{z}_2(\theta) \mathrm{d}\theta\\
    &\qquad\qquad\qquad + \sigma^{ij}_{\tilde{x}\tilde{x}} \int_{-d}^0 f_{\sigma x}(\theta_1) \mathbf{z}_1(\theta_1) \mathrm{d} \theta_1 \cdot \int_{-d}^0 f_{\sigma x}(\theta_2) \mathbf{z}_2(\theta_2) \mathrm{d} \theta_2 \Bigg )_{i=1,\dots,n }.
\end{split}
\end{equation}
In view of Assumption \ref{assumption_b_sigma} and since $f_{\sigma x},f_{\sigma y}\in L^2([-d,0];\mathbb{R})$, we have that $\Sigma^j_X (t,X,Y,u)$, $\Sigma^j_Y(t,X,Y,u) \in  L(H)$, and $\Sigma^j_{XX}(t,X,Y,u) \in  L(H\times H, H)$ and are bounded. A similar calculation for the derivatives of $B$ shows that $B_X(t,X,Y,u) : H \to H$ is an operator with finite rank, and thus Hilbert--Schmidt. The uniform bound on the Hilbert--Schmidt norm of $B_X$ follows from the explicit formula for $B_X$ corresponding to \eqref{first_derivative_sigma}, and the uniform bound on $b_x$ and $b_{\tilde{x}}$.
\end{proof}

Due to Lemmas \ref{lemma_A_pseudo_contraction} and \ref{lemma_B_Sigma_Lipschitz}, we can apply \cite[Proposition 2.8]{cosso_gozzi_kharroubi_pham_rosestolato_2023} to deduce the well posedness of the lifted state equation \eqref{lifted_state_equation}:
\begin{proposition}\label{well_posedness_lifted_state_equation}
    Let Assumption \ref{assumption_b_sigma} be satisfied. Then, for any $\mathcal{F}_0$-measurable $X_0\in L^2(\Omega;H)$, there exists a unique process $X(\cdot) \in L^2_{\mathcal{F}} (\Omega;C([0,T];H))$ that solves equation \eqref{lifted_state_equation} in a mild sense, that is, for every $t\in [0,T]$,
    \begin{equation}
        X(t)= e^{ t A}X_0 + \int_0^t e^{ (t-s) A} B(s,X(s),\mathbb{E}[X(s)],u(s)) \, \mathrm{d}s +  \int_0^t e^{ (t-s) A} \Sigma(s,X(s),\mathbb{E}[X(s)],u(s)) \mathrm{d}W(s)
    \end{equation}
    $\mathbb{P}$-almost surely. Moreover, there is a constant $C$ independent of $u$ such that
    \begin{equation}\label{estimate_state}
    \mathbb{E}\left [ \sup_{t \in [0,T]} \|X(t)\|^2 \right ] \leq C \left ( 1+ \mathbb{E} \left [ \|X_0\|^2 \right ] \right ).  
    \end{equation}
\end{proposition}

In the following lemma, we deduce properties of the running and terminal cost of the lifted cost functional \eqref{lifted-cost} that follow from our Assumption \ref{assumption_l_m} on the running and terminal cost in \eqref{cost_functional}.
\begin{lemma}\label{lemma_properties_L_M}
    Let Assumption \ref{assumption_l_m} be satisfied. Let $L:[0,T]\times H \times H \times U_{\text{ad}} \to \mathbb{R}$ and $M:H\times H \to \mathbb{R}$ be the coefficients introduced in \eqref{L_and_M}. Then, $L$ and $M$ admit Fr\'echet derivatives $L_X$, $M_X$, $L_Y$, $M_Y$, $L_{XX}$, $M_{XX}$. The functions $L$ and $M$ as well as these derivatives are continuous in the variables $(t,X,Y,u)$ and $(X,Y)$, respectively. Moreover, $L_{XX}$ and $M_{XX}$ are bounded, $L_X$, $M_X$, $L_Y$, $M_Y$ are bounded by $C(1+\|X\|+ \|Y\|+\|u\|_U)$, and $|L(t,0,0,u)| \leq C(1+\|u\|_U^2)$. Finally, $L_{XX}(t,X,Y,u)$ and $M_{XX}(X,Y)$ are Hilbert--Schmidt operators for all $t\in [0,T]$, $X,Y \in H$ and $u\in U_{\text{ad}}$.
\end{lemma}

\begin{proof} 
    As in the proof of Lemma \ref{lemma_B_Sigma_Lipschitz}, we have explicit formulae for the derivatives. Let us only consider the running cost $L$; the calculations for the terminal cost $M$ are similar. Throughout this proof, to shorten the notation, let us drop the arguments of the derivatives of $l$, i.e., we write
    \begin{equation}
        l_x := l_x \left ( t,x,\int_{-d}^0 \mathbf{x}(\theta) f_{l x}(\theta) \mathrm{d}\theta,y, \int_{-d}^0 \mathbf{y}(\theta) f_{l y}(\theta) \mathrm{d}\theta ,u\right )
    \end{equation}
    and similarly for the other derivatives
    $l_{\tilde{x}}$, $l_y$, $l_{\tilde{y}}$, $l_{xx}$, $l_{x\tilde{x}}$, $l_{\tilde{x}x}$, $l_{\tilde{x}\tilde{x}}$. For every $Z=(z,\mathbf{z})$, $Z_1 = (z_1, \mathbf{z}_1)$, $Z_2 =  (z_2, \mathbf{z}_2) \in H$ we have that
\begin{equation}
\begin{split}
    L_X(t,X,Y,u) Z& = l_x \cdot z +  l_{\tilde{x}} \cdot \int_{-d}^0 f_{lx}(\theta) \mathbf{z}(\theta)\mathrm{d} \theta \\
    L_Y(t,X,Y,u) Z &= l_y \cdot z +  l_{\tilde{y}} \cdot \int_{-d}^0  f_{ly}(\theta) \mathbf{z}(\theta) \mathrm{d}\theta \\ 
    L_{XX}(t,X,Y,u) Z_1 Z_2  &=  l_{xx} z_1 \cdot z_2 + l_{x\tilde{x}} \int_{-{d}}^0 f_{l x}(\theta) \mathbf{z}_1(\theta) \mathrm{d}\theta \cdot z_2 + l_{\tilde{x}x} z_1 \cdot \int_{-d}^0 f_{l x}(\theta) \mathbf{z}_2(\theta) \mathrm{d}\theta \\
    &\quad + l_{\tilde{x}\tilde{x}} \int_{-d}^0 f_{l x}(\theta _1) \mathbf{z}_1(\theta_1) \mathrm{d}\theta_1 \cdot \int_{-d}^0 f_{l x}(\theta_2) \mathbf{z}_2(\theta_2) \mathrm{d}\theta_2.
\end{split}
\end{equation}
In view of Assumption \ref{assumption_l_m} and since $f_{l x},f_{l y}\in L^2([-d,0];\mathbb{R})$, we have that for $t \in [0,T]$, $X,Y\in H$, $u\in U_{\text{ad}}$, $L_X(t,X,Y,u)$, $L_Y(t,X,Y,u)\in L(H,\R)$ are bounded by $C(1+\|X\|+\|Y\|+\|u\|_U)$, and $L_{XX}(t,X,Y,u) \in L(H \times H, \R)$ is bounded. Moreover, the bilinear form $L_{XX}(t,X,Y,u)$ can be identified with an operator from $H$ to $H$ which we also denote by $L_{XX}(t,X,Y,u): H \to H$, and which is given by
\begin{equation}
    L_{XX}(t,X,Y,u) Z = \begin{pmatrix}
        l_{xx} z + l_{x\tilde{x}} \int_{-d}^0 f_{lx}(\theta) \mathbf{z}(\theta) \mathrm{d}\theta \\
        \left ( l_{\tilde{x}x} z + l_{\tilde{x}\tilde{x}} \int_{-d}^0 f_{lx}(\theta) \mathbf{z}(\theta)\mathrm{d}\theta \right ) f_{lx}
    \end{pmatrix}.
\end{equation}
Since this operator has finite rank, it is Hilbert--Schmidt.
\end{proof}

\section{Variational Equations}\label{section_variational_equations}

For the remainder of this paper, let $u(\cdot)$ be an optimal control and let $X(\cdot)$ be the corresponding optimal trajectory. We introduce the following spike variation: For fixed $\tau\in (0,T)$, $0<\varepsilon<T-\tau$, $\mathfrak{F}\in \mathcal{F}_{\tau}$, and $v \in U_{\text{ad}}$, we set
\begin{equation}\label{spike_variation_definition}
\begin{split}
    u^{\varepsilon}(t,\omega) = \begin{cases}
        v,& t\in [\tau,\tau+\varepsilon],\omega \in \mathfrak{F}\\
        u(t,\omega),&\text{otherwise}.
    \end{cases}
\end{split}
\end{equation}
Let $X^{\varepsilon}(\cdot)$ denote the solution of the lifted state equation \eqref{lifted_state_equation} with control $u^{\varepsilon}(\cdot)$. For simplicity, we are going to use the following notation. For a function $\phi$ defined on $[0,T]\times H\times H\times U_{\text{ad}}$, let
\begin{equation}
\begin{split}\label{phi-notation}
    \phi(t) &:= \phi(t,X(t),\mathbb{E}[X(t)],u(t))\\
    \phi^{\varepsilon}(t) &:= \phi(t,X^{\varepsilon}(t),\mathbb{E}[X^{\varepsilon}(t)],u^{\varepsilon}(t))\\
    \delta^v\phi(t) &:= \phi(t,X(t),\mathbb{E}[X(t)],u^{\varepsilon}(t)) - \phi(t,X(t),\mathbb{E}[X(t)],u(t)).
\end{split}
\end{equation}
Now, we introduce the first and second order variational processes. Let $Y^{\varepsilon}(\cdot)$ be the solution of the first order variational equation
\begin{equation}\label{first_variational_equation}
\begin{cases}
    \mathrm{d}Y^{\varepsilon}(t) = [ AY^{\varepsilon}(t) + B_X(t) Y^{\varepsilon}(t) + B_{Y}(t) \mathbb{E} [ Y^{\varepsilon}(t) ] + \delta^v B(t) ] \mathrm{d}t\\
    \qquad\qquad\qquad + \left [ \Sigma_X(t) Y^{\varepsilon}(t) + \Sigma_{Y}(t) \mathbb{E} [ Y^{\varepsilon}(t) ] + \delta^v \Sigma(t) \right ] \mathrm{d}W(t),\quad t\in[0,T]\\
    Y^{\varepsilon}(0) = 0 \in H
\end{cases}
\end{equation}
and let $Z^{\varepsilon}(\cdot)$ be the solution of the second order variational equation
\begin{equation}\label{second_variational_equation}
\begin{cases}
    \mathrm{d}Z^{\varepsilon}(t) = \Big [ AZ^{\varepsilon}(t) + B_X(t) Z^{\varepsilon}(t) + B_{Y}(t) \mathbb{E} [ Z^{\varepsilon}(t) ] + \delta^v B_X(t) Y^{\varepsilon}(t) + \frac12 B_{XX}(t) Y^{\varepsilon}(t)^2 \Big ] \mathrm{d}t\\
    \qquad\qquad + \Big [ \Sigma_X(t) Z^{\varepsilon}(t) + \Sigma_{Y}(t) \mathbb{E} [ Z^{\varepsilon}(t) ] + \delta^v \Sigma_X(t) Y^{\varepsilon}(t) + \frac12 \Sigma_{XX}(t) Y^{\varepsilon}(t)^2 \Big ] \mathrm{d}W(t),\quad t\in[0,T]\\
    Z^{\varepsilon}(0) = 0 \in H.
\end{cases}
\end{equation}

\begin{remark}\label{remark_second_lions_derivative}
    In the second order variational equation \eqref{second_variational_equation} one might expect to see the second order derivative with respect to $Y$. However, it turns out that this term can be neglected since it is of higher order, see Proposition \ref{proposition_asymptotic_expectation_new} below. Dropping this term also seems to be crucial, since the expression involves the term $\mathbb{E}[ Y^{\varepsilon}(t) ]^2$ for which it is not clear how it can be linearized using the tensor product $Y^{\varepsilon}(t) \otimes Y^{\varepsilon}(t)$.
\end{remark}

\subsection{Preliminary Estimates}

In this subsection, we prove preliminary estimates for the first and second order variational processes.

\begin{proposition}
    Let Assumption \ref{assumption_b_sigma} be satisfied. Then, we have for all $j \geq 1$,
    \begin{align}
        \label{asymptotic_X_X_e} \mathbb{E} \left [ \sup_{t\in [0,T]} \| X(t) - X^{\varepsilon}(t) \|^{2j} \right ] &= \mathcal{O}(\varepsilon^j)\\
        \label{asymptotic_Y_e} \mathbb{E} \left [ \sup_{t\in [0,T]} \| Y^{\varepsilon}(t) \|^{2j} \right ] &= \mathcal{O}(\varepsilon^j)\\
        \label{asymptotic_Z_e} \mathbb{E} \left [ \sup_{t\in [0,T]} \| Z^{\varepsilon}(t) \|^{2j} \right ] &= \mathcal{O}(\varepsilon^{2j}).
    \end{align}
\end{proposition}

\begin{proof} 
    Let us start with inequality \eqref{asymptotic_X_X_e}. For $s\in [0,t]$, we have
    \begin{equation}\label{estimate_1_X_X_e}
    \begin{split}
        \| X(s) - X^{\varepsilon}(s) \|^2 &\leq 2 \left \| \int_0^s e^{(s-r)A} ( B(r,X(r),\mathbb{E}[X(r)],u(r)) - B(r,X^{\varepsilon}(r),\mathbb{E}[X^{\varepsilon}(r)], u^{\varepsilon}(r)) ) \mathrm{d}r \right \|^2 \\
        &\quad + 2 \left \| \int_0^s e^{(s-r)A} ( \Sigma(r,X(r),\mathbb{E}[X(r)],u(r)) - \Sigma(r,X^{\varepsilon}(r),\mathbb{E}[X^{\varepsilon}(r)], u^{\varepsilon}(r)) ) \mathrm{d}W(r) \right \|^2.
    \end{split}
    \end{equation}
    Using \cite[Lemma 3.3]{gawarecki_mandrekar_2011}, we obtain for the stochastic integral
    \begin{equation}
    \begin{split}
        &\mathbb{E} \left [ \sup_{s\in [0,t]} \left \| \int_0^s e^{(s-r)A} ( \Sigma(r,X(r),\mathbb{E}[X(r)],u(r)) - \Sigma(r,X^{\varepsilon}(r),\mathbb{E}[X^{\varepsilon}(r)], u^{\varepsilon}(r)) ) \mathrm{d}W(r) \right \|^{2j} \right ]\\
        &\leq C \mathbb{E} \left [ \left ( \int_0^t \left \| \Sigma(r,X(r),\mathbb{E}[X(r)],u(r)) - \Sigma(r,X^{\varepsilon}(r),\mathbb{E}[X^{\varepsilon}(r)], u^{\varepsilon}(r)) \right \|_{L(\mathbb{R}^w,H)}^2 \mathrm{d}r \right )^j \right ].
    \end{split}
    \end{equation}
    Now, we note that
    \begin{equation}
    \begin{split}
        &\mathbb{E} \left [ \left ( \int_0^t \left \| \Sigma(r,X(r),\mathbb{E}[X(r)], u(r)) - \Sigma(r,X^{\varepsilon}(r),\mathbb{E}[X^{\varepsilon}(r)], u^{\varepsilon}(r)) \right \|_{L(\mathbb{R}^w,H)}^{2} \mathrm{d}r \right )^j \right ]\\
        &\leq C \mathbb{E} \left [ \int_0^t \left \| \Sigma(r,X(r),\mathbb{E}[X(r)], u(r)) - \Sigma(r,X^{\varepsilon}(r),\mathbb{E}[X^{\varepsilon}(r)], u(r)) \right \|_{L(\mathbb{R}^w,H)}^{2j} \mathrm{d}r \right ]\\
        &\quad + C\mathbb{E} \left [ \left ( \int_0^t \left \| \Sigma(r,X^{\varepsilon}(r),\mathbb{E}[X^{\varepsilon}(r)], u(r)) - \Sigma(r,X^{\varepsilon}(r),\mathbb{E}[X^{\varepsilon}(r)], u^{\varepsilon}(r)) \right \|_{L(\mathbb{R}^w,H)}^2 \mathrm{d}r \right )^j \right ].
    \end{split}
    \end{equation}
    Moreover, due to the Lipschitz continuity of $\Sigma$ in the second and third variable, we have
    \begin{equation}
        \mathbb{E} \left [ \int_0^t \left \| \Sigma(r,X(r),\mathbb{E}[X(r)], u(r)) - \Sigma(r,X^{\varepsilon}(r),\mathbb{E}[X^{\varepsilon}(r)], u(r)) \right \|_{L(\mathbb{R}^w,H)}^{2j} \mathrm{d}r \right ] \leq C \mathbb{E} \left [ \int_0^t \| X(r) - X^{\varepsilon}(r) \|^{2j} \mathrm{d}r \right ].
    \end{equation}
    Using similar estimates for the remaining terms on the right-hand side of inequality \eqref{estimate_1_X_X_e}, we obtain
    \begin{equation}
    \begin{split}
        &\mathbb{E} \left [ \sup_{s\in [0,t]} \| X(s) - X^{\varepsilon}(s) \|^{2j} \right ]\\
        &\leq C \mathbb{E} \left [ \int_0^t \sup_{r\in [0,s]} \| X(r) - X^{\varepsilon}(r) \|^{2j} \mathrm{d}s \right ] \\
        &\quad + C \mathbb{E} \left [ \left ( \int_0^t \left \| B(r,X^{\varepsilon}(r),\mathbb{E}[X^{\varepsilon}(r)], u(r)) - B(r,X^{\varepsilon}(r),\mathbb{E}[X^{\varepsilon}(r)], u^{\varepsilon}(r)) \right \|^2 \mathrm{d}r \right )^j \right ]\\
        &\quad + C \mathbb{E} \left [ \left ( \int_0^t \left \| \Sigma(r,X^{\varepsilon}(r),\mathbb{E}[X^{\varepsilon}(r)], u(r)) - \Sigma(r,X^{\varepsilon}(r),\mathbb{E}[X^{\varepsilon}(r)], u^{\varepsilon}(r)) \right \|_{L(\mathbb{R}^w,H)}^2 \mathrm{d}r \right )^j \right ].
    \end{split}
    \end{equation}
    Since $u(\cdot)$ and $u^{\varepsilon}(\cdot)$ only differ on a set of measure $\varepsilon$, \eqref{asymptotic_X_X_e} follows from Gr\"onwall's inequality.

    Now, let us turn to inequality \eqref{asymptotic_Y_e}. For $s\in [0,t]$, we have
    \begin{equation}\label{estimate_1_Y_e}
    \begin{split}
        \| Y^{\varepsilon}(s) \|^2 &\leq 2 \left \| \int_0^s e^{(s-r)A} \left ( B_X(r) Y^{\varepsilon}(r) + B_{Y}(r) \mathbb{E}[ Y^{\varepsilon}(r) ] + \delta^v B(r) \right ) \mathrm{d}r \right \|^2 \\
        &\quad + 2 \left \| \int_0^s e^{(s-r)A} \left ( \Sigma_X(r) Y^{\varepsilon}(r) + \Sigma_{Y}(r) \mathbb{E}[ Y^{\varepsilon}(r) ] + \delta^v \Sigma(r) \right ) \mathrm{d}W(r) \right \|^2.
    \end{split}
    \end{equation}
    Applying \cite[Lemma 3.3]{gawarecki_mandrekar_2011} and Lemma \ref{lemma_B_Sigma_Lipschitz}, we obtain for the stochastic integral
    \begin{equation}
    \begin{split}
        &\mathbb{E} \left [ \sup_{s\in [0,t]} \left \| \int_0^s e^{(s-r)A} \left ( \Sigma_X(r) Y^{\varepsilon}(r) + \Sigma_{Y}(r) \mathbb{E}[ Y^{\varepsilon}(r) ] + \delta^v \Sigma(r) \right ) \mathrm{d}W(r) \right \|^{2j} \right ]\\
        &\leq C \mathbb{E} \left [ \sup_{s\in [0,t]} \left ( \int_0^s \left \| \Sigma_X(r) Y^{\varepsilon}(r) + \Sigma_{Y}(r) \mathbb{E}[ Y^{\varepsilon}(r) ] + \delta^v \Sigma(r) \right \|_{L(\mathbb{R}^w,H)}^2 \mathrm{d}r \right )^j \right ]\\
        &\leq C \mathbb{E} \left [ \int_0^t \sup_{r\in [0,s]} \| Y^{\varepsilon}(r) \|^{2j} \mathrm{d}s \right ] + \mathbb{E} \left [ \left ( \int_0^t \| \delta^v \Sigma(s) \|_{L(\mathbb{R}^w,H)}^2 \mathrm{d}s \right )^j \right ].
    \end{split}
    \end{equation}
    Using similar estimates for the remaining terms on the right-hand side of \eqref{estimate_1_Y_e}, we obtain
    \begin{equation}
    \begin{split}
        \mathbb{E} \left [ \sup_{s\in [0,t]} \| Y^{\varepsilon}(s) \|^{2j} \right ] &\leq C \mathbb{E} \left [ \int_0^t \sup_{r\in [0,s]} \| Y^{\varepsilon}(r) \|^{2j} \mathrm{d}s \right ] + C \mathbb{E} \left [ \left ( \int_0^t \left ( \| \delta^v B(r) \|^2 + \| \delta^v \Sigma(r) \|_{L(\mathbb{R}^w,H)}^2 \right ) \mathrm{d}r \right )^j \right ],
    \end{split}
    \end{equation}
    and hence by Gr\"onwall's inequality
    \begin{equation}
    \begin{split}
        \mathbb{E} \left [ \sup_{s\in [0,T]} \| Y^{\varepsilon}(s) \|^{2j} \right ] &\leq C \mathbb{E} \left [ \left ( \int_0^T \left ( \| \delta^v B(r) \|^2 + \| \delta^v \Sigma(r) \|_{L(\mathbb{R}^w,H)}^2 \right ) \mathrm{d}r \right )^j \right ],
    \end{split}
    \end{equation}
    which shows \eqref{asymptotic_Y_e}.

    Finally, let us turn to \eqref{asymptotic_Z_e}. We have
    \begin{equation}\label{estimate_1_Z_e}
    \begin{split}
        \| Z^{\varepsilon}(s) \|^2&\leq 2 \left \| \int_0^s e^{(s-r)A} \left ( B_X(r) Z^{\varepsilon}(r) + B_{Y}(r) \mathbb{E} [ Z^{\varepsilon}(r) ] + \delta^v B_X(r) Y^{\varepsilon}(r) + \frac12 B_{XX}(r) Y^{\varepsilon}(r)^2 \right ) \mathrm{d}r \right \|^2 \\
        &\quad + 2 \left \| \int_0^s e^{(s-r)A} \left ( \Sigma_X(r) Z^{\varepsilon}(r) + \Sigma_{Y}(r) \mathbb{E} [ Z^{\varepsilon}(r) ] + \delta^v \Sigma_X(r) Y^{\varepsilon}(r) + \frac12 \Sigma_{XX}(r) Y^{\varepsilon}(r)^2 \right ) \mathrm{d}W(r) \right \|^2.
    \end{split}
    \end{equation}
    Applying \cite[Lemma 3.3]{gawarecki_mandrekar_2011} and Lemma \ref{lemma_B_Sigma_Lipschitz}, we obtain for the stochastic integral
    \begin{equation}
    \begin{split}
        &\mathbb{E} \left [ \sup_{s\in [0,t]} \left \| \int_0^s e^{(s-r)A} \left ( \Sigma_X(r) Z^{\varepsilon}(r) + \Sigma_{Y}(r) \mathbb{E} [ Z^{\varepsilon}(r) ] + \delta^v \Sigma_X(r) Y^{\varepsilon}(r) + \frac12 \Sigma_{XX}(r) Y^{\varepsilon}(r)^2 \right ) \mathrm{d}W(r) \right \|^{2j} \right ]\\
        &\leq C \mathbb{E} \left [ \left ( \int_0^t \left \| \Sigma_X(r) Z^{\varepsilon}(r) + \Sigma_{Y}(r) \mathbb{E} [ Z^{\varepsilon}(r) ] + \delta^v \Sigma_X(r) Y^{\varepsilon}(r) + \frac12 \Sigma_{XX}(r) Y^{\varepsilon}(r)^2 \right \|_{L(\mathbb{R}^w,H)}^2 \mathrm{d}r \right )^j \right ]\\
        &\leq C \mathbb{E} \left [ \int_0^s \sup_{r\in [0,s]} \| Z^{\varepsilon}(r) \|^{2j} \mathrm{d}s \right ] + C \mathbb{E} \left [ \left ( \int_0^t \| \delta^v \Sigma_X(s) \|_{L(H,L(\mathbb{R}^w,H))}^2 \mathrm{d}s \right )^{2j} + \sup_{s\in [0,t]} \| Y^{\varepsilon}(s) \|^{4j} \right ].
    \end{split}
    \end{equation}
    Using similar estimates for the remaining terms on the right-hand side of \eqref{estimate_1_Z_e}, we obtain
    \begin{equation}
        \mathbb{E} \left [ \sup_{s\in [0,t]} \| Z^{\varepsilon}(s) \|^{2j} \right ] \leq C \mathbb{E} \left [ \int_0^t \sup_{r\in [0,s]} \| Z^{\varepsilon}(r) \|^{2j} \mathrm{d}s + \left ( \int_0^t \| \delta^v \Sigma_X(s) \|_{L(H,L(\mathbb{R}^w,H))}^2 \mathrm{d}s \right )^{2j} + \sup_{s\in [0,t]} \| Y^{\varepsilon}(s) \|^{4j} \right ],
    \end{equation}
    and hence by Gr\"onwall's inequality
    \begin{equation}
    \begin{split}
        \mathbb{E} \left [ \sup_{s\in [0,T]} \| Z^{\varepsilon}(s) \|^{2j} \right ] &\leq C \mathbb{E} \left [ \left ( \int_0^T \| \delta^v \Sigma_X(s) \|_{L(H,L(\mathbb{R}^w,H))}^2 \mathrm{d}s \right )^{2j} + \sup_{s\in [0,T]} \| Y^{\varepsilon}(s) \|^{4j} \right ],
    \end{split}
    \end{equation}
    which, in view of \eqref{asymptotic_Y_e}, proves \eqref{asymptotic_Z_e}.
\end{proof}

Now, let us turn to the improved asymptotics for $\mathbb{E}[ Y^{\varepsilon}(t) ]$ as alluded to in Remark \ref{remark_second_lions_derivative}.

\begin{proposition}\label{proposition_asymptotic_expectation_new}
    Let Assumption \ref{assumption_b_sigma} be satisfied and let $\tau\in (0,T)$ be the time at which the spike variation is applied. Then, we have for almost every $\tau\in [0,T]$,
    \begin{equation}\label{estimate_Y_varepsilon}
        \sup_{t\in[0,T]} \| \mathbb{E} [ Y^{\varepsilon}(t) ] \| = \mathcal{O}(\varepsilon).
    \end{equation}
    The exceptional set for $\tau \in (0,T)$ is independent of the set $\mathfrak{F} \in \mathcal{F}_{\tau}$ that is used in the construction of the spike variation in \eqref{spike_variation_definition}.
\end{proposition}

\begin{proof}
    The proof is based on a duality argument between $Y^{\varepsilon}$ and the solutions of a family of $L_2(H)$-valued BSDEs with varying terminal time $s\in [0,T]$ and appropriately chosen terminal condition $\phi^*\in L^2_{\mathcal{F}}(\Omega;C( [0,T];L_2(H)))$. More precisely, we consider the family of equations
    \begin{equation}\label{bsde_dual}
    \begin{cases}
        - \mathrm{d} p^s(r) = \left [ A^* p^s(r) + B_X^*(r) p^s(r) + \mathbb{E} [ B_{Y}^*(r) p^s(r) ] + \Sigma^*_X(r) q^s(r) + \mathbb{E} [ \Sigma^*_{Y}(r) q^s(r) ] \right ] \mathrm{d}r - q^s(r) \mathrm{d}W(r)\\
        p^s(s) = \phi^*(s),
    \end{cases}
    \end{equation}    
    for $0\leq r\leq s \leq T$. Since $B_X$ and $\Sigma_X$ are bounded, we may apply Theorem \ref{BSDE_existence_of_solution} to obtain, for every $s\in (0,T]$, a unique mild solution $(p^s,q^s)$ to this equation.

    We first prove continuity of the family of solutions $(p^s,q^s)$ of \eqref{bsde_dual} in $s$. To this end, we extend  $(p^s,q^s)$ to the whole interval $[0,T]$ by setting
    \begin{equation}\label{extension}
        p^s(r)=\phi^*(s), \quad \text{for all }  r>s,  \ \mathbb{P}-a.s.\qquad  \text{   and }\qquad q^s(r)=0,\qquad \text{ for a.e.  } r>s,  \ \mathbb{P}-a.s.
    \end{equation}
    We show that
    \[
        s\mapsto (p^s,q^s), \quad [0,T] \to \mathcal X\times\mathcal Y := L^2_{\mathcal F} \left (\Omega;C([0,T];L_2(H)) \right ) \times L^2_{\mathcal F} \left ( \Omega\times[0,T];L(\mathbb R^w,L_2(H))\right )
    \]
    is continuous. To this end, fix $s_0\in[0,T]$. We prove the claim for $s>s_0$; the case $s<s_0$ is analogous. On $[0,s_0]$, the two processes $(p^s,q^s)$ and $(p^{s_0},q^{s_0})$ solve the same linear BSDE, with the same coefficients and the same terminal time $s_0$, but with terminal data $p^s(s_0)$ and $\phi^*(s_0)$, respectively. Hence estimate \eqref{a_priori_estimate_bsde} in Theorem \ref{BSDE_existence_of_solution} gives
    \begin{equation}\label{3.11a}
        \E\left [\sup_{r\in[0,s_0]}\|p^s(r)-p^{s_0}(r)\|_{L_2(H)}^2 + \int_0^{s_0}\|q^s(r)-q^{s_0}(r)\|_{L(\mathbb R^w,L_2(H))}^2\mathrm{d}r \right ] \leq C \E\left [ \|p^s(s_0)-\phi^*(s_0)\|_{L_2(H)}^2 \right ].
    \end{equation}
    It remains to prove that the right-hand side tends to zero as $s\downarrow s_0$. To this end, in order to estimate the difference  $p^s(r)-\phi^*(s_0),\, r\in[s_0,s]$, we compare $(p^s,q^s)$ on $[s_0,s]$ with the pair $(e^{(s-\cdot)A^*} \phi^*(s_0),0)$. Applying the estimate \eqref{a_priori_estimate_bsde} on the interval $[s_0,s]$, we obtain
    \begin{equation}\label{3.11b}
    \begin{split}
        &\E\left [\sup_{r\in[s_0,s]} \| p^s(r)-\phi^*(s_0) \|_{L_2(H)}^2 + \int_{s_0}^{s} \| q^s(r) \|_{L(\mathbb R^w,L_2(H))}^2 \mathrm{d}r\right ]\\
        &\leq C \E \left [ \|\phi^*(s)-\phi^*(s_0)\|_{L_2(H)}^2 \right ] + C \E \left [\sup_{r\in[s_0,s]} \left \|\left ( e^{(s-r)A^*} - \text{Id} \right )\phi^*(s_0) \right \|_{L_2(H)}^2 \right ] \\
        &\quad + C \E \left [ \int_{s_0}^{s} \left ( \left \| B_X^*(r) e^{(s-r)A^*} \phi^*(s_0) + \E \left [ B_Y^*(r) e^{(s-r)A^*} \phi^*(s_0) \right ] \right \|_{L_2(H)}^2 \right ) \mathrm{d}r \right ].
    \end{split}
    \end{equation}
    The first term on the right-hand side tends to zero as $s\downarrow s_0$ due to the continuity assumption on $\phi^*$. The second term tends to zero due to the strong continuity of the semigroup $(e^{tA^*})_{t\ge0}$ and dominated convergence. The last term is bounded by $C(s-s_0)\E [ \|\phi^*(s_0)\|_{L_2(H)}^2 ]$, since the coefficients are bounded. Hence,
    \begin{equation}\label{3.11c}
        \E \left [ \sup_{r\in[s_0,s]} \| p^s(r)- \phi^*(s_0)\|_{L_2(H)}^2 + \int_{s_0}^{s}\|q^s(r)\|_{L(\mathbb R^w,L_2(H))}^2 \mathrm{d}r\right ] \to 0, \qquad s\downarrow s_0. 
    \end{equation}
    In particular, $\E[ \| p^s(s_0) -\phi^*(s_0) \|_{L_2(H)}^2 ] \to 0$. Combining this with \eqref{3.11a}, and taking into account the extensions $p^s(r)=\phi^*(s)$, $q^s(r)=0$ for $r>s$, we obtain
    \begin{equation}
    \begin{split}
        \|(p^s,q^s)-(p^{s_0},q^{s_0})\|_{\mathcal X\times\mathcal Y}^2 &\leq C \E \left [ \sup_{r\in[0,s_0]}\|p^s(r)-p^{s_0}(r)\|_{L_2(H)}^2 +\int_0^{s_0}\|q^s(r)-q^{s_0}(r)\|_{L(\mathbb R^w,L_2(H))}^2\mathrm{d}r \right ] \\
        &\quad + C \E \left [ \sup_{r\in[s_0,s]} \| p^s(r) -\phi^*(s_0) \|_{L_2(H)}^2 + \int_{s_0}^{s}\|q^s(r)\|_{L(\mathbb R^w,L_2(H))}^2\mathrm{d}r \right ] \\
        & \quad + C \E \left [ \|\phi^*(s)-\phi^*(s_0)\|_{L_2(H)}^2 \right ] \to 0.
    \end{split}
    \end{equation}
    Thus, $s\mapsto(p^s,q^s)$ is continuous as a map from $[0,T]$ to $\mathcal X\times\mathcal Y$. In particular, we can find a jointly measurable version $(s,r,\omega)\mapsto(p^s(r,\omega),q^s(r,\omega))$. Moreover, again estimate \eqref{a_priori_estimate_bsde} together with the extension \eqref{extension}, yields
    \begin{equation}
        \sup_{s\in[0,T]} \E \left [ \sup_{r\in[0,T]} \| p^s(r) \|_{L_2(H)}^2 + \int_0^T \| q^s(r) \|_{L(\mathbb R^w,L_2(H))}^2 \mathrm{d}r \right ] \leq C \E \left [ \sup_{s\in[0,T]}\|\phi^*(s)\|_{L_2(H)}^2 \right ].
    \end{equation}

    Next, we establish the duality relation between $Y^{\varepsilon}$ and $(p^s,q^s)$. For $\lambda>0$, let $A_{\lambda}$ be the Yosida approximation of $A$. Let $Y^{\varepsilon,\lambda}$ be the solution of equation \eqref{first_variational_equation} with $A$ replaced by $A_{\lambda}$, and let $(p^{s,\lambda},q^{s,\lambda})$ be the solution of equation \eqref{bsde_dual} with $A$ replaced by $A_{\lambda}$. Then, by It\^o's formula, we have for $h\in H$
    \begin{equation}\label{Ito_p_t_Y_epsilon}
    \begin{split}
        \mathbb{E} \left [ \langle \phi(s) Y^{\varepsilon,\lambda}(s), h \rangle \right ]
        &= \mathbb{E} \Bigg [ - \int_0^s \bigg \langle Y^{\varepsilon,\lambda}(r), A_{\lambda}^* p^{s,\lambda}(r) h + B_X^*(r) p^{s,\lambda}(r) h\\
        &\qquad\qquad\qquad + \mathbb{E} [ B_Y^*(r) p^{s,\lambda}(r) h ] + \Sigma^*_X(r) q^{s,\lambda}(r) h + \mathbb{E} [ \Sigma^*_Y(r) q^{s,\lambda}(r)h ] \bigg \rangle \mathrm{d}r \\
        &\qquad\quad + \int_0^s \langle A_{\lambda}Y^{\varepsilon,\lambda}(r) + B_X(r) Y^{\varepsilon,\lambda}(r) + B_Y(r) \mathbb{E} [ Y^{\varepsilon,\lambda}(r) ] + \delta^v B(r), p^{s,\lambda}(r) h \rangle \mathrm{d}r\\
        &\qquad\quad + \int_0^s \text{Tr} ( ( \Sigma_X(r) Y^{\varepsilon,\lambda}(r) + \Sigma_Y(r) \mathbb{E}[Y^{\varepsilon,\lambda}(r)] + \delta^v \Sigma(r) )^* q^{s,\lambda}(r) h ) \mathrm{d}r \Bigg ].
    \end{split}
    \end{equation}
    Noting that
    \begin{equation}
        \text{Tr}( (\Sigma_X(r) Y^{\varepsilon,\lambda}(r) )^* q^{s,\lambda}(r) h ) = \langle \Sigma_X(r) Y^{\varepsilon,\lambda}(r), q^{s,\lambda}(r) h \rangle_{L_2(\mathbb{R}^w,H)} = \langle Y^{\varepsilon,\lambda}(r), \Sigma^*_X(r) q^{s,\lambda}(r) h \rangle,
    \end{equation}
    and with a similar calculation for the term involving $\Sigma_Y(r) \mathbb{E} [ Y^{\varepsilon,\lambda}(r)]$ we obtain from \eqref{Ito_p_t_Y_epsilon}
    \begin{equation}
       \mathbb{E} \left [ \langle \phi(s) Y^{\varepsilon,\lambda}(s), h \rangle \right ] = \mathbb{E} \left [ \int_0^s \left ( \langle \delta^v B(r), p^{s,\lambda}(r) h \rangle + \text{Tr}(\delta^v \Sigma^*(r) q^{s,\lambda}(r) h ) \right ) \mathrm{d}r \right ].
    \end{equation}
    Taking the limit $\lambda \to \infty$, we obtain from Theorem \ref{theorem_general_bsde_yosida}
    \begin{equation}
       \mathbb{E} \left [ \langle \phi(s) Y^{\varepsilon}(s), h \rangle \right ] = \mathbb E \left [ \int_0^s \left ( \langle \delta^v B(r), p^{s}(r) h \rangle + \text{Tr}(\delta^v \Sigma^*(r) q^{s}(r) h ) \right ) \mathrm{d}r \right ].
    \end{equation}
    Now, we identify the map $(h \mapsto \text{Tr}(\delta^v \Sigma^*(r) q^s(r) h ) ) \in H^*$ with an element in $H$, denoted by $\text{Tr}(\delta^v \Sigma^*(r) q^s(r) ) $. Then, we obtain the duality
    \begin{equation}\label{equality_1_y_epsilon_p_B_sigma_q}
    \begin{split}
        \mathbb{E} \left [ \phi(s) Y^{\varepsilon}(s) \right ] 
        &= \mathbb{E} \left [ \int_0^s \left ( p^{s}(r)^* \delta^v B(r) + \text{Tr} ( \delta^v \Sigma^*(r) q^{s}(r) ) \right ) \mathrm{d}r \right ].
    \end{split}
    \end{equation}
    We have
    \begin{equation}
        \| \text{Tr} ( \delta^v \Sigma^*(r) q^{s}(r)) \| \leq C \| \delta^v \Sigma(r) \|_{L(\mathbb{R}^w,H)} \| q^{s}(r) \|_{L(\mathbb{R}^w,L_2(H))}.
    \end{equation}
    Thus, moving the norm inside the expectation and the integral, and applying the Fubini--Tonelli Theorem, we obtain
    \begin{equation}\label{estimate_1_p_B_sigma_q}
    \begin{split}
        &\left \| \mathbb{E} \left [ \int_0^s \left ( p^{s}(r)^* \delta^v B(r) + \text{Tr} ( \delta^v \Sigma^*(r) q^{s}(r) ) \right ) \mathrm{d}r \right ] \right \|\\
        &\leq \int_0^s \left ( \mathbb{E} \left [ \| p^{s}(r) \|_{L_2(H)}^2 \right ]^{\frac12} \mathbb{E} \left [ \| \delta^v B(r) \|^2 \right ]^{\frac12} + \mathbb{E} \left [ \| \delta^v \Sigma(r) \|_{L(\mathbb{R}^w,H)}^2 \right ]^{\frac12} \mathbb{E} \left [ \| q^{s}(r) \|_{L(\mathbb{R}^w,L_2(H))}^2 \right ]^{\frac12} \right ) \mathrm{d}r.
    \end{split}
    \end{equation}
    Let us now assume without loss of generality that $\tau <s$ (otherwise $Y^{\varepsilon}(r) = 0$ for all $0\leq r\leq s$). Integrating over $s\in [0,T]$, applying Fubini's theorem, dividing by $\varepsilon$ and taking the limit $\varepsilon\to 0$, we obtain for the first term on the right-hand side for almost every $\tau\in [0,T]$,
    \begin{equation}
    \begin{split}
        &\limsup_{\varepsilon\to 0} \frac{1}{\varepsilon} \int_{\tau}^{\tau+\varepsilon} \int_r^T \mathbb{E} \left [ \| p^s(r) \|^2_{L_2(H)} \right ]^{\frac12} \mathrm{d}s\; \mathbb{E} \left [ \| \delta^v B(r) \|^2 \right ]^{\frac12} \mathrm{d}r\\
        &\leq \mathbb{E} \left [ \| B(\tau,X(\tau),\mathbb{E}[X(\tau)],v) - B(\tau,X(\tau),\mathbb{E}[X(\tau)],u(\tau)) \|^2 \right ]^{\frac12} \int_{\tau}^T \mathbb{E} \left [ \| p^s(\tau) \|_{L_2(H)}^2 \right ]^{\frac12} \mathrm{d}s.
    \end{split}
    \end{equation}
    For the second term on the right-hand side of equation \eqref{estimate_1_p_B_sigma_q}, for $\varepsilon>0$ sufficiently small, we have
    \begin{equation}
    \begin{split}
        &\int_0^s \mathbb{E} \left [ \| \delta^v \Sigma(r) \|_{L(\mathbb{R}^w,H)}^2 \right ]^{\frac12} \mathbb{E} \left [ \| q^{s}(r) \|_{L(\mathbb{R}^w,L_2(H))}^2 \right ]^{\frac12} \mathrm{d}r\\
        &\leq \int_{\tau}^{\tau+\varepsilon} \mathbb{E} \left [ \| \Sigma(r,X(r),\mathbb{E}[X(r)],v) - \Sigma(r,X(r),\mathbb{E}[X(r)],u(r)) \|_{L(\mathbb{R}^w,H)}^2 \right ]^{\frac12} \mathbb{E} \left [ \| q^{s}(r) \|_{L(\mathbb{R}^w,L_2(H))}^2 \right ]^{\frac12} \mathrm{d}r.
    \end{split}
    \end{equation}
    Notice that by our assumption the integrands are integrable, so integrating over $s\in [0,T]$, applying Fubini's theorem, dividing by $\varepsilon$ and taking the limit $\varepsilon\to 0$, we obtain by Lebesgue's differentiation theorem for almost every $\tau \in [0,T]$,
    \begin{equation}
    \begin{split}
        &\lim_{\varepsilon\to 0} \frac{1}{\varepsilon} \int_{\tau}^{\tau+\varepsilon} \mathbb{E} \left [ \| \Sigma(r,X(r),\mathbb{E}[X(r)],v) - \Sigma(r,X(r),\mathbb{E}[X(r)],u(r)) \|_{L(\mathbb{R}^w,H)}^2 \right ]^{\frac12} \int_r^T \mathbb{E} \left [ \| q^{s}(r) \|_{L(\mathbb{R}^w,L_2(H))}^2 \right ]^{\frac12} \mathrm{d}s \mathrm{d}r\\
        &= \mathbb{E} \left [ \| \Sigma(\tau,X(\tau),\mathbb{E}[X(\tau)],v) - \Sigma(\tau,X(\tau),\mathbb{E}[X(\tau)],u(\tau)) \|_{L(\mathbb{R}^w,H)}^2 \right ]^{\frac12} \int_{\tau}^T \mathbb{E} \left [ \| q^{s}(\tau) \|_{L(\mathbb{R}^w,L_2(H))}^2 \right ]^{\frac12} \mathrm{d}s.
    \end{split}
    \end{equation}
    Thus, we derive from \eqref{equality_1_y_epsilon_p_B_sigma_q} and \eqref{estimate_1_p_B_sigma_q}
    \begin{equation}\label{intermediate_asymptotic}
    \begin{split}
        &\limsup_{\varepsilon\to 0} \frac{1}{\varepsilon} \int_0^T \| \mathbb{E} [ \phi(s) Y^{\varepsilon}(s) ] \| \mathrm{d}s\\
        &\leq \mathbb{E} \left [ \| B(\tau,X(\tau),\mathbb{E}[X(\tau)],v) - B(\tau,X(\tau),\mathbb{E}[X(\tau)],u(\tau)) \|^2 \right ]^{\frac{1}{2}} \int_{\tau}^T \mathbb{E} \left [ \| p^{s}(\tau) \|_{L_2(H)}^2 \right ]^{\frac{1}{2}} \mathrm{d}s\\
        &\quad + \mathbb{E} \left [ \| \Sigma(\tau,X(\tau),\mathbb{E}[X(\tau)],v) - \Sigma(\tau,X(\tau),\mathbb{E}[X(\tau)],u(\tau)) \|_{L(\mathbb{R}^w,H)}^2 \right ]^{\frac{1}{2}} \int_{\tau}^T \mathbb{E} \left [ \| q^{s}(\tau) \|_{L(\mathbb{R}^w,L_2(H))}^2 \right ]^{\frac12} \mathrm{d}s.
    \end{split}
    \end{equation}
    Note that, by Theorem \ref{BSDE_existence_of_solution}, we have
    \begin{equation}
    \begin{split}
        &\int_0^T \int_r^T \mathbb{E} \left [ \| p^s(r) \|_{L_2(H)}^2 + \| q^s(r) \|_{L(\mathbb{R}^w,L_2(H))}^2 \right ] \mathrm{d}s \mathrm{d}r \\
        &= \int_0^T \int_0^s \mathbb{E} \left [ \| p^s(r) \|_{L_2(H)}^2 + \| q^s(r) \|_{L(\mathbb{R}^w,L_2(H))}^2 \right ] \mathrm{d}r \mathrm{d}s \leq C \mathbb{E} \left [ \int_0^T \| \phi(s) \|_{L_2(H)}^2 \mathrm{d}s \right ].
    \end{split}
    \end{equation}
    Thus, for $\phi \in L^2_{\mathcal{F}}(\Omega;C( [0,T];L_2(H)))$, $\int_{\tau}^T \mathbb{E} [ \| p^s(\tau) \|_{L_2(H)}^2 + \| q^s(\tau) \|_{L(\mathbb{R}^w,L_2(H))}^2 ]^{1/2} \mathrm{d}s$ is finite for almost every $\tau\in [0,T]$ and the exceptional null set is independent of the set $\mathfrak{F} \in \mathcal{F}_{\tau}$. Therefore by \eqref{intermediate_asymptotic}
    \begin{equation}\label{intermediate_claim}
        \int_0^T \| \mathbb{E} [ \phi(s) Y^{\varepsilon}(s) ] \| \mathrm{d}s = \mathcal{O}(\varepsilon)
    \end{equation}
    for all $\phi \in L^2_{\mathcal{F}}(\Omega;C( [0,T];L_2(H)))$, for almost every $\tau\in [0,T]$.
    
    We now extend the asymptotics \eqref{intermediate_claim} to processes $\phi \in L^2_{\mathcal F}(\Omega\times[0,T];L_2(H))$. Let \(\rho\in C_c^\infty((0,1))\) be non-negative and such that
    \begin{equation}
        \int_0^1 \rho(r)\mathrm{d}r = 1.
    \end{equation}
    For $k\in \mathbb{N}$, set
    \begin{equation}
        \rho_k(r) := k \rho \left( k r \right).
    \end{equation}
    We extend $\phi$ by zero on $(-\infty,0)$, and define the regularization
    \begin{equation}\label{3.16}
        \phi_{k}(s) := \int_0^{\frac{1}{k}} \rho_k(r) \phi(s-r) \mathrm{d}r, \qquad s\in[0,T].
    \end{equation}
    Since $\phi_k$ is a convolution, it is adapted because for $r\in [0,1/k]$, \(\phi(s-r)\) is \(\mathcal F_{s-r}\)-measurable, and \(\mathcal F_{s-r}\subset \mathcal F_s\). Hence, since \(\phi_k\) has continuous sample paths, it is progressively measurable. Moreover,
    \[
        \phi_k \in L^2_{\mathcal F}\left ( \Omega ; C([0,T];L_2(H)) \right ),
    \]
    and the standard approximation property of mollifiers gives
    \begin{equation}\label{3.17}
        \phi_k \to \phi \quad\text{in}\quad  L^2_{\mathcal F}(\Omega\times[0,T];L_2(H)) \quad \text{as}\quad k \to \infty.
    \end{equation}
    By \eqref{intermediate_claim}, we obtain, for every fixed \( k\in\mathbb{N} \),
    \begin{equation}\label{3.18}
        \int_0^T \left\| \mathbb E \left [ \phi_k (s) Y^\varepsilon(s) \right ] \right\| \mathrm{d}s = \mathcal{O}_k(\varepsilon).
    \end{equation}
    In fact, from the estimate preceding \eqref{intermediate_claim}, the constant can be chosen proportional to
    \begin{equation}
        \mathbb E \left [ \int_0^T \| \phi_k(s) \|_{L_2(H)}^2 \mathrm{d}s \right ].
    \end{equation}
    Since \(\phi_k \to \phi\) in \(L^2_{\mathcal F} (\Omega \times [0,T];L_2(H))\), and since the family \((\phi_k)_{k\in\mathbb{N}}\) is bounded in this space, we may pass to the limit \(k\to\infty\) and obtain for all $\phi \in L^2_{\mathcal F}(\Omega\times[0,T];L_2(H))$
    \begin{equation}\label{3.19}
        \int_0^T \left\| \mathbb E \left [ \phi(s) Y^\varepsilon(s) \right ] \right\| \mathrm{d}s = \mathcal{O}(\varepsilon).
    \end{equation}
    Indeed, for each fixed \(\varepsilon>0\),
    \begin{equation}
        \int_0^T \left \| \mathbb E \left [ ( \phi_k(s) - \phi(s) ) Y^\varepsilon(s) \right ] \right\| \mathrm{d}s \leq \mathbb E \left [ \int_0^T \| \phi_k(s) - \phi(s) \|_{L_2(H)}^2 \mathrm{d}s \right ]^{\frac12} \mathbb E \left [ \int_0^T \|Y^\varepsilon(s)\|^2 \mathrm{d}s \right ]^{\frac12},
    \end{equation}
    which tends to \(0\) as \(k\to\infty\). Hence \eqref{3.19} follows from \eqref{3.18} and the \(L^2\)-boundedness of \(\phi_k\).

    Now, we are in a position to complete the proof. For any $t\in [0,T]$, we have
    \begin{equation}\label{gronwall}
    \begin{split}
        \sup_{s\in [0,t]} \| \mathbb{E} [ Y^{\varepsilon}(s) ] \| &= \sup_{s\in [0,t]} \left \| \mathbb{E} \left [ \int_0^s e^{(s-r)A} ( B_X(r) Y^{\varepsilon}(r) + B_Y(r) \mathbb{E} [ Y^{\varepsilon}(r) ] + \delta^v B(r) ) \mathrm{d}r \right ] \right \|\\
        &\leq C \int_0^t \left ( \left \| \mathbb{E} [ B_X(s) Y^{\varepsilon}(s) ] \right \| + \sup_{r\in [0,s]} \left \| \mathbb{E} [ Y^{\varepsilon}(r) ] \right \| + \left \| \mathbb{E} [ \delta^v B(s) ] \right \| \right ) \mathrm{d}s.
    \end{split}
    \end{equation}
    Choosing $\phi = B_X$ in \eqref{3.19}, we obtain
    \begin{equation}
        \int_0^T \left\| \mathbb E \left [ B_X(s) Y^\varepsilon(s) \right ] \right\| \mathrm{d}s = \mathcal{O}(\varepsilon).
    \end{equation}
    Moreover,
    \begin{equation}
        \int_0^T \| \mathbb{E} [ \delta^v B(s) ] \| \mathrm{d}s = \mathcal{O}(\varepsilon).
    \end{equation}
    Therefore, applying Gr\"onwall's inequality in \eqref{gronwall} concludes the proof.
\end{proof}

\begin{remark}\label{remark-improved-estimate}
    We note that a crucial point in proving estimate \eqref{estimate_Y_varepsilon} is the fact that $B_X(s)$, $s\in [0,T]$, is a Hilbert--Schmidt operator. This Hilbert--Schmidt property is due to the special structure of the problem: $B$ takes values in the finite dimensional component of the space $H$.
\end{remark}

Now, we are in a position to prove the asymptotics for $X^{\varepsilon}(t) - X(t) - Y^{\varepsilon}(t)$ and $X^{\varepsilon}(t) - X(t) - Y^{\varepsilon}(t) - Z^{\varepsilon}(t)$. The result in the following proposition matches the classical result without law dependence, see e.g. \cite[Proposition 4.4]{fuhrman_hu_tessitore_2013}. However, the proof of \eqref{asymptotic_X_e_X_Y_e} typically requires a second order derivative of the coefficients $B$ and $\Sigma$ with respect to $Y$, which is not available to us. Moreover, in the equation for $Z^{\varepsilon}$ one would expect the second order derivative with respect to $Y$ which is not present in our equation for $Z^{\varepsilon}$. For these reasons, the proofs require the improved asymptotics for $\mathbb{E}[ Y^{\varepsilon}(t) ]$ proved in the previous proposition.

\begin{proposition}
    Let Assumption \ref{assumption_b_sigma} be satisfied. Then, we have for all $j\geq 1$ and almost every $\tau\in (0,T)$
    \begin{align}
        \label{asymptotic_X_e_X_Y_e} \mathbb{E} \left [ \sup_{t\in [0,T]} \| X^{\varepsilon}(t) - X(t) - Y^{\varepsilon}(t)\|^{2j} \right ] &= \mathcal{O}(\varepsilon^{2j})\\
        \label{asymptotic_X_e_X_Y_e_Z_e} \mathbb{E} \left [ \sup_{t\in [0,T]} \| X^{\varepsilon}(t) - X(t) - Y^{\varepsilon}(t) - Z^{\varepsilon}(t) \|^{2j} \right ] &= o(\varepsilon^{2j}).
    \end{align}
    The exceptional set for $\tau \in (0,T)$ is independent of the set $\mathfrak{F} \in \mathcal{F}_{\tau}$ that is used in the construction of the spike variation in \eqref{spike_variation_definition}.
\end{proposition}

\begin{proof}
    Let us first prove inequality \eqref{asymptotic_X_e_X_Y_e}. Note that
    \begin{equation}
    \begin{cases}
        \mathrm{d} ( X^{\varepsilon}(t) - X(t) - Y^{\varepsilon}(t)) = A ( X^{\varepsilon}(t) - X(t) - Y^{\varepsilon}(t)) \mathrm{d}t\\
        \qquad\qquad\qquad\qquad\qquad\qquad + [ B^{\varepsilon}(t) - B(t) - B_X(t) Y^{\varepsilon}(t) - B_{Y}(t) \mathbb{E} [ Y^{\varepsilon}(t) ] - \delta^v B(t) ] \mathrm{d}t\\
        \qquad\qquad\qquad\qquad\qquad\qquad + [ \Sigma^{\varepsilon}(t) - \Sigma(t) - \Sigma_X(t) Y^{\varepsilon}(t) - \Sigma_{Y}(t) \mathbb{E} [ Y^{\varepsilon}(t) ] - \delta^v \Sigma(t) ]\mathrm{d}W(t)\\
        X^{\varepsilon}(0) - X(0) - Y^{\varepsilon}(0) = 0.
    \end{cases}
    \end{equation}
    Therefore, for $s\in [0,t]$, we have
    \begin{equation}\label{estimate_1_X_e_X_Y_e}
    \begin{split}
        &\| X^{\varepsilon}(s) - X(s) - Y^{\varepsilon}(s) \|^2\\
        &\leq 2 \left \| \int_0^s e^{(s-r)A} \left ( B^{\varepsilon}(r) - B(r) - B_X(r) Y^{\varepsilon}(r) - B_{Y}(r) \mathbb{E} [ Y^{\varepsilon}(r) ] - \delta^v B(r) \right ) \mathrm{d}r \right \|^2\\
        &\quad + 2 \left \| \int_0^s e^{(s-r)A} \left ( \Sigma^{\varepsilon}(r) - \Sigma(r) - \Sigma_X(r) Y^{\varepsilon}(r) - \Sigma_{Y}(r) \mathbb{E} [ Y^{\varepsilon}(r) ] - \delta^v \Sigma(r) \right ) \mathrm{d}W(r) \right \|^2.
    \end{split}
    \end{equation}
    Applying \cite[Lemma 3.3]{gawarecki_mandrekar_2011}, we obtain for the stochastic integral
    \begin{equation}\label{bdg_X_e_X_Y_e}
    \begin{split}
        &\mathbb{E} \left [ \sup_{s\in [0,t]} \left \| \int_0^s e^{(s-r)A} \left ( \Sigma^{\varepsilon}(r) - \Sigma(r) - \Sigma_X(r) Y^{\varepsilon}(r) - \Sigma_{Y}(r) \mathbb{E} [ Y^{\varepsilon}(r) ] - \delta^v \Sigma(r) \right ) \mathrm{d}W(r) \right \|^{2j} \right ]\\
        &\leq C \mathbb{E} \left [ \left ( \int_0^t \left \| \Sigma^{\varepsilon}(r) - \Sigma(r) - \Sigma_X(r) Y^{\varepsilon}(r) - \Sigma_{Y}(r) \mathbb{E} [ Y^{\varepsilon}(r) ] - \delta^v \Sigma(r) \right \|_{L(\mathbb{R}^w,H)}^2 \mathrm{d}r \right )^{j} \right ].
    \end{split}
    \end{equation}
    We have
    \begin{equation}\label{decomposition_Sigma}
    \begin{split}
        &\Sigma^{\varepsilon}(r) - \Sigma(r) - \Sigma_X(r) Y^{\varepsilon}(r) - \Sigma_Y(r) \mathbb{E} [ Y^{\varepsilon}(r) ] - \delta^v \Sigma(r)\\
        &= \int_0^1 \bigg ( \Sigma_X(r,X(r) + \theta (X^{\varepsilon}(r) - X(r)), \mathbb{E} [ X(r) ] , u^{\varepsilon}(r)) ( X^{\varepsilon}(r) - X(r) - Y^{\varepsilon}(r))\\
        &\qquad\qquad + \Sigma_Y(r,X^{\varepsilon}(r), \mathbb{E} [ X(r) + \theta ( X^{\varepsilon}(r)- X(r)) ], u^{\varepsilon}(r)) \mathbb{E} [ X^{\varepsilon}(r) - X(r) - Y^{\varepsilon}(r) ] \\
        &\qquad\qquad + \left ( \Sigma_X(r,X(r) + \theta (X^{\varepsilon}(r) - X(r)), \mathbb{E} [ X(r) ] , u^{\varepsilon}(r)) - \Sigma_X(r) \right ) Y^{\varepsilon}(r) \\
        &\qquad\qquad + \left ( \Sigma_Y(r,X^{\varepsilon}(r), \mathbb{E} [ X(r) + \theta (X^{\varepsilon}(r) - X(r)) ] , u^{\varepsilon}(r) ) - \Sigma_Y(r) \right ) \mathbb{E} [ Y^{\varepsilon}(r) ] \bigg ) \mathrm{d}\theta.
    \end{split}
    \end{equation}
    Regarding the first and second term on the right-hand side, we have due to Lemma \ref{lemma_B_Sigma_Lipschitz}
    \begin{equation}\label{decomposition_sigma_term_1_and_2}
    \begin{split}
        &\mathbb{E} \left [ \left ( \int_0^t \left \| \int_0^1 \Sigma_X(r,X(r) + \theta (X^{\varepsilon}(r) - X(r)), \mathbb{E} [ X(r) ] , u^{\varepsilon}(r)) ( X^{\varepsilon}(r) - X(r) - Y^{\varepsilon}(r)) \mathrm{d}\theta \right \|_{L(\mathbb{R}^w,H)}^{2} \mathrm{d}r \right )^j \right ]\\
        &+ \mathbb{E} \left [ \left ( \int_0^t \left \| \int_0^1 \Sigma_Y(r,X^{\varepsilon}(r), \mathbb{E} [ X(r) + \theta ( X^{\varepsilon}(r)- X(r)) ], u^{\varepsilon}(r)) \mathbb{E} [ X^{\varepsilon}(r) - X(r) - Y^{\varepsilon}(r) ] \mathrm{d}\theta \right \|_{L(\mathbb{R}^w,H)}^{2} \mathrm{d}r \right )^j \right ]\\
        &\leq C \mathbb{E} \left [ \int_0^t \sup_{r\in [0,s]} \left \| X^{\varepsilon}(r) - X(r) - Y^{\varepsilon}(r) \right \|^{2j} \mathrm{d}s \right ].
    \end{split}
    \end{equation}
    For the third term on the right-hand side of \eqref{decomposition_Sigma}, we have
    \begin{equation}\label{decompisition_Sigma_intermediate_1}
    \begin{split}
        &\mathbb{E} \left [ \left ( \int_0^t \left \| \int_0^1 \left ( \Sigma_X(r,X(r) + \theta (X^{\varepsilon}(r) - X(r)), \mathbb{E} [ X(r) ] , u^{\varepsilon}(r)) - \Sigma_X(r) \right ) Y^{\varepsilon}(r) \mathrm{d}\theta \right \|_{L(\mathbb{R}^w,H)}^{2} \mathrm{d}r \right )^j \right ]\\
        &\leq C \mathbb{E} \Bigg [ \bigg ( \int_0^T \int_0^1 \| \Sigma_X(r,X(r) + \theta (X^{\varepsilon}(r) - X(r)), \mathbb{E} [ X(r) ] , u^{\varepsilon}(r))\\
        &\qquad\qquad\qquad\qquad\quad - \Sigma_X(r,X(r) + \theta (X^{\varepsilon}(r) - X(r)), \mathbb{E} [ X(r) ] , u(r)) \|_{L(H,L(\mathbb{R}^w,H))}^{2} \mathrm{d}\theta \mathrm{d}r \bigg )^{2j} \Bigg ]^{\frac12} \mathbb{E} \left [ \sup_{r\in [0,T]} \| Y^{\varepsilon}(r) \|^{4j} \right ]^{\frac12}\\
        &\quad + C \mathbb{E} \left [ \int_0^T \int_0^1 \| \Sigma_X(r,X(r) + \theta (X^{\varepsilon}(r) - X(r)), \mathbb{E} [ X(r) ] , u(r)) - \Sigma_X(r) \|_{L(H,L(\mathbb{R}^w,H))}^{4j} \mathrm{d}\theta \mathrm{d}r \right ]^{\frac12} \mathbb{E} \left [ \sup_{r\in [0,T]} \| Y^{\varepsilon}(r) \|^{4j} \right ]^{\frac12}.
    \end{split}
    \end{equation}
    Since $\Sigma_X$ is bounded, we have for the first term on the right-hand side of this equation
    \begin{equation}
    \begin{split}
        &\mathbb{E} \Bigg [ \bigg ( \int_0^T \int_0^1 \| \Sigma_X(r,X(r) + \theta (X^{\varepsilon}(r) - X(r)), \mathbb{E} [ X(r) ] , u^{\varepsilon}(r))\\
        &\qquad\qquad\qquad\qquad - \Sigma_X(r,X(r) + \theta (X^{\varepsilon}(r) - X(r)), \mathbb{E} [ X(r) ] , u(r)) \|_{L(H,L(\mathbb{R}^w,H))}^{2} \mathrm{d}\theta \mathrm{d}r \bigg )^{2j} \Bigg ]^{\frac12} = \mathcal{O}(\varepsilon^j).
    \end{split}
    \end{equation}
    Moreover, using \eqref{asymptotic_X_X_e}, we obtain for the second term on the right-hand side of \eqref{decompisition_Sigma_intermediate_1}
    \begin{equation}
    \begin{split}
        &\mathbb{E} \left [ \int_0^T \int_0^1 \| \Sigma_X(r,X(r) + \theta (X^{\varepsilon}(r) - X(r)), \mathbb{E} [ X(r) ] , u(r)) - \Sigma_X(r) \|_{L(H,L(\mathbb{R}^w,H))}^{4j} \mathrm{d}\theta \mathrm{d}r \right ]^{\frac12}\\
        &= \mathbb{E} \left [ \int_0^T \int_0^1 \left \| \int_0^1 \theta \Sigma_{XX}(r,X(r) + \gamma\theta (X^{\varepsilon}(r) - X(r)), \mathbb{E} [ X(r) ] , u(r)) (X^{\varepsilon}(r) - X(r)) \mathrm{d}\gamma \right \|_{L(H,L(\mathbb{R}^w,H))}^{4j} \mathrm{d}\theta \mathrm{d}r \right ]^{\frac12}\\
        &\leq C \mathbb{E} \left [ \sup_{r\in [0,T]} \| X^{\varepsilon}(r) - X(r) \|^{4j} \right ]^{\frac12} = \mathcal{O}(\varepsilon^j).
    \end{split}
    \end{equation}
    Thus, due to \eqref{asymptotic_Y_e}, we derive from \eqref{decompisition_Sigma_intermediate_1}
    \begin{equation}\label{decomposition_Sigma_intermediate_2}
        \mathbb{E} \left [ \left ( \int_0^t \left \| \int_0^1 \left ( \Sigma_X(r,X(r) + \theta (X^{\varepsilon}(r) - X(r)), \mathbb{E} [ X(r) ] , u^{\varepsilon}(r)) - \Sigma_X(r) \right ) Y^{\varepsilon}(r) \mathrm{d}\theta \right \|_{L(\mathbb{R}^w,H)}^{2} \mathrm{d}r \right )^j \right ] = \mathcal{O}(\varepsilon^{2j}).
    \end{equation}
    For the fourth term in \eqref{decomposition_Sigma}, we have due to Proposition \ref{proposition_asymptotic_expectation_new}
    \begin{equation}\label{decomposition_Sigma_term_4}
    \begin{split}
        &\mathbb{E} \left [ \left ( \int_0^t \left \| \int_0^1 \left ( \Sigma_Y(r,X^{\varepsilon}(r), \mathbb{E} [ X(r) + \theta (X^{\varepsilon}(r) - X(r)) ] , u^{\varepsilon}(r)) - \Sigma_Y(r) \right ) \mathbb{E} [ Y^{\varepsilon}(r) ] \mathrm{d}\theta \right \|_{L(\mathbb{R}^w,H)}^{2} \mathrm{d}r \right )^j \right ]\\
        &\leq C \sup_{r\in [0,T]} \left \| \mathbb{E} \left [ Y^{\varepsilon}(r) \right ] \right \|^{2j} = \mathcal{O}(\varepsilon^{2j}).
    \end{split}
    \end{equation}
    Therefore, using \eqref{decomposition_Sigma}, \eqref{decomposition_sigma_term_1_and_2}, \eqref{decomposition_Sigma_intermediate_2} and \eqref{decomposition_Sigma_term_4}, we derive from \eqref{bdg_X_e_X_Y_e}
    \begin{equation}\label{tbd}
    \begin{split}
        &\mathbb{E} \left [ \sup_{s\in [0,t]} \left \| \int_0^s e^{(s-r)A} \left ( \Sigma^{\varepsilon}(r) - \Sigma(r) - \Sigma_X(r) Y^{\varepsilon}(r) - \Sigma_{Y}(r) \mathbb{E} [ Y^{\varepsilon}(r) ] - \delta^v \Sigma(r) \right ) \mathrm{d}W(r) \right \|^{2j} \right ]\\
        &\leq C \mathbb{E} \left [ \int_0^t \sup_{r\in [0,s]} \left \| X^{\varepsilon}(r) - X(r) - Y^{\varepsilon}(r) \right \|^{2j} \mathrm{d}s \right ] + \mathcal{O}(\varepsilon^{2j}).
    \end{split}
    \end{equation}
    Using similar arguments, we can obtain the same bound for the term involving $B$ in equation \eqref{estimate_1_X_e_X_Y_e}. Thus, taking the power $j$, the supremum over $s\in [0,t]$ and the expectation in inequality \eqref{estimate_1_X_e_X_Y_e}, we obtain
    \begin{equation}
        \mathbb{E} \left [ \sup_{r\in [0,t]} \left \| X^{\varepsilon}(r) - X(r) - Y^{\varepsilon}(r) \right \|^{2j} \right ] \leq C \int_0^t \mathbb{E} \left [ \sup_{r\in [0,s]} \| X^{\varepsilon}(r) - X(r) - Y^{\varepsilon}(r) \|^{2j} \right ] \mathrm{d}s + \mathcal{O}(\varepsilon^{2j}).
    \end{equation}
    Now, the result follows from Gr\"onwall's inequality.

    Next, let us turn to inequality \eqref{asymptotic_X_e_X_Y_e_Z_e}. Let $\eta^{\varepsilon}(t) := X^{\varepsilon}(t) - X(t) - Y^{\varepsilon}(t) - Z^{\varepsilon}(t)$. Then, we have
    \begin{equation}
    \begin{cases}
        \mathrm{d}\eta^{\varepsilon}(t) = [ A \eta^{\varepsilon}(t) + \alpha^{\varepsilon}(t) ] \mathrm{d}t + \beta^{\varepsilon}(t) \mathrm{d}W(t),\quad t\in [0,T]\\
        \eta^{\varepsilon}(0) = 0 \in H,
    \end{cases}
    \end{equation}
    where
    \begin{equation}
    \begin{split}
        \alpha^{\varepsilon}(t) &= B^{\varepsilon}(t) - B(t) - B_X(t) (Y^{\varepsilon}(t) + Z^{\varepsilon}(t)) - B_{Y}(t) \mathbb{E} [ Y^{\varepsilon}(t) + Z^{\varepsilon}(t) ]\\
        &\quad - \frac12 B_{XX}(t) Y^{\varepsilon}(t)^2 - \delta^v B(t) - \delta^v B_X(t) Y^{\varepsilon}(t)
    \end{split}
    \end{equation}
    and
    \begin{equation}
    \begin{split}
        \beta^{\varepsilon}(t) &= \Sigma^{\varepsilon}(t) - \Sigma(t) - \Sigma_X(t) (Y^{\varepsilon}(t) + Z^{\varepsilon}(t)) - \Sigma_{Y}(t) \mathbb{E} [ Y^{\varepsilon}(t) + Z^{\varepsilon}(t) ]\\
        &\quad - \frac12 \Sigma_{XX}(t) Y^{\varepsilon}(t)^2 - \delta^v \Sigma(t) - \delta^v \Sigma_X(t) Y^{\varepsilon}(t).
    \end{split}
    \end{equation}
    Thus, we have
    \begin{equation}\label{estimate_eta_epsilon}
        \| \eta^{\varepsilon}(s) \|^2 \leq 2 \left \| \int_0^s e^{(s-r)A} \alpha^{\varepsilon}(r) \mathrm{d}r \right \|^2 + 2 \left \| \int_0^s e^{(s-r)A} \beta^{\varepsilon}(r) \mathrm{d}W(r) \right \|^2.
    \end{equation}
    Applying \cite[Lemma 3.3]{gawarecki_mandrekar_2011}, we obtain for the stochastic integral
    \begin{equation}
        \mathbb{E} \left [ \sup_{s\in [0,t]} \left \| \int_0^s e^{(s-r)A} \beta^{\varepsilon}(r) \mathrm{d}W(r) \right \|^{2j} \right ] \leq C \mathbb{E} \left [ \int_0^t \| \beta^{\varepsilon}(s) \|_{L(\mathbb{R}^w,H)}^{2j} \mathrm{d}s \right ].
    \end{equation}
    Therefore, from \eqref{estimate_eta_epsilon} we obtain
    \begin{equation}\label{estimate_eta_epsilon_2}
        \mathbb{E} \left [ \sup_{s\in [0,t]} \| \eta^{\varepsilon}(s) \|^{2j} \right ] \leq C \mathbb{E} \left [ \int_0^t \left ( \| \alpha^{\varepsilon}(s) \|^{2j} + \| \beta^{\varepsilon}(s) \|_{L(\mathbb{R}^w,H)}^{2j} \right ) \mathrm{d}s \right ].
    \end{equation}
    Next, we are going to rewrite $\alpha^{\varepsilon}$. To this end, first note that
    \begin{equation}
    \begin{split}
        &B^{\varepsilon}(s) - B(s) - B_X(s) (Y^{\varepsilon}(s) + Z^{\varepsilon}(s))- B_{Y}(s) \mathbb{E} [ Y^{\varepsilon}(s) + Z^{\varepsilon}(s) ] - \delta^v B(s)\\
        &= \int_0^1 \bigg ( B_X(s,X(s) + \theta ( X^{\varepsilon}(s) - X(s) ), \mathbb{E} [ X(s) ] , u^{\varepsilon}(s) ) \eta^{\varepsilon}(s)\\
        &\qquad\qquad + B_Y(s,X^{\varepsilon}(s), \mathbb{E} [ X(s) + \theta ( X^{\varepsilon}(s) - X(s) ) ] , u^{\varepsilon}(s) ) \mathbb{E} [ \eta^{\varepsilon}(s) ]\\
        &\qquad\qquad + ( B_X(s,X(s) + \theta ( X^{\varepsilon}(s) - X(s) ), \mathbb{E} [ X(s) ] , u^{\varepsilon}(s) ) - B_X(s) ) ( Y^{\varepsilon}(s) + Z^{\varepsilon}(s))\\
        &\qquad\qquad + ( B_Y(s,X^{\varepsilon}(s) , \mathbb{E} [ X(s) + \theta ( X^{\varepsilon}(s) - X(s) ) ] , u^{\varepsilon}(s) ) - B_Y(s) ) \mathbb{E} [ Y^{\varepsilon}(s) + Z^{\varepsilon}(s) ] \bigg ) \mathrm{d}\theta.
    \end{split}
    \end{equation}
    Moreover, we have
    \begin{equation}
    \begin{split}
        &B_X(s,X(s) + \theta ( X^{\varepsilon}(s) - X(s) ), \mathbb{E} [ X(s) ] , u^{\varepsilon}(s) ) - B_X(s) - \delta^v B_X(s)\\
        &= \int_0^1 \theta B_{XX}(s,X(s)+\gamma\theta(X^{\varepsilon}(s)-X(s)),\mathbb{E}[X(s)],u^{\varepsilon}(s)) \eta^{\varepsilon}(s) \mathrm{d}\gamma\\
        &\quad + \int_0^1 \theta B_{XX}(s,X(s) + \gamma\theta(X^{\varepsilon}(s)-X(s)),\mathbb{E}[X(s)],u^{\varepsilon}(s))(Y^{\varepsilon}(s)+Z^{\varepsilon}(s)) \mathrm{d}\gamma.
    \end{split}
    \end{equation}
    Therefore,
    \begin{equation}
    \begin{split}
        &\int_0^1 ( B_X(s,X(s) + \theta ( X^{\varepsilon}(s) - X(s) ), \mathbb{E} [ X(s) ] , u^{\varepsilon}(s) ) - B_X(s) ) Y^{\varepsilon}(s) \mathrm{d}\theta - \delta^v B_X(s) Y^{\varepsilon}(s) - \frac12 B_{XX}(s) Y^{\varepsilon}(s)^2\\
        &= \int_0^1 \int_0^1 \theta B_{XX}(s,X(s)+\gamma\theta (X^{\varepsilon}(s)-X(s)), \mathbb{E}[X(s)],u^{\varepsilon}(s)) \eta^{\varepsilon}(s) Y^{\varepsilon}(s) \mathrm{d}\gamma\mathrm{d}\theta\\
        &\quad + \int_0^1 \int_0^1 \theta B_{XX}(s,X(s)+\gamma\theta (X^{\varepsilon}(s)-X(s)), \mathbb{E}[X(s)],u^{\varepsilon}(s)) Z^{\varepsilon}(s) Y^{\varepsilon}(s) \mathrm{d}\gamma\mathrm{d}\theta\\
        &\quad + \int_0^1 \int_0^1 \theta \left ( B_{XX}(s,X(s)+\gamma\theta(X^{\varepsilon}(s)-X(s)), \mathbb{E}[X(s)],u^{\varepsilon}(s)) - B_{XX}(s) \right ) Y^{\varepsilon}(s)^2 \mathrm{d}\gamma\mathrm{d}\theta.
    \end{split}
    \end{equation}
    Altogether, we obtain
    \begin{equation}\label{alpha_decomposition}
    \begin{split}
        \alpha^{\varepsilon}(s) &= \int_0^1 B_X(s,X(s) + \theta ( X^{\varepsilon}(s) - X(s) ), \mathbb{E} [ X(s) ] , u^{\varepsilon}(s) ) \eta^{\varepsilon}(s) \mathrm{d}\theta\\
        &\quad + \int_0^1 B_Y(s,X^{\varepsilon}(s), \mathbb{E} [ X(s) + \theta ( X^{\varepsilon}(s) - X(s) ) ] , u^{\varepsilon}(s) ) \mathbb{E} [ \eta^{\varepsilon}(s) ] \mathrm{d}\theta \\
        &\quad + \int_0^1 ( B_X(s,X(s) + \theta ( X^{\varepsilon}(s) - X(s) ), \mathbb{E} [ X(s) ] , u^{\varepsilon}(s) ) - B_X(s) ) Z^{\varepsilon}(s) \mathrm{d}\theta\\
        &\quad + \int_0^1 ( B_Y(s,X^{\varepsilon}(s) , \mathbb{E} [ X(s) + \theta ( X^{\varepsilon}(s) - X(s) ) ] , u^{\varepsilon}(s) ) - B_Y(s) ) \mathbb{E} [ Y^{\varepsilon}(s) + Z^{\varepsilon}(s) ] \mathrm{d}\theta\\
        &\quad + \int_0^1 \int_0^1 \theta B_{XX}(s,X(s)+\gamma\theta (X^{\varepsilon}(s)-X(s)), \mathbb{E}[X(s)],u^{\varepsilon}(s)) \eta^{\varepsilon}(s) Y^{\varepsilon}(s) \mathrm{d}\gamma\mathrm{d}\theta\\
        &\quad + \int_0^1 \int_0^1 \theta B_{XX}(s,X(s)+\gamma\theta (X^{\varepsilon}(s)-X(s)), \mathbb{E}[X(s)],u^{\varepsilon}(s)) Z^{\varepsilon}(s) Y^{\varepsilon}(s) \mathrm{d}\gamma\mathrm{d}\theta\\
        &\quad + \int_0^1 \int_0^1 \theta \left ( B_{XX}(s,X(s)+\gamma\theta(X^{\varepsilon}(s)-X(s)), \mathbb{E}[X(s)],u^{\varepsilon}(s)) - B_{XX}(s) \right ) Y^{\varepsilon}(s)^2 \mathrm{d}\gamma\mathrm{d}\theta.
    \end{split}
    \end{equation}
    Since $B_X$ and $B_Y$ are bounded, we have for the first two terms on the right-hand side of this equation,
    \begin{equation}
    \begin{split}
        &\mathbb{E} \left [ \left ( \int_0^t \left \| \int_0^1 B_X(s,X(s) + \theta ( X^{\varepsilon}(s) - X(s) ), \mathbb{E} [ X(s) ] , u^{\varepsilon}(s) ) \eta^{\varepsilon}(s) \mathrm{d}\theta \right \|^2 \mathrm{d}s \right )^{j} \right ]\\
        &+ \mathbb{E} \left [ \left ( \int_0^t \left \| \int_0^1 B_Y(s,X^{\varepsilon}(s), \mathbb{E} [ X(s) + \theta ( X^{\varepsilon}(s) - X(s) ) ] , u^{\varepsilon}(s) ) \mathbb{E} [ \eta^{\varepsilon}(s) ] \mathrm{d}\theta \right \|^2 \mathrm{d}s \right )^j \right ] \leq C \int_0^t \mathbb{E} \left [ \sup_{r\in [0,s]} \| \eta^{\varepsilon}(r) \|^{2j} \right ] \mathrm{d}s.
    \end{split}
    \end{equation}
    For the third term on the right-hand side of \eqref{alpha_decomposition}, we have by \eqref{asymptotic_Z_e} and Lebesgue's dominated convergence theorem
    \begin{equation}\label{estimate_B_X_Z}
    \begin{split}
        &\mathbb{E} \left [ \left ( \int_0^t \left \| \int_0^1 ( B_X(s,X(s) + \theta ( X^{\varepsilon}(s) - X(s) ), \mathbb{E} [ X(s) ] , u^{\varepsilon}(s) ) - B_X(s) ) Z^{\varepsilon}(s) \mathrm{d}\theta \right \|^2 \mathrm{d}s \right )^{j} \right ]\\
        &\leq C \mathbb{E} \Bigg [ \int_0^T \int_0^1 \| B_X(s,X(s)+\theta(X^{\varepsilon}(s)-X(s)),\mathbb{E}[X(s)],u^{\varepsilon}(s))\\
        &\qquad\qquad\qquad\qquad - B_X(s,X(s)+\theta(X^{\varepsilon}(s)-X(s)),\mathbb{E}[X(s)],u(s)) \|_{L(H)}^{4j} \mathrm{d}\theta \mathrm{d}s \Bigg ]^{\frac12} \mathbb{E} \left [ \int_0^T \| Z^{\varepsilon}(s) \|^{4j} \mathrm{d}s \right ]^{\frac12}\\
        &\quad + C \mathbb{E} \left [ \int_0^T \int_0^1 \| B_X(s,X(s)+\theta(X^{\varepsilon}(s)-X(s)),\mathbb{E}[X(s)],u(s)) - B_X(s) \|_{L(H)}^{4j} \mathrm{d}\theta \mathrm{d}s \right ]^{\frac12} \mathbb{E} \left [ \int_0^T \| Z^{\varepsilon}(s) \|^{4j} \mathrm{d}s \right ]^{\frac12}\\
        &= o(\varepsilon^{2j}).
    \end{split}
    \end{equation}
    For the fourth term on the right-hand side of \eqref{alpha_decomposition}, we obtain the same asymptotics using similar arguments and Proposition \ref{proposition_asymptotic_expectation_new}. For the fifth term on the right-hand side of \eqref{alpha_decomposition}, we have due to the boundedness of $B_{XX}$, the definition of $\eta^{\varepsilon}$, and the estimates \eqref{asymptotic_Y_e}, \eqref{asymptotic_Z_e}, and \eqref{asymptotic_X_e_X_Y_e}
    \begin{equation}
    \begin{split}
        &\mathbb{E} \left [ \left ( \int_0^t \left \| \int_0^1 \int_0^1 \theta B_{XX}(s,X(s)+\gamma\theta (X^{\varepsilon}(s)-X(s)), \mathbb{E}[X(s)],u^{\varepsilon}(s)) \eta^{\varepsilon}(s) Y^{\varepsilon}(s) \mathrm{d}\gamma\mathrm{d}\theta \right \|^2 \mathrm{d}s \right )^{j} \right ]\\
        &\leq \mathbb{E} \left [ \sup_{s\in [0,T]} \| \eta^{\varepsilon}(s) \|^{4j} \right ]^{\frac12} \mathbb{E} \left [ \sup_{s\in [0,T]} \| Y^{\varepsilon}(s) \|^{4j} \right ]^{\frac12} = o(\varepsilon^{2j}).
    \end{split}
    \end{equation}
    Similarly, we obtain the same asymptotics for the sixth term on the right-hand side of \eqref{alpha_decomposition}. The last term on the right-hand side of \eqref{alpha_decomposition} can be estimated as in \eqref{estimate_B_X_Z}. Altogether, we obtain
    \begin{equation}
        \mathbb{E} \left [ \int_0^t \| \alpha^{\varepsilon}(s) \|^{2j} \mathrm{d}s \right ] \leq C \int_0^t \mathbb{E} \left [ \sup_{r\in [0,s]} \| \eta^{\varepsilon}(r) \|^{2j} \right ] \mathrm{d}s + o(\varepsilon^{2j}).
    \end{equation}
    Using the same arguments, we obtain the same bound for the corresponding term involving $\beta^{\varepsilon}$. Thus, we derive from \eqref{estimate_eta_epsilon_2}
    \begin{equation}
        \mathbb{E} \left [ \sup_{s\in [0,t]} \| \eta^{\varepsilon}(s) \|^{2j} \right ] \leq C \int_0^t \mathbb{E} \left [ \sup_{r\in [0,s]} \| \eta^{\varepsilon}(r) \|^{2j} \right ] \mathrm{d}s + o(\varepsilon^{2j}).
    \end{equation}
    Applying Gr\"onwall's inequality concludes the proof.
\end{proof}

\subsection{Expansion of the Cost Functional}

In this section, using Taylor expansions, we derive an expansion of the cost functional in terms of the first and second variational processes.

\begin{proposition}\label{proposition_cost_functional_expansion}
    Let Assumptions \ref{assumption_b_sigma} and \ref{assumption_l_m} be satisfied. Then, we have for almost every $\tau \in (0,T)$
    \begin{equation}\label{proposition_cost_functional_expansion_equation}
    \begin{split}
        J(u^{\varepsilon}(\cdot)) - J(u(\cdot))&= \mathbb{E} [ M_X(T) ( Y^{\varepsilon}(T) + Z^{\varepsilon}(T)) ] + \frac12 \mathbb{E} [ M_{XX}(T) Y^{\varepsilon}(T)^2 ] + \mathbb{E} [ M_{Y}(T) ] \mathbb{E} [ Y^{\varepsilon}(T) + Z^{\varepsilon}(T) ] \\
        &\quad + \mathbb{E} \left [ \int_0^T L_X(t) (Y^{\varepsilon}(t) + Z^{\varepsilon}(t)) \mathrm{d}t \right ]+ \frac12 \mathbb{E} \left [ \int_0^T L_{XX}(t) Y^{\varepsilon}(t)^2 \mathrm{d}t \right ]\\
        &\quad + \int_0^T \mathbb{E} \left [ L_{Y}(t) \right ] \mathbb{E} [ Y^{\varepsilon}(t) + Z^{\varepsilon}(t) ] \mathrm{d}t + \mathbb{E} \left [ \int_0^T \delta^v L(t) \mathrm{d}t \right ] + o(\varepsilon).
    \end{split}
    \end{equation}
    The exceptional set for $\tau \in (0,T)$ is independent of the set $\mathfrak{F} \in \mathcal{F}_{\tau}$ that is used in the construction of the spike variation in \eqref{spike_variation_definition}.
\end{proposition}

\begin{proof}
    From the definition of the cost functional, it follows that
    \begin{equation}
    \begin{split}
        J(u^{\varepsilon}(\cdot)) - J(u(\cdot)) &= \mathbb{E} \left [ \int_0^T \left ( L(t,X^{\varepsilon}(t),\mathbb{E}[X^{\varepsilon}(t)],u^{\varepsilon}(t)) - L(t,X(t),\mathbb{E}[X(t)],u(t)) \right ) \mathrm{d}t \right ]\\
        &\quad + \mathbb{E} \left [ M(X^{\varepsilon}(T),\mathbb{E}[X^{\varepsilon}(T)]) - M(X(T),\mathbb{E}[X(T)]) \right ].
    \end{split}
    \end{equation}
    Let us consider the term involving $L$. We have the expansion\begingroup\makeatletter\def\f@size{9}\check@mathfonts
    \begin{equation}\label{difference_L}
    \begin{split}
        &\mathbb{E} \left [ \int_0^T \left ( L(t,X^{\varepsilon}(t),\mathbb{E}[X^{\varepsilon}(t)],u^{\varepsilon}(t)) - L(t,X(t),\mathbb{E}[X(t)],u(t)) \right ) \mathrm{d}t \right ]\\
        &= \mathbb{E} \left [ \int_0^T \left ( L(t,X^{\varepsilon}(t),\mathbb{E}[X^{\varepsilon}(t)],u^{\varepsilon}(t)) - L(t,X(t)+Y^{\varepsilon}(t) + Z^{\varepsilon}(t),\mathbb{E}[X(t)+Y^{\varepsilon}(t)+Z^{\varepsilon}(t)],u^{\varepsilon}(t)) \right ) \mathrm{d}t \right ]\\
        &\quad + \mathbb{E} \bigg [ \int_0^T \Big ( L(t,X(t)+Y^{\varepsilon}(t) + Z^{\varepsilon}(t),\mathbb{E}[X(t)+Y^{\varepsilon}(t)+Z^{\varepsilon}(t)],u^{\varepsilon}(t)) \\
        &\qquad\qquad\qquad\qquad - L(t,X(t)+Y^{\varepsilon}(t) + Z^{\varepsilon}(t),\mathbb{E}[X(t)+Y^{\varepsilon}(t)+Z^{\varepsilon}(t)],u(t)) \Big ) \mathrm{d}t \bigg ]\\
        &\quad + \mathbb{E} \left [ \int_0^T \left ( L(t,X(t)+Y^{\varepsilon}(t) + Z^{\varepsilon}(t),\mathbb{E}[X(t)+Y^{\varepsilon}(t)+Z^{\varepsilon}(t)],u(t)) - L(t,X(t)+Y^{\varepsilon}(t) + Z^{\varepsilon}(t),\mathbb{E}[X(t)],u(t)) \right ) \mathrm{d}t \right ]\\
        &\quad + \mathbb{E} \left [ \int_0^T \left ( L(t,X(t)+Y^{\varepsilon}(t) + Z^{\varepsilon}(t),\mathbb{E}[X(t)],u(t)) - L(t,X(t),\mathbb{E}[X(t)],u(t)) \right ) \mathrm{d}t \right ].
    \end{split}
    \end{equation}\endgroup
    For the first term on the right-hand side, we have by \eqref{asymptotic_X_e_X_Y_e_Z_e} and since $L_X$ and $L_Y$ have linear growth, see Lemma \ref{lemma_properties_L_M}, 
    \begin{equation}
    \begin{split}
        &\mathbb{E} \left [ \int_0^T \left ( L(t,X^{\varepsilon}(t),\mathbb{E}[X^{\varepsilon}(t)],u^{\varepsilon}(t)) - L(t,X(t)+Y^{\varepsilon}(t) + Z^{\varepsilon}(t),\mathbb{E}[X(t)+Y^{\varepsilon}(t)+Z^{\varepsilon}(t)],u^{\varepsilon}(t)) \right ) \mathrm{d}t \right ]\\
        &= \mathbb{E} \left [ \int_0^T \int_0^1 L_X(t,\theta X^{\varepsilon}(t) + (1-\theta)( X(t)+Y^{\varepsilon}(t) + Z^{\varepsilon}(t)),\mathbb{E}[X^{\varepsilon}(t)],u^{\varepsilon}(t)) (X^{\varepsilon}(t) - X(t)-Y^{\varepsilon}(t) - Z^{\varepsilon}(t)) \mathrm{d}\theta \mathrm{d}t \right ]\\
        &\quad + \mathbb{E} \Bigg [ \int_0^T \int_0^1 L_{Y}(t,X(t) + Y^{\varepsilon}(t) + Z^{\varepsilon}(t), \mathbb{E}[\theta X^{\varepsilon}(t) + (1-\theta) (X(t)+Y^{\varepsilon}(t)+Z^{\varepsilon}(t))], u^{\varepsilon}(t))\\
        &\qquad\qquad\qquad\qquad\qquad\qquad \mathbb{E} \left [ X^{\varepsilon}(t) - X(t)-Y^{\varepsilon}(t)-Z^{\varepsilon}(t) \right ] \mathrm{d}\theta \mathrm{d}t \Bigg ]\\
        &= o(\varepsilon).
    \end{split}
    \end{equation}
    For the second term on the right-hand side of \eqref{difference_L}, we have
    \begin{equation}\label{difference_L_2}
    \begin{split}
        &\mathbb{E} \bigg [ \int_0^T \Big ( L(t,X(t)+Y^{\varepsilon}(t) + Z^{\varepsilon}(t),\mathbb{E}[X(t)+Y^{\varepsilon}(t)+Z^{\varepsilon}(t)],u^{\varepsilon}(t))\\
        &\qquad\qquad - L(t,X(t)+Y^{\varepsilon}(t) + Z^{\varepsilon}(t),\mathbb{E}[X(t)+Y^{\varepsilon}(t)+Z^{\varepsilon}(t)],u(t)) \Big ) \mathrm{d}t \bigg ]\\
        &= \mathbb{E} \left [ \int_0^T \left ( L(t,X(t),\mathbb{E}[X(t)],u^{\varepsilon}(t))  - L(t,X(t),\mathbb{E}[X(t)],u(t)) \right ) \mathrm{d}t \right ]\\
        &\quad + \mathbb{E} \left [ \int_0^T \left ( L(t,X(t)+Y^{\varepsilon}(t) + Z^{\varepsilon}(t),\mathbb{E}[X(t)+Y^{\varepsilon}(t)+Z^{\varepsilon}(t)],u^{\varepsilon}(t)) - L(t,X(t),\mathbb{E}[X(t)],u^{\varepsilon}(t)) \right ) \mathrm{d}t \right ]\\
        &\quad + \mathbb{E} \left [ \int_0^T \left ( L(t,X(t),\mathbb{E}[X(t)],u(t)) - L(t,X(t)+Y^{\varepsilon}(t) + Z^{\varepsilon}(t),\mathbb{E}[X(t)+Y^{\varepsilon}(t)+Z^{\varepsilon}(t)],u(t)) \right ) \mathrm{d}t \right ].
    \end{split}
    \end{equation}
    The first term on the right-hand side of this equation yields the term involving $\delta^v L$ in equation \eqref{proposition_cost_functional_expansion_equation}. For the second and third line on the right-hand side of this equation, we note that
    \begin{equation}\label{decomposition_L_intermediate_step}
    \begin{split}
        &\mathbb{E} \left [ \int_0^T \left ( L(t,X(t)+Y^{\varepsilon}(t) + Z^{\varepsilon}(t),\mathbb{E}[X(t)+Y^{\varepsilon}(t)+Z^{\varepsilon}(t)],u^{\varepsilon}(t)) - L(t,X(t),\mathbb{E}[X(t)],u^{\varepsilon}(t)) \right ) \mathrm{d}t \right ]\\
        &\quad + \mathbb{E} \left [ \int_0^T \left ( L(t,X(t),\mathbb{E}[X(t)],u(t)) - L(t,X(t)+Y^{\varepsilon}(t) + Z^{\varepsilon}(t),\mathbb{E}[X(t)+Y^{\varepsilon}(t)+Z^{\varepsilon}(t)],u(t)) \right ) \mathrm{d}t \right ]\\
        &= \mathbb{E} \left [ \int_0^T \int_0^1 L_X(t,X(t) + \theta (Y^{\varepsilon}(t) + Z^{\varepsilon}(t)),\mathbb{E}[X(t)+Y^{\varepsilon}(t)+Z^{\varepsilon}(t)],u^{\varepsilon}(t)) (Y^{\varepsilon}(t) + Z^{\varepsilon}(t)) \mathrm{d}\theta \mathrm{d}t \right ]\\
        &\quad - \mathbb{E} \left [ \int_0^T \int_0^1 L_X(t,X(t) + \theta (Y^{\varepsilon}(t) + Z^{\varepsilon}(t)),\mathbb{E}[X(t)+Y^{\varepsilon}(t)+Z^{\varepsilon}(t)],u(t)) (Y^{\varepsilon}(t) + Z^{\varepsilon}(t)) \mathrm{d}\theta \mathrm{d}t \right ]\\
        &\quad + \mathbb{E} \left [ \int_0^T \int_0^1 L_{Y}(t,X(t),\mathbb{E}[X(t)+\theta (Y^{\varepsilon}(t)+Z^{\varepsilon}(t))],u^{\varepsilon}(t)) \mathbb{E} \left [ Y^{\varepsilon}(t)+Z^{\varepsilon}(t) \right ] \mathrm{d}\theta \mathrm{d}t \right ]\\
        &\quad - \mathbb{E} \left [ \int_0^T \int_0^1 L_{Y}(t,X(t),\mathbb{E}[X(t)+\theta (Y^{\varepsilon}(t)+Z^{\varepsilon}(t))],u(t)) \mathbb{E} \left [ Y^{\varepsilon}(t)+Z^{\varepsilon}(t) \right ] \mathrm{d}\theta \mathrm{d}t \right ].
    \end{split}
    \end{equation}
    For the first and second line on the right-hand side of this equation, we have
    \begin{equation}
    \begin{split}
        &\Bigg | \mathbb{E} \left [ \int_0^T \int_0^1 L_X(t,X(t) + \theta (Y^{\varepsilon}(t) + Z^{\varepsilon}(t)),\mathbb{E}[X(t)+Y^{\varepsilon}(t)+Z^{\varepsilon}(t)],u^{\varepsilon}(t)) (Y^{\varepsilon}(t) + Z^{\varepsilon}(t)) \mathrm{d}\theta \mathrm{d}t \right ]\\
        & - \mathbb{E} \left [ \int_0^T \int_0^1 L_X(t,X(t) + \theta (Y^{\varepsilon}(t) + Z^{\varepsilon}(t)),\mathbb{E}[X(t)+Y^{\varepsilon}(t)+Z^{\varepsilon}(t)],u(t)) (Y^{\varepsilon}(t) + Z^{\varepsilon}(t)) \mathrm{d}\theta \mathrm{d}t \right ] \Bigg |\\
        &\leq \mathbb{E} \Bigg [ \bigg ( \int_0^T \int_0^1 \Big \| L_X(t,X(t) + \theta (Y^{\varepsilon}(t) + Z^{\varepsilon}(t)),\mathbb{E}[X(t)+Y^{\varepsilon}(t)+Z^{\varepsilon}(t)],u^{\varepsilon}(t)) \\
        &\qquad\qquad\qquad\qquad - L_X(t,X(t) + \theta (Y^{\varepsilon}(t) + Z^{\varepsilon}(t)),\mathbb{E}[X(t)+Y^{\varepsilon}(t)+Z^{\varepsilon}(t)],u(t))  \Big \| \mathrm{d}\theta \mathrm{d}t \bigg )^2 \Bigg ]^{\frac12} \\
        &\quad \times \mathbb{E} \left [ \sup_{t\in [0,T]} \| Y^{\varepsilon}(t) + Z^{\varepsilon}(t) \|^2 \right ]^{\frac12}
    \end{split}
    \end{equation}
    Note that the first integral on the right-hand side is actually an integral over $[\tau,\tau+\varepsilon]$. Thus, using Jensen's inequality and the linear growth of $L_X$, we have
    \begin{equation}
    \begin{split}
        &\mathbb{E} \Bigg [ \bigg ( \frac{1}{\varepsilon} \int_{\tau}^{\tau+\varepsilon} \int_0^1 \Big \| L_X(t,X(t) + \theta (Y^{\varepsilon}(t) + Z^{\varepsilon}(t)),\mathbb{E}[X(t)+Y^{\varepsilon}(t)+Z^{\varepsilon}(t)],u^{\varepsilon}(t)) \\
        &\qquad\qquad\qquad - L_X(t,X(t) + \theta (Y^{\varepsilon}(t) + Z^{\varepsilon}(t)),\mathbb{E}[X(t)+Y^{\varepsilon}(t)+Z^{\varepsilon}(t)],u(t))  \Big \| \mathrm{d}\theta \mathrm{d}t \bigg )^2 \Bigg ]^{\frac12}\\
        &\leq \mathbb{E} \Bigg [ \frac{1}{\varepsilon} \int_{\tau}^{\tau+\varepsilon} \int_0^1 \Big \| L_X(t,X(t) + \theta (Y^{\varepsilon}(t) + Z^{\varepsilon}(t)),\mathbb{E}[X(t)+Y^{\varepsilon}(t)+Z^{\varepsilon}(t)],u^{\varepsilon}(t)) \\
        &\qquad\qquad\qquad - L_X(t,X(t) + \theta (Y^{\varepsilon}(t) + Z^{\varepsilon}(t)),\mathbb{E}[X(t)+Y^{\varepsilon}(t)+Z^{\varepsilon}(t)],u(t))  \Big \|^2 \mathrm{d}\theta \mathrm{d}t \Bigg ]^{\frac12}\\
        &\leq \left ( \frac{C}{\varepsilon} \int_{\tau}^{\tau+\varepsilon} \mathbb{E} \left [ 1+ \| X(t) \|^2 + \| Y^{\varepsilon}(t) \|^2 + \| Z^{\varepsilon}(t) \|^2 + \| v \|_U^2 + \| u(t) \|_U^2 \right ] \mathrm{d}t \right )^{\frac12}.
    \end{split}
    \end{equation}
    The right-hand side is finite for almost all $\tau$ in the limit $\varepsilon \to 0$. Using similar arguments for the third and fourth line on the right-hand side of equation \eqref{decomposition_L_intermediate_step} and using \eqref{asymptotic_Y_e} and \eqref{asymptotic_Z_e} shows that
    \begin{equation}
    \begin{split}
        &\mathbb{E} \left [ \int_0^T L(t,X(t)+Y^{\varepsilon}(t) + Z^{\varepsilon}(t),\mathbb{E}[X(t)+Y^{\varepsilon}(t)+Z^{\varepsilon}(t)],u^{\varepsilon}(t)) - L(t,X(t),\mathbb{E}[X(t)],u^{\varepsilon}(t)) \mathrm{d}t \right ]\\
        & + \mathbb{E} \left [ \int_0^T L(t,X(t),\mathbb{E}[X(t)],u(t)) - L(t,X(t)+Y^{\varepsilon}(t) + Z^{\varepsilon}(t),\mathbb{E}[X(t)+Y^{\varepsilon}(t)+Z^{\varepsilon}(t)],u(t)) \mathrm{d}t \right ] = o(\varepsilon).
    \end{split}
    \end{equation}
    For the third term in \eqref{difference_L}, due to Proposition \ref{proposition_asymptotic_expectation_new} we have
    \begin{equation}
    \begin{split}
        &\mathbb{E} \left [ \int_0^T L(t,X(t)+Y^{\varepsilon}(t) + Z^{\varepsilon}(t),\mathbb{E}[X(t)+Y^{\varepsilon}(t)+Z^{\varepsilon}(t)],u(t)) - L(t,X(t)+Y^{\varepsilon}(t) + Z^{\varepsilon}(t),\mathbb{E}[X(t)],u(t)) \mathrm{d}t \right ]\\
        &= \mathbb{E} \left [ \int_0^T L_Y(t,X(t),\mathbb{E}[X(t)],u(t)) \mathbb{E} [ Y^{\varepsilon}(t) + Z^{\varepsilon}(t)] \mathrm{d}t \right ]\\
        &\quad +\mathbb{E} \bigg [ \int_0^T \int_0^1 \bigg ( L_Y(t,X(t)+Y^{\varepsilon}(t)+Z^{\varepsilon}(t), \mathbb{E} [ X(t) + \theta (Y^{\varepsilon}(t) + Z^{\varepsilon}(t)) ], u(t))\\
        &\qquad\qquad\qquad\qquad\qquad - L_Y(t,X(t),\mathbb{E}[X(t)],u(t)) \bigg ) \mathbb{E} [ Y^{\varepsilon}(t) + Z^{\varepsilon}(t)] \mathrm{d}\theta \mathrm{d}t \bigg ]
    \end{split}
    \end{equation}
    The first term on the right-hand side yields the corresponding term in \eqref{proposition_cost_functional_expansion_equation}. For the second and third line on the right-hand side of this equation, we have
    \begin{equation}
    \begin{split}
        &\Bigg | \mathbb{E} \bigg [ \int_0^T \int_0^1 \bigg ( L_Y(t,X(t)+Y^{\varepsilon}(t)+Z^{\varepsilon}(t), \mathbb{E} [ X(t) + \theta (Y^{\varepsilon}(t) + Z^{\varepsilon}(t)) ], u(t))\\
        &\qquad\qquad\qquad\qquad\qquad - L_Y(t,X(t),\mathbb{E}[X(t)],u(t)) \bigg ) \mathbb{E} [ Y^{\varepsilon}(t) + Z^{\varepsilon}(t)] \mathrm{d}\theta \mathrm{d}t \bigg ] \Bigg |\\
        &\leq \mathbb{E} \bigg [ \int_0^T \int_0^1 \bigg \| L_Y(t,X(t)+Y^{\varepsilon}(t)+Z^{\varepsilon}(t), \mathbb{E} [ X(t) + \theta (Y^{\varepsilon}(t) + Z^{\varepsilon}(t)) ], u(t))\\
        &\qquad\qquad\qquad\qquad\qquad - L_Y(t,X(t),\mathbb{E}[X(t)],u(t)) \bigg \| \mathrm{d}\theta \mathrm{d}t \bigg ] \sup_{t\in [0,T]} \| \mathbb{E} [ Y^{\varepsilon}(t) + Z^{\varepsilon}(t)] \|,
    \end{split}
    \end{equation}
    which, due to the linear growth and continuity of $L_Y$, Lebesgue's dominated convergence theorem, equation \eqref{asymptotic_Z_e} and Proposition \ref{proposition_asymptotic_expectation_new}, is of order $o(\varepsilon)$.
    
    Finally, for the fourth term in \eqref{difference_L}, similar arguments yield
    \begin{equation}
    \begin{split}
        &\mathbb{E} \left [ \int_0^T L(t,X(t)+Y^{\varepsilon}(t) + Z^{\varepsilon}(t),\mathbb{E}[X(t)],u(t)) - L(t,X(t),\mathbb{E}[X(t)],u(t)) \mathrm{d}t \right ]\\
        &= \mathbb{E} \left [ \int_0^T \left ( L_X(t,X(t),\mathbb{E}[X(t)],u(t)) (Y^{\varepsilon}(t) + Z^{\varepsilon}(t)) + \frac12 L_{XX}(t,X(t),\mathbb{E}[X(t)],u(t)) Y^{\varepsilon}(t)^2 \right ) \mathrm{d}t \right ] + o(\varepsilon),
    \end{split}
    \end{equation}
    which concludes the discussion of the terms involving $L$. The terms involving $M$ can be handled similarly.
\end{proof}

\section{Adjoint Equations}\label{section_adjoint_equations}

\subsection{First Order Adjoint Equation}

We introduce the first order adjoint equation
\begin{equation}\label{first_order_adjoint}
\begin{cases}
    - \mathrm{d} p(t) = \Big [ A^* p(t) + B_X^*(t) p(t) + \mathbb{E} [ B_{Y}^*(t) p(t) ] - L_X(t) - \mathbb{E} [ L_{Y}(t) ] + \Sigma^*_X(t) q(t) + \mathbb{E} [ \Sigma^*_{Y}(t) q(t) ] \Big ] \mathrm{d}t\\
    \qquad\qquad\qquad - q(t) \mathrm{d}W(t)\\
    p(T) = - M_X(T) - \mathbb{E} [ M_{Y}(T)].
\end{cases}
\end{equation}

\begin{theorem}
    Let Assumptions \ref{assumption_b_sigma} and \ref{assumption_l_m} be satisfied. Then, equation \eqref{first_order_adjoint} admits a unique mild solution, i.e., there is a unique pair of processes $(p,q) \in L_{\mathcal{F}}^2(\Omega;C([0,T];H)) \times L_{\mathcal{F}}^2(\Omega \times [0,T];L(\mathbb{R}^w,H))$ such that for every $t\in [0,T]$
    \begin{equation}
    \begin{split}
        p(t) &= - e^{(T-t)A^*} M_X(T) - e^{(T-t)A^*} \mathbb{E} [ M_{Y}(T)] + \int_t^T e^{(s-t)A^*} B_X^*(s) p(s) \mathrm{d}s\\
        &\quad + \int_t^T e^{(s-t)A^*} \mathbb{E} [ B_{Y}^*(s) p(s) ] \mathrm{d}s - \int_t^T e^{(s-t)A^*} L_X(s) \mathrm{d}s - \int_t^T e^{(s-t)A^*} \mathbb{E} [ L_{Y}(s) ] \mathrm{d}s\\
        &\quad + \int_t^T e^{(s-t)A^*} \Sigma^*_X(s) q(s) \mathrm{d}s + \int_t^T e^{(s-t)A^*} \mathbb{E} [ \Sigma^*_{Y}(s) q(s) ] \mathrm{d}s - \int_t^T e^{(s-t)A^*} q(s) \mathrm{d}W(s)
    \end{split}
    \end{equation}
    $\mathbb{P}$-almost surely.
\end{theorem}

\begin{proof}
    See Theorem \ref{BSDE_existence_of_solution}.
\end{proof}

We have the following duality relations.

\begin{proposition}\label{proposition-duality}
    Let Assumptions \ref{assumption_b_sigma} and \ref{assumption_l_m} be satisfied. Then, it holds
    \begin{equation}\label{adjoint_state_property_Y}
    \begin{split}
        \mathbb{E} [ \langle p(T), Y^{\varepsilon}(T)\rangle ] &= - \mathbb{E} \left [ \langle M_X(T), Y^{\varepsilon}(T) \rangle \right ] - \langle \mathbb{E} [ M_{Y}(T) ], \mathbb{E} [ Y^{\varepsilon}(T) ] \rangle \\
        &= \mathbb{E} \left [ \int_0^T \left ( L_X(t) Y^{\varepsilon}(t) + \mathbb{E} [ L_{Y}(t) ] Y^{\varepsilon}(t) + \langle p(t), \delta^v B(t) \rangle + \langle q(t), \delta^v \Sigma(t) \rangle_{L_2(\mathbb{R}^w,H)} \right ) \mathrm{d}t \right ]
    \end{split}
    \end{equation}
    and
    \begin{equation}\label{adjoint_state_property_Z}
    \begin{split}
        \mathbb{E} [ \langle p(T), Z^{\varepsilon}(T) \rangle ] &= - \mathbb{E} \left [ \langle M_X(T), Z^{\varepsilon}(T) \rangle \right ] - \langle \mathbb{E} [ M_{Y}(T) ], \mathbb{E} [ Z^{\varepsilon}(T) ] \rangle \\
        &= \mathbb{E} \Bigg [ \int_0^T \bigg ( L_X(t) Z^{\varepsilon}(t) + \mathbb{E} [ L_{Y}(t) ] Z^{\varepsilon}(t) + \left \langle p(t), \delta^v B_X(t) Y^{\varepsilon}(t) + \frac12 B_{XX}(t) Y^{\varepsilon}(t)^2 \right \rangle\\
        &\qquad\qquad + \left \langle q(t), \delta^v \Sigma_X(t) Y^{\varepsilon}(t) + \frac12 \Sigma_{XX}(t) Y^{\varepsilon}(t)^2 \right \rangle_{L_2(\mathbb{R}^w,H)} \bigg ) \mathrm{d}t \Bigg ].
    \end{split}
    \end{equation}
\end{proposition}

\begin{proof}
    Let $A_{\lambda}$ be the Yosida approximation of $A$. Let $Y^{\varepsilon,\lambda}$ and $(p^{\lambda},q^{\lambda})$ be the solutions of equations \eqref{first_variational_equation} and \eqref{first_order_adjoint} with $A$ and $A^*$ replaced by $A_{\lambda}$ and $A_{\lambda}^*$, respectively. By It\^o's formula, we have
    \begin{equation}
    \begin{split}
        \mathbb{E} \left [ \langle p^{\lambda}(T), Y^{\varepsilon,\lambda}(T) \rangle \right ]
        &= \mathbb{E} \left [ \int_0^T \left \langle p^{\lambda}(t), A_{\lambda}Y^{\varepsilon,\lambda}(t) + B_X(t) Y^{\varepsilon,\lambda}(t) + B_{Y}(t) \mathbb{E} [ Y^{\varepsilon,\lambda}(t) ] + \delta^v B(t) \right \rangle \mathrm{d}t \right ]\\
        &\quad - \mathbb{E} \Bigg [ \int_0^T \Big \langle Y^{\varepsilon,\lambda}(t) , A_{\lambda}^* p^{\lambda}(t) + B_X^*(t) p^{\lambda}(t) + \mathbb{E} [ B_{Y}^*(t) p^{\lambda}(t) ] - L_X(t) - \mathbb{E} [ L_{Y}(t) ] \\
        &\qquad\qquad\qquad + \Sigma^*_X(t) q^{\lambda}(t) + \mathbb{E} [ \Sigma^*_{Y}(t) q^{\lambda}(t) ] \Big \rangle \mathrm{d}t \Bigg ]\\
        &\quad + \mathbb{E} \left [ \int_0^T \left \langle q^{\lambda}(t), \Sigma_X(t) Y^{\varepsilon,\lambda}(t) + \Sigma_{Y}(t) \mathbb{E} [ Y^{\varepsilon,\lambda}(t) ] + \delta^v \Sigma(t) \right \rangle_{L_2(\mathbb{R}^w,H)} \mathrm{d}t\right ]\\
        &= \mathbb{E} \left [ \int_0^T \left ( L_X(t) Y^{\varepsilon,\lambda}(t) + \mathbb{E} [ L_{Y}(t) ] Y^{\varepsilon,\lambda}(t) + \langle p^{\lambda}(t), \delta^v B(t) \rangle + \langle q^{\lambda}(t), \delta^v \Sigma(t) \rangle_{L_2(\mathbb{R}^w,H)} \right ) \mathrm{d}t \right ].
    \end{split}
    \end{equation}
    Taking the limit $\lambda\to\infty$, due to \cite[Proposition 2.10]{cosso_gozzi_kharroubi_pham_rosestolato_2023} and Theorem \ref{theorem_general_bsde_yosida}, we obtain \eqref{adjoint_state_property_Y}. Equation \eqref{adjoint_state_property_Z} follows along the same lines.
\end{proof}

\subsection{Tensor Product of First Order Variational Process}

Now, let us derive an equation for the operator-valued process associated with the tensor-valued process $Y^{\varepsilon}(t) \otimes Y^{\varepsilon}(t) \in H \otimes H$, $t\in [0,T]$. We recall that $H\otimes H$ is isometrically isomorphic to the space of Hilbert--Schmidt operators $L_2(H)$, i.e., for $f,g\in H$, the tensor $f\otimes g\in H\otimes H$ can be identified with a Hilbert--Schmidt operator from $H$ to $H$ which we denote by $\mathcal{T}_{f\otimes g}$. The identification is given by $\mathcal{T}_{f\otimes g}(h) := \langle f,h\rangle g$, $h\in H$, see e.g. \cite[Proposition 2.6.9]{kadison_ringrose_1983}.

\begin{theorem}\label{theorem_tensor_product}
    Let Assumption \ref{assumption_b_sigma} be satisfied. For $t\in [0,T]$, let $\mathcal{Y}^{\varepsilon}(t) : H \to H$ be the operator associated with $Y^{\varepsilon}(t) \otimes Y^{\varepsilon}(t)$. Then, $\mathcal{Y}^{\varepsilon} \in L^2_{\mathcal{F}}(\Omega; C([0,T];L_2(H)))$ is a mild solution in the sense of Definition \ref{definition_mild_solution_forward_equation} of the equation
    \begin{equation}\label{third_variational_equation}
    \begin{cases}
        \mathrm{d}\mathcal{Y}^{\varepsilon}(t) = \left [ (A+ B_X(t)) \mathcal{Y}^{\varepsilon}(t) + \mathcal{Y}^{\varepsilon}(t) (A^*+ B^*_X(t)) + \text{Tr} ( \Sigma_X(t) \mathcal{Y}^{\varepsilon}(t) \Sigma^*_X(t) ) + \mathcal{T}_{\Phi^{\varepsilon}}(t) \right ] \mathrm{d}t\\
        \qquad\qquad + \left [ \Sigma_X(t) \mathcal{Y}^{\varepsilon}(t) + \mathcal{Y}^{\varepsilon}(t) \Sigma^*_X(t) + \mathcal{T}_{\Psi^{\varepsilon}}(t) \right ] \mathrm{d}W(t)\\
    \mathcal{Y}^{\varepsilon}(0) = 0 \in L_2(H).
    \end{cases}
    \end{equation}
    Here, we use the notation
    \begin{equation}\label{T_Phi_epsilon_T_Psi_epsilon}
    \begin{split}
        \mathcal{T}_{\Phi^{\varepsilon}} &: [0,T] \times \Omega \to L_2(H), \qquad\qquad\quad \mathcal{T}_{\Phi^{\varepsilon}}(t) := \mathcal{T}_{\Phi^{\varepsilon}(t)}\\
        \mathcal{T}_{\Psi^{\varepsilon}} &: [0,T]\times \Omega \to L(\mathbb{R}^w,L_2(H)), \qquad \mathcal{T}_{\Psi^{\varepsilon}}(t) := ( \xi \mapsto \mathcal{T}_{\Psi^{\varepsilon}(t,\xi)} )
    \end{split}
    \end{equation}
    where $\Phi^{\varepsilon}:[0,T] \times \Omega \to H\otimes H$ is given by
    \begin{equation}\label{Phi_epsilon}
    \begin{split}
        \Phi^{\varepsilon}(t)& = Y^{\varepsilon}(t) \otimes B_{Y}(t) \mathbb{E} [ Y^{\varepsilon}(t) ] + Y^{\varepsilon}(t) \otimes \delta^v B(t) + B_{Y}(t) \mathbb{E} [ Y^{\varepsilon}(t) ] \otimes Y^{\varepsilon}(t) + \delta^v B(t) \otimes Y^{\varepsilon}(t)\\
        &\quad + \sum_{i=1}^{w} \bigg ( \Sigma_X(t) Y^{\varepsilon}(t) \xi_i \otimes \Sigma_{Y}(t) \mathbb{E} [ Y^{\varepsilon}(t) ] \xi_i + \Sigma_{Y}(t) \mathbb{E} [ Y^{\varepsilon}(t) ] \xi_i \otimes \Sigma_X(t) Y^{\varepsilon}(t) \xi_i + \Sigma_X(t) Y^{\varepsilon}(t) \xi_i \otimes \delta^v \Sigma(t) \xi_i\\
        &\qquad\qquad\quad + \delta^v \Sigma(t) \xi_i \otimes \Sigma_X(t) Y^{\varepsilon}(t) \xi_i + \Sigma_{Y}(t) \mathbb{E} [ Y^{\varepsilon}(t) ] \xi_i \otimes \Sigma_{Y}(t) \mathbb{E} [ Y^{\varepsilon}(t) ] \xi_i + \Sigma_{Y}(t) \mathbb{E} [ Y^{\varepsilon}(t) ] \xi_i \otimes \delta^v \Sigma(t) \xi_i\\
        &\qquad\qquad\quad + \delta^v \Sigma(t) \xi_i \otimes \Sigma_{Y}(t) \mathbb{E} [ Y^{\varepsilon}(t) ] \xi_i + \delta^v \Sigma(t) \xi_i \otimes \delta^v \Sigma(t) \xi_i \bigg ),
    \end{split}
    \end{equation}
    where $(\xi_i)_{i=1,\dots,w}$ denotes the standard basis of $\mathbb{R}^w$, and $\Psi^{\varepsilon} : [0,T] \times \Omega \times \mathbb{R}^w \to H\otimes H$, is given by
    \begin{equation}\label{Psi_epsilon}
        \Psi^{\varepsilon}(t,\xi) = Y^{\varepsilon}(t) \otimes \Sigma_{Y}(t) \mathbb{E} [ Y^{\varepsilon}(t) ] \xi + Y^{\varepsilon}(t) \otimes \delta^v \Sigma(t) \xi + \Sigma_{Y}(t) \mathbb{E} [ Y^{\varepsilon}(t) ] \xi \otimes Y^{\varepsilon}(t) + \delta^v \Sigma(t) \xi \otimes Y^{\varepsilon}(t).
    \end{equation}
\end{theorem}

\begin{remark}
    Regarding the definition of the trace term in equation \eqref{third_variational_equation}, see Remark \ref{remark_trace_term_appendix}.
\end{remark}

\begin{proof}
    First, note that
    \begin{equation}
        \| \mathcal{Y}^{\varepsilon}(t) - \mathcal{Y}^{\varepsilon}(s) \|_{L_2(H)} \leq \| Y^{\varepsilon}(t) \| \|Y^{\varepsilon}(t) - Y^{\varepsilon}(s) \| + \| Y^{\varepsilon}(t) - Y^{\varepsilon}(s) \| \| Y^{\varepsilon}(s) \|,
    \end{equation}
    which shows that $\mathcal{Y}^{\varepsilon}$ is continuous $\mathbb{P}$-almost surely. Moreover,
    \begin{equation}
        \mathbb{E} \left [ \sup_{t\in [0,T]} \| \mathcal{Y}^{\varepsilon}(t) \|_{L_2(H)}^2 \right ] = \mathbb{E} \left [ \sup_{t\in [0,T]} \| Y^{\varepsilon}(t) \|^4 \right ] < \infty.
    \end{equation}
    The progressive measurability of $\mathcal{Y}^{\varepsilon}$ follows directly from the progressive measurability of $Y^{\varepsilon}$. Hence, $\mathcal{Y}^{\varepsilon}\in L^2_{\mathcal{F}}(\Omega; C([0,T];L_2(H)))$.

    Let $A_{\lambda}$ be the Yosida approximation of $A$. Let $Y^{\varepsilon,\lambda}$ be the solution of equation \eqref{first_variational_equation} with $A$ replaced by $A_{\lambda}$, and let $\mathcal{Y}^{\varepsilon,\lambda}(t)$ be the Hilbert--Schmidt operator associated with $Y^{\varepsilon,\lambda}(t)\otimes Y^{\varepsilon,\lambda}(t)$. Let $(e_j)_{j\in\mathbb{N}}$ be an orthonormal basis of $H$. Then $(\mathcal{T}_{e_j\otimes e_k})_{j,k\in\mathbb{N}}$ is an orthonormal basis of $L_2(H)$ and we have
    \begin{equation}\label{derivation_tensor_product}
    \begin{split}
        &\mathrm{d} \langle \mathcal{Y}^{\varepsilon,\lambda}(t) , \mathcal{T}_{e_j\otimes e_k} \rangle_{L_2(H)}\\
        &= \mathrm{d} \left ( \langle Y^{\varepsilon,\lambda}(t), e_j \rangle \langle Y^{\varepsilon,\lambda}(t), e_k \rangle \right )\\
        &= \langle Y^{\varepsilon,\lambda}(t), e_j \rangle \langle A_{\lambda} Y^{\varepsilon,\lambda}(t) + B_X(t) Y^{\varepsilon,\lambda}(t) + B_Y(t) \mathbb{E} [ Y^{\varepsilon,\lambda}(t) ] + \delta^v B(t) ,e_k \rangle \mathrm{d}t\\
        &\quad + \langle Y^{\varepsilon,\lambda}(t), e_j \rangle \langle \Sigma_X(t) Y^{\varepsilon,\lambda}(t) + \Sigma_Y(t) \mathbb{E} [ Y^{\varepsilon,\lambda}(t) ] + \delta^v \Sigma(t) , e_k \rangle \mathrm{d}W(t)\\
        &\quad + \langle Y^{\varepsilon,\lambda}(t), e_k \rangle \langle A_{\lambda} Y^{\varepsilon,\lambda}(t) + B_X(t) Y^{\varepsilon,\lambda}(t) + B_Y(t) \mathbb{E} [ Y^{\varepsilon,\lambda}(t) ] + \delta^v B(t) ,e_j \rangle \mathrm{d}t\\
        &\quad + \langle Y^{\varepsilon,\lambda}(t), e_k \rangle \langle \Sigma_X(t) Y^{\varepsilon,\lambda}(t) + \Sigma_Y(t) \mathbb{E} [ Y^{\varepsilon,\lambda}(t) ] + \delta^v \Sigma(t) , e_j \rangle \mathrm{d}W(t)\\
        &\quad + \sum_{i=1}^{w} \langle ( \Sigma_X(t) Y^{\varepsilon,\lambda}(t) + \Sigma_Y(t) \mathbb{E} [ Y^{\varepsilon,\lambda}(t) ] + \delta^v \Sigma(t) ) \xi_i , e_j \rangle \langle ( \Sigma_X(t) Y^{\varepsilon,\lambda}(t) + \Sigma_Y(t) \mathbb{E} [ Y^{\varepsilon,\lambda}(t) ] + \delta^v \Sigma(t) ) \xi_i , e_k \rangle \mathrm{d}t.
    \end{split}
    \end{equation}
    Now, we note that
    \begin{equation}
    \begin{split}
        &\langle Y^{\varepsilon,\lambda}(t), e_j \rangle \langle A_{\lambda} Y^{\varepsilon,\lambda}(t) , e_k \rangle + \langle A_{\lambda}Y^{\varepsilon,\lambda}(t), e_j \rangle \langle Y^{\varepsilon,\lambda}(t), e_k \rangle\\
        &= \langle ( \mathcal{T}_{Y^{\varepsilon,\lambda}(t) \otimes A_{\lambda} Y^{\varepsilon,\lambda}(t)} + \mathcal{T}_{A_{\lambda} Y^{\varepsilon,\lambda}(t) \otimes Y^{\varepsilon,\lambda}(t)} ) e_j, e_k \rangle\\
        &= \langle \mathcal{T}_{(I\otimes A_{\lambda})(Y^{\varepsilon,\lambda}(t) \otimes Y^{\varepsilon,\lambda}(t))} + \mathcal{T}_{(A_{\lambda} \otimes I)(Y^{\varepsilon,\lambda}(t) \otimes Y^{\varepsilon,\lambda}(t))}, \mathcal{T}_{e_j\otimes e_k} \rangle_{L_2(H)}\\
        &= \langle A_{\lambda} \mathcal{Y}^{\varepsilon,\lambda}(t) + \mathcal{Y}^{\varepsilon,\lambda}(t) A_{\lambda}^*, \mathcal{T}_{e_j \otimes e_k} \rangle_{L_2(H)}
    \end{split}
    \end{equation}
    where in the last step, we used the fact that for two operators $\mathcal{A},\mathcal{B}\in L(H)$, we have $\mathcal{T}_{(\mathcal{A}\otimes \mathcal{B})(f\otimes g)} = \mathcal{T}_{\mathcal{A}f\otimes \mathcal{B}g} = \mathcal{B} \mathcal{T}_{f\otimes g} \mathcal{A}^*$. Indeed, $\langle \mathcal{B} \mathcal{T}_{f\otimes g} \mathcal{A}^* e_j,e_k \rangle = \langle f, \mathcal{A}^* e_j \rangle \langle \mathcal{B} g, e_k \rangle = \langle \mathcal{T}_{\mathcal{A}f \otimes \mathcal{B}g} e_j, e_k \rangle$. Moreover, we have
    \begin{equation}
    \begin{split}
        &\sum_{i=1}^{w} \langle \Sigma_X(t) Y^{\varepsilon,\lambda}(t) \xi_i, e_j \rangle \langle \Sigma_X(t) Y^{\varepsilon,\lambda}(t) \xi_i, e_k \rangle\\
        &= \left \langle \sum_{i=1}^{w} ( \Sigma_X(t)(\cdot) \xi_i ) \mathcal{T}_{Y^{\varepsilon,\lambda}(t) \otimes Y^{\varepsilon,\lambda}(t)} ( \Sigma^*_X(t)(\cdot) \xi_i ), \mathcal{T}_{e_j\otimes e_k} \right \rangle_{L_2(H)}\\
        &= \left \langle \text{Tr} \left ( \Sigma_X(t) \mathcal{Y}^{\varepsilon,\lambda}(t) \Sigma^*_X(t) \right ) , \mathcal{T}_{e_j\otimes e_k} \right \rangle_{L_2(H)}.
    \end{split}
    \end{equation}
    With similar calculations for the remaining terms in \eqref{derivation_tensor_product} we obtain that $\mathcal{Y}^{\varepsilon,\lambda}(t)$ satisfies the equation
    \begin{equation}\label{third_variational_equation_Yosida}
    \begin{cases}
        \mathrm{d}\mathcal{Y}^{\varepsilon,\lambda}(t) = \left [ (A_{\lambda} +B_X(t)) \mathcal{Y}^{\varepsilon,\lambda}(t) + \mathcal{Y}^{\varepsilon,\lambda}(t) (A_{\lambda}^*+B^*_X(t)) + \text{Tr} ( \Sigma_X(t) \mathcal{Y}^{\varepsilon,\lambda}(t) \Sigma^*_X(t) ) + \mathcal{T}_{\Phi^{\varepsilon}}(t) \right ] \mathrm{d}t\\
        \qquad\qquad\qquad + \left [ \Sigma_X(t) \mathcal{Y}^{\varepsilon,\lambda}(t) + \mathcal{Y}^{\varepsilon,\lambda}(t) \Sigma^*_X(t) + \mathcal{T}_{\Psi^{\varepsilon}}(t) \right ] \mathrm{d}W(t)\\
        \mathcal{Y}^{\varepsilon,\lambda}(0) = 0 \in L_2(H),
    \end{cases}
    \end{equation}
    where $\mathcal{T}_{\Phi^{\varepsilon}}$ and $\mathcal{T}_{\Psi^{\varepsilon}}$ are given by \eqref{T_Phi_epsilon_T_Psi_epsilon}.
    
    By Theorem \ref{theorem_yosida_approximation_forward_equation}, $\mathcal{Y}^{\varepsilon,\lambda}$ converges to the mild solution of
    \begin{equation}
    \begin{cases}
        \mathrm{d}\mathcal{Y}^{\varepsilon}(t) = \left [ (A+B_X(t)) \mathcal{Y}^{\varepsilon}(t) + \mathcal{Y}^{\varepsilon}(t) (A^*+B^*_X(t)) + \text{Tr} ( \Sigma_X(t) \mathcal{Y}^{\varepsilon}(t) \Sigma^*_X(t) ) + \mathcal{T}_{\Phi^{\varepsilon}(t)} \right ] \mathrm{d}t\\
        \qquad\qquad + \left [ \Sigma_X(t) \mathcal{Y}^{\varepsilon}(t) + \mathcal{Y}^{\varepsilon}(t) \Sigma^*_X(t) + \mathcal{T}_{\Psi^{\varepsilon}(t)} \right ] \mathrm{d}W(t)\\
        \mathcal{Y}^{\varepsilon}(0) = 0 \in L_2(H).
    \end{cases}
    \end{equation}
    Moreover, since $Y^{\varepsilon,\lambda} \to Y^{\varepsilon}$ as $\lambda\to\infty$, $\mathcal{Y}^{\varepsilon,\lambda}(t)$ also converges to $\mathcal{T}_{Y^{\varepsilon}(t) \otimes Y^{\varepsilon}(t)}$. Thus, the solution of this equation is given by $\mathcal{Y}^{\varepsilon}(t) = \mathcal{T}_{Y^{\varepsilon}(t) \otimes Y^{\varepsilon}(t)}$.
\end{proof}

\subsection{Second Order Adjoint Equation}
    
We introduce the $L_2(H)$-valued second order adjoint equation
\begin{equation}\label{second_order_adjoint}
\begin{cases}
    -\mathrm{d}P(t) = \Big [ P(t) ( A + B_X(t) ) + ( A^* + B_X^*(t) ) P(t) + \text{Tr} ( \Sigma^*_X(t) P(t) \Sigma_X(t) + \Sigma^*_X(t) Q(t) + Q(t) \Sigma_X(t) )\\
    \qquad\qquad\qquad + \mathcal{H}_{XX}(t) \Big ] \mathrm{d}t - Q(t) \mathrm{d}W(t)\\
    P(T) = - M_{XX}(T) \in L_2(H),
\end{cases}
\end{equation}
where $(p,q)$ is the mild solution of \eqref{first_order_adjoint}, the Hamiltonian $\mathcal{H}: [0,T]\times H \times H \times U_{\text{ad}} \times H \times L(\mathbb{R}^w,H) \to \mathbb{R}$ is given by
\begin{equation}
    \mathcal{H}(t,X, Y,u,p,q) = \langle B(t,X,Y,u), p \rangle +\langle \Sigma(t,X,Y,u), q \rangle_{L_2(\mathbb{R}^w,H)} - L(t,X,Y,u)
\end{equation}
and we denote
\begin{equation}
    \mathcal{H}(t) = \mathcal{H}(t,X(t), \mathbb{E}[X(t)],u(t),p(t),q(t)), \qquad \mathcal{H}_{XX}(t) = \mathcal{H}_{XX}(t,X(t), \mathbb{E}[X(t)],u(t),p(t),q(t)).
\end{equation}
Note that equation \eqref{second_order_adjoint} is not of mean-field type in contrast to the first order adjoint equation \eqref{first_order_adjoint}.

\begin{theorem}
    Let Assumptions \ref{assumption_b_sigma} and \ref{assumption_l_m} be satisfied. Then, the second order adjoint equation \eqref{second_order_adjoint} admits a unique mild solution $(P,Q) \in L_{\mathcal{F}}^2(\Omega; C([0,T];L_2(H))) \times L_{\mathcal{F}}^2(\Omega\times [0,T]; L(\mathbb{R}^w,L_2(H)))$ in the sense of Definition \ref{definition_mild_solution_backward_riccati}.
\end{theorem}

\begin{proof}
    The proof can be found in \cite[Theorem 5.4]{guatteri_tessitore_2005}.
\end{proof}

We have the following duality relation.

\begin{proposition}\label{proposition_second_order_adjoint_state_property}
    Let Assumptions \ref{assumption_b_sigma} and \ref{assumption_l_m} be satisfied, and let $\mathcal{Y}^{\varepsilon}(t)$, $t\in [0,T]$, be the solution of equation \eqref{third_variational_equation}. Then, it holds for almost every $\tau\in (0,T)$
    \begin{equation}
    \begin{split}
        \mathbb{E} [ \langle P(T), \mathcal{Y}^{\varepsilon}(T) \rangle_{L_2(H)} ] &= -\mathbb{E} [ M_{XX}(T) Y^{\varepsilon}(T)^2 ]\\
        &= \mathbb{E} \left [ \int_0^T \left ( - \mathcal{H}_{XX}(t) Y^{\varepsilon}(t)^2 + \text{Tr} \left ( \delta^v \Sigma^*(t) P(t) \delta^v \Sigma(t) \right ) \right ) \mathrm{d}t \right ] + o(\varepsilon).
    \end{split}
    \end{equation}
    The exceptional set for $\tau \in (0,T)$ is independent of the set $\mathfrak{F} \in \mathcal{F}_{\tau}$ that is used in the construction of the spike variation in \eqref{spike_variation_definition}.
\end{proposition}

\begin{proof}
    Let $\mathcal{Y}^{\varepsilon,\lambda}$ and $(P^{\lambda},Q^{\lambda})$ denote the solutions of \eqref{third_variational_equation} and \eqref{second_order_adjoint}, respectively, where $A$ is replaced by its Yosida approximation $A_{\lambda}$. By It\^o's formula, we have
    \begin{equation}\label{second_adjoint_state_property_proof}
    \begin{split}
        &\mathbb{E} \left [ \langle P^{\lambda}(T) , \mathcal{Y}^{\varepsilon,\lambda}(T) \rangle_{L_2(H)} \right ]\\
        &= \mathbb{E} \left [ \int_0^T \left ( \langle P^{\lambda}(t), \mathcal{T}_{\Phi^{\varepsilon}(t)} \rangle_{L_2(H)} + \langle Q^{\lambda}(t), \mathcal{T}_{\Psi^{\varepsilon}(t)} \rangle_{L_2(\mathbb{R}^w,L_2(H))} \right ) \mathrm{d}t - \int_0^T \langle \mathcal{H}_{XX}(t), \mathcal{Y}^{\varepsilon,\lambda}(t) \rangle_{L_2(H)} \mathrm{d}t \right ].
    \end{split}
    \end{equation}
    We note that
    \begin{equation}
    \begin{split}
        &\left | \mathbb{E} \left [ \langle P^{\lambda}(T) , \mathcal{Y}^{\varepsilon,\lambda}(T) \rangle_{L_2(H)} \right ] - \mathbb{E} \left [ \langle P(T) , \mathcal{Y}^{\varepsilon}(T) \rangle_{L_2(H)} \right ] \right |\\
        &\leq \mathbb{E} \left [ | \langle P^{\lambda}(T)- P(T) , \mathcal{Y}^{\varepsilon,\lambda}(T) \rangle_{L_2(H)} | + | \langle P(T), \mathcal{Y}^{\varepsilon,\lambda}(T) - \mathcal{Y}^{\varepsilon}(T) \rangle_{L_2(H)} | \right ]\\
        &\leq \mathbb{E} \left [ \| P^{\lambda}(T)- P(T) \|_{L_2(H)} \| \mathcal{Y}^{\varepsilon,\lambda}(T) \|_{L_2(H)} \right ] + \mathbb{E} \left [ \| P(T) \|_{L_2(H)} \| \mathcal{Y}^{\varepsilon,\lambda}(T) - \mathcal{Y}^{\varepsilon}(T) \|_{L_2(H)} \right ]\\
        &\leq \mathbb{E} \left [ \| P^{\lambda}(T)- P(T) \|_{L_2(H)}^2 \right ]^{\frac12} \mathbb{E} \left [ \| \mathcal{Y}^{\varepsilon,\lambda}(T) \|^2_{L_2(H)} \right ]^{\frac12} + \mathbb{E} \left [ \| P(T) \|_{L_2(H)}^2 \right ]^{\frac12} \mathbb{E} \left [ \| \mathcal{Y}^{\varepsilon,\lambda}(T) - \mathcal{Y}^{\varepsilon}(T) \|_{L_2(H)}^2 \right ]^{\frac12},
    \end{split}
    \end{equation}
    which, by Theorems \ref{theorem_yosida_approximation_forward_equation} and \ref{theorem_yosida_approximation_backward_Riccati}, see also \cite[Theorem 5.4]{guatteri_tessitore_2005}, tends to zero as $\lambda$ tends to $\infty$. Thus, arguing similarly for the remaining terms in \eqref{second_adjoint_state_property_proof}, we obtain
    \begin{equation}
        \mathbb{E} \left [ \langle P(T) , \mathcal{Y}^{\varepsilon}(T) \rangle_{L_2(H)} \right ] = \mathbb{E} \left [ \int_0^T \left ( \langle P(t), \mathcal{T}_{\Phi^{\varepsilon}(t)} \rangle_{L_2(H)} + \langle Q(t), \mathcal{T}_{\Psi^{\varepsilon}(t)} \rangle_{L_2(\mathbb{R}^w,L_2(H))} \right ) \mathrm{d}t - \int_0^T \langle \mathcal{H}_{XX}(t), \mathcal{Y}^{\varepsilon}(t) \rangle_{L_2(H)} \mathrm{d}t \right ].
    \end{equation}
    Noting that
    \begin{equation}
        \mathbb{E} \left [ \int_0^T \left ( \langle P(t), \mathcal{T}_{\Phi^{\varepsilon}(t)} \rangle_{L_2(H)} + \langle Q(t), \mathcal{T}_{\Psi^{\varepsilon}(t)} \rangle_{L_2(\mathbb{R}^w,L_2(H))} \right ) \mathrm{d}t \right ] = \mathbb{E} \left [ \int_0^T \text{Tr} \left (\delta^v \Sigma^*(t) P(t) \delta^v \Sigma(t) \right ) \mathrm{d}t \right ] + o(\varepsilon)
    \end{equation}
    concludes the proof.
\end{proof}

\section{Stochastic Maximum Principle}\label{section_stochastic_maximum_principle}

Now we are in a position to state and prove our main result.

\begin{theorem}
    Let Assumptions \ref{assumption_b_sigma} and \ref{assumption_l_m} be satisfied. Let $u(\cdot)$ be an optimal control, and let $(p,q)$ and $(P,Q)$ be the solutions of the first and second order adjoint equations \eqref{first_order_adjoint} and \eqref{second_order_adjoint}, respectively. Then, it holds
    \begin{equation}
        \mathcal{H}(t,X(t),\mathbb{E}[X(t)],u(t),p(t),q(t)) - \mathcal{H}(t,X(t),\mathbb{E}[X(t)],v,p(t),q(t)) - \frac12 \text{Tr} (\delta^v \Sigma^*(t) P(t) \delta^v \Sigma(t)) \geq 0
    \end{equation}
    $\mathrm{d}t\otimes \mathbb{P}$-almost everywhere, for all $v\in U_{\text{ad}}$.
\end{theorem}

\begin{proof}
    By Proposition \ref{proposition_cost_functional_expansion} we have for $v \in U_{\text{ad}}$, for almost every $\tau\in (0,T)$, and for $\mathfrak{F} \in \mathcal{F}_{\tau}$,
    \begin{equation}
    \begin{split}
        0 \leq J(u^{\varepsilon}(\cdot)) - J(u(\cdot))&= \mathbb{E} [ M_X(T) ( Y^{\varepsilon}(T) + Z^{\varepsilon}(T)) ] + \frac12 \mathbb{E} [ M_{XX}(T) Y^{\varepsilon}(T)^2 ] + \mathbb{E} [ M_{Y}(T) ] \mathbb{E} [ Y^{\varepsilon}(T) + Z^{\varepsilon}(T) ] \\
        &\quad + \mathbb{E} \left [ \int_0^T L_X(t) (Y^{\varepsilon}(t) + Z^{\varepsilon}(t)) \mathrm{d}t \right ]+ \frac12 \mathbb{E} \left [ \int_0^T L_{XX}(t) Y^{\varepsilon}(t)^2 \mathrm{d}t \right ]\\
        &\quad + \int_0^T \mathbb{E} \left [ L_{Y}(t) \right ] \mathbb{E} [ Y^{\varepsilon}(t) + Z^{\varepsilon}(t) ] \mathrm{d}t + \mathbb{E} \left [ \int_0^T \delta^v L(t) \mathrm{d}t \right ] + o(\varepsilon).
    \end{split}
    \end{equation}
    Using the first order adjoint state property with $Y^{\varepsilon}$ \eqref{adjoint_state_property_Y}, we obtain
    \begin{equation}
    \begin{split}
        0\leq J(u^{\varepsilon}(\cdot)) - J(u(\cdot))&= \mathbb{E} [ M_X(T) Z^{\varepsilon}(T) ] + \frac12 \mathbb{E} [ M_{XX}(T) Y^{\varepsilon}(T)^2 ] + \mathbb{E} \left [ M_{Y}(T) \right ] \mathbb{E} [ Z^{\varepsilon}(T) ] \\
        &\quad + \mathbb{E} \left [ \int_0^T L_X(t) Z^{\varepsilon}(t) \mathrm{d}t \right ]+ \frac12 \mathbb{E} \left [ \int_0^T L_{XX}(t) Y^{\varepsilon}(t)^2 \mathrm{d}t \right ] + \int_0^T  \mathbb{E} \left [ L_{Y}(t) \right ] \mathbb{E} [ Z^{\varepsilon}(t) ] \mathrm{d}t \\
        &\quad + \mathbb{E} \left [ \int_0^T \left ( \delta^v L(t) - \langle p(t), \delta^v B(t) \rangle - \langle q(t), \delta^v \Sigma(t) \rangle_{L_2(\mathbb{R}^w,H)} \right ) \mathrm{d}t \right ] + o(\varepsilon).
    \end{split}
    \end{equation}
    Next, using the first order adjoint state property with $Z^{\varepsilon}$ \eqref{adjoint_state_property_Z}, we obtain
    \begin{equation}
    \begin{split}
        0&\leq J(u^{\varepsilon}(\cdot)) - J(u(\cdot))\\
        &= \frac12 \mathbb{E} \left [ M_{XX}(T) Y^{\varepsilon}(T)^2 + \int_0^T L_{XX}(t) Y^{\varepsilon}(t)^2 \mathrm{d}t \right ] + \mathbb{E} \left [ \int_0^T \left ( \delta^v L(t) - \langle p(t), \delta^v B(t) \rangle - \langle q(t), \delta^v \Sigma(t) \rangle_{L_2(\mathbb{R}^w,H)} \right ) \mathrm{d}t \right ]\\
        &\quad - \mathbb{E} \Bigg [ \int_0^T \left \langle p(t) , \delta^v B_X(t) Y^{\varepsilon}(t) + \frac12 B_{XX}(t) Y^{\varepsilon}(t)^2 \right \rangle + \left \langle q(t), \delta^v \Sigma_X(t) Y^{\varepsilon}(t) + \frac12 \Sigma_{XX}(t) Y^{\varepsilon}(t)^2 \right \rangle_{L_2(\mathbb{R}^w,H)} \mathrm{d}t \Bigg ] + o(\varepsilon).
    \end{split}
    \end{equation}
    Noting that
    \begin{equation}
        \mathbb{E} \left [ \int_0^T \left \langle p(t), \delta^v B_X(t) Y^{\varepsilon}(t) \right \rangle + \left \langle q(t), \delta^v \Sigma_X(t) Y^{\varepsilon}(t) \right \rangle_{L_2(\mathbb{R}^w,H)} \mathrm{d}t \right ] = o(\varepsilon)
    \end{equation}
    we derive
    \begin{equation}\label{cost_functional_expansion_2}
    \begin{split}
        0\leq J(u^{\varepsilon}(\cdot)) - J(u(\cdot)) &= \frac12 \mathbb{E} \left [ M_{XX}(T) Y^{\varepsilon}(T)^2 + \int_0^T L_{XX}(t) Y^{\varepsilon}(t)^2 \mathrm{d}t \right ] \\
        &\quad - \frac12 \mathbb{E} \Bigg [ \int_0^T \left \langle p(t) , B_{XX}(t) Y^{\varepsilon}(t)^2 \right \rangle + \left \langle q(t), \Sigma_{XX}(t) Y^{\varepsilon}(t)^2 \right \rangle_{L_2(\mathbb{R}^w,H)} \mathrm{d}t \Bigg ]\\
        &\quad + \mathbb{E} \left [ \int_0^T \left ( \delta^v L(t) - \langle p(t), \delta^v B(t) \rangle - \langle q(t), \delta^v \Sigma(t) \rangle_{L_2(\mathbb{R}^w,H)} \right ) \mathrm{d}t \right ] + o(\varepsilon).
    \end{split}
    \end{equation}
    Applying the second order adjoint state property, Proposition \ref{proposition_second_order_adjoint_state_property}, we obtain
    \begin{equation}
    \begin{split}
        0 &\leq J(u^{\varepsilon}(\cdot)) - J(u(\cdot))\\
        &= \mathbb{E} \left [ \int_0^T \left ( \delta^v L(t) - \langle p(t), \delta^v B(t) \rangle - \langle q(t), \delta^v \Sigma(t) \rangle_{L_2(\mathbb{R}^w,H)} - \frac12 \text{Tr} \left ( \delta^v \Sigma^*(t) P(t) \delta^v \Sigma(t) \right ) \right ) \mathrm{d}t \right ] + o(\varepsilon).
    \end{split}
    \end{equation}
    It remains to pass from this integral variational inequality to the pointwise maximum principle. To this end, we first note that the integrand in the previous inequality is only nonzero for $t\in [\tau,\tau+\varepsilon]$ and $\omega \in \mathfrak{F}$. Thus, we obtain
    \[
        0 \leq \mathbb{E} \left [ \int_{\tau}^{\tau+\varepsilon} \mathbf{1}_{\mathfrak{F}} \left ( \delta^v L(t) - \langle p(t), \delta^v B(t) \rangle - \langle q(t), \delta^v \Sigma(t) \rangle_{L_2(\mathbb{R}^w,H)} - \frac12 \text{Tr} \left ( \delta^v \Sigma^*(t) P(t) \delta^v \Sigma(t) \right ) \right ) \mathrm{d}t \right ] + o(\varepsilon).
    \]
    Dividing by $\varepsilon$ and letting $\varepsilon \downarrow 0$, Lebesgue's differentiation theorem and the Fubini theorem yield, for a.e. $\tau\in[0,T]$,
    \[
        \mathbb E \left [ \mathbf 1_{\mathfrak{F}} \left ( \delta^v L(\tau) - \langle p(\tau), \delta^v B(\tau) \rangle - \langle q(\tau), \delta^v \Sigma(\tau) \rangle_{L_2(\mathbb{R}^w,H)} - \frac12 \text{Tr} \left ( \delta^v \Sigma^*(\tau) P(\tau) \delta^v \Sigma(\tau) \right ) \right ) \right ] \ge 0 \quad \forall \mathfrak{F} \in\mathcal{F}_{\tau}.
    \]
    Since $\mathfrak{F} \in\mathcal F_\tau$ is arbitrary, it follows that
    \[
        \mathbb E \left [ \delta^v L(\tau) - \langle p(\tau), \delta^v B(\tau) \rangle - \langle q(\tau), \delta^v \Sigma(\tau) \rangle_{L_2(\mathbb{R}^w,H)} - \frac12 \text{Tr} \left ( \delta^v \Sigma^*(\tau) P(\tau) \delta^v \Sigma(\tau) \right ) \;\middle|\; \mathcal{F}_{\tau} \right ] \ge 0.
    \]
    Since the integrand is $\mathcal{F}_{\tau}$-measurable, this implies
    \[
        \delta^v L(\tau) - \langle p(\tau), \delta^v B(\tau) \rangle - \langle q(\tau), \delta^v \Sigma(\tau) \rangle_{L_2(\mathbb{R}^w,H)} - \frac12 \text{Tr} \left ( \delta^v \Sigma^*(\tau) P(\tau) \delta^v \Sigma(\tau) \right ) \geq 0 \quad \mathbb P\text{-a.s.}
    \]
    for a.e. $\tau\in[0,T]$. Hence
    \[
        \delta^v L(\tau) - \langle p(\tau), \delta^v B(\tau) \rangle - \langle q(\tau), \delta^v \Sigma(\tau) \rangle_{L_2(\mathbb{R}^w,H)} - \frac12 \text{Tr} \left ( \delta^v \Sigma^*(\tau) P(\tau) \delta^v \Sigma(\tau) \right ) \geq 0 \quad \mathrm{d}t \otimes \mathbb P\text{-a.e.}
    \]
    for every fixed $v\in U_{\text{ad}}$. We now make the exceptional null set independent of $v$. Since $U$ is a separable Banach space, $U_{\text{ad}}$ endowed with the relative topology is separable. Let $D=\{v_k\}_{k\ge1}\subset U_{\text{ad}}$ be a countable dense subset. Applying the previous argument to every $v_k$ and taking the union of the corresponding exceptional sets, we find a null set $N\subset[0,T]\times\Omega$ such that
    \[
        \delta^{v_k} L(\tau) - \langle p(\tau), \delta^{v_k} B(\tau) \rangle - \langle q(\tau), \delta^{v_k} \Sigma(\tau) \rangle_{L_2(\mathbb{R}^w,H)} - \frac12 \text{Tr} \left ( \delta^{v_k} \Sigma^*(\tau) P(\tau) \delta^{v_k} \Sigma(\tau) \right ) \geq 0 \quad \forall k\ge1,\quad \forall (\tau,\omega ) \notin N.
    \]
    Using the continuity of $B$, $\Sigma$ and $L$ with respect to the control variable and the density of $D$ in $U_{\text{ad}}$, we obtain
    \[
        \delta^{v} L(\tau) - \langle p(\tau), \delta^{v} B(\tau) \rangle - \langle q(\tau), \delta^{v} \Sigma(\tau) \rangle_{L_2(\mathbb{R}^w,H)} - \frac12 \text{Tr} \left ( \delta^{v} \Sigma^*(\tau) P(\tau) \delta^{v} \Sigma(\tau) \right ) \geq 0 \quad \forall v\in U_{\text{ad}}, \quad \mathrm{d}t \otimes \mathbb P\text{-a.e.}
    \]
    Finally, using the definition of the Hamiltonian, this inequality is equivalent to
    \[
        \mathcal{H}(t,X(t),\mathbb E[X(t)],u(t),p(t),q(t)) - \mathcal{H}(t,X(t),\mathbb E[X(t)],v,p(t),q(t)) - \frac12
    \text{Tr} \left ( \delta^v \Sigma^*(t) P(t) \delta^v \Sigma(t) \right ) \geq 0,
    \]
    for every $v\in U_{\text{ad}}$, $\mathrm{d}t\otimes \mathbb P$-a.e. This concludes the proof.
\end{proof}

\appendix

\section{McKean--Vlasov BSDEs in Infinite Dimensions}\label{appendix_McKean_Vlasov_BSDE}

In this section, we discuss the existence and uniqueness of mild solutions for McKean--Vlasov BSDEs in infinite dimensional spaces in order to treat the first order adjoint equation \eqref{first_order_adjoint} and the family of operator-valued BSDEs \eqref{bsde_dual}. Since this result is of independent interest, we work in a more general setting than in the main part of the paper. To this end, let $\Xi$ and $K$ be separable Hilbert spaces. For a Hilbert space $\mathcal{X}$, let $\mathcal{P}_2(\mathcal{X})$ be the Wasserstein space over $\mathcal{X}$, i.e., the space of all probability measures with finite second moment, endowed with the Wasserstein distance $\mathbf{d}_{2}:\mathcal{P}_2(\mathcal{X}) \times \mathcal{P}_2(\mathcal{X}) \to \mathbb{R}$.  For more details on the Wasserstein space, see e.g. \cite{ambrosio_gigli_savare_2005,villani_2009}. Let $(W(t))_{t\in [0,T]}$ be a cylindrical Wiener process in $\Xi$ on some probability space $(\Omega,\mathcal{F},(\mathcal{F}_t)_{t\in [0,T]}, \mathbb{P})$, where $(\mathcal{F}_t)$ is its natural filtration augmented by all $\mathbb{P}$-null sets. For a random variable $U:\Omega \to \mathcal{X}$, let $\mathcal{L}(U)$ denote the law of $U$, i.e., the measure $\mathbb{P} \circ U^{-1}$ on $\mathcal{X}$.

Let $F:[0,T]\times \Omega \times K \times L_2(\Xi,K) \times \mathcal{P}_2(K) \times \mathcal{P}_2(L_2(\Xi,K)) \to K$ and let $\xi \in L^2(\Omega;K)$ be $\mathcal{F}_T$-measurable. We consider the BSDE
\begin{equation}\label{general_bsde}
\begin{cases}
    -\mathrm{d}U(t) = \left [ \mathcal{G} U(t) + F(t,U(t),V(t),\mathcal{L}(U(t)),\mathcal{L}(V(t))) \right ] \mathrm{d}t - V(t) \mathrm{d}W(t),\quad t\in [0,T]\\
    U(T) = \xi,
\end{cases}
\end{equation}
and impose the following assumptions on $\mathcal{G}$ and $F$.

\begin{assumption}\label{assumption_general_bsde}
    \begin{enumerate}[label=(\roman*)]
    \item $\mathcal{G}: \mathcal{D}(\mathcal{G}) \subset K \to K$ is an unbounded operator that generates a $C_0$-semigroup $(e^{t\mathcal{G}})_{t\geq 0}$.
    \item There is a constant $C\geq 0$ such that $\mathbb{P}$-almost surely it holds
        \begin{equation}
            \| F(t,u,v,\mu,\beta) - F(t,u',v',\mu',\beta') \|_K \leq C \left ( \|u-u' \| _K+ \|v-v' \|_{L_2(\Xi,K)} + \mathbf{d}_2(\mu,\mu') + \mathbf{d}_2(\beta,\beta') \right )
        \end{equation}
        for all $t\in [0,T]$, $u,u'\in K$, $v,v'\in L_2(\Xi,K)$, $\mu,\mu'\in \mathcal{P}_2(K)$ and $\beta,\beta'\in \mathcal{P}_2(L_2(\Xi,K))$.
        \item $F(\cdot,0,0,\delta_0,\delta_0) \in L^2_{\mathcal{F}}([0,T]\times \Omega;K)$.
    \end{enumerate}
\end{assumption}

\begin{theorem}\label{BSDE_existence_of_solution}
    Let Assumption \ref{assumption_general_bsde} be satisfied. Then, for every $\mathcal{F}_T$-measurable $\xi \in L^2(\Omega;K)$, the McKean--Vlasov BSDE \eqref{general_bsde} has a unique mild solution, i.e., there is a unique pair of processes $(U,V)\in L_{\mathcal{F}}^2(\Omega;C([0,T];K)) \times L_{\mathcal{F}}^2 (\Omega \times [0,T]; L_2(\Xi,K))$ such that
    \begin{equation}\label{general_bsde_mild_solution_biba}
        U(t) = e^{(T-t)\mathcal{G}} \xi + \int_t^T e^{(s-t)\mathcal{G}} F(s,U(s),V(s),\mathcal{L}(U(s)),\mathcal{L}(V(s))) \mathrm{d}s - \int_t^T e^{(s-t) \mathcal{G}} V(s) \mathrm{d}W(s).
    \end{equation}
    for all $t\in [0,T]$, $\mathbb{P}$-almost surely. Moreover,
    \begin{equation}\label{a_priori_estimate_bsde}
        \mathbb{E} \left [ \sup_{t\in [0,T]} \| U(t) \|_K^2 + \int_0^T \| V(t) \|_{L_2(\Xi,K)}^2 \mathrm{d}t \right ] \leq C \mathbb{E} \left [ \| \xi \|_K^2 + \int_0^T \| F(t,0,0,\delta_0,\delta_0) \|_K^2 \mathrm{d}t \right ]
    \end{equation}
    for some constant $C$ that depends only on the unbounded operator $\mathcal{G}$, the Lipschitz constant of $F$, and on $T$.
\end{theorem}

\begin{proof}
    The proof of existence and uniqueness of a mild solution follows along the same lines as the one for the finite dimensional case, see e.g. \cite[Theorem 3.1]{buckdahn_li_peng_2009}. It is divided into two steps:
    
    {\bf Step 1}: For $\alpha >0$, $U\in
    L_{\mathcal{F}}^2(\Omega;C([0,T];K))$ and $V \in  L_{\mathcal{F}}^2 (\Omega \times [0,T]; L_2(\Xi,K))$, we introduce the following equivalent exponential norms:
    \begin{equation}
        \| U \|^2_{\alpha,K} := \mathbb{E} \left [ \sup_{t\in [0,T]} e^{ 2\alpha t} \| U(t) \|^2 _K \right ],  \qquad  \| V \|_{\alpha,L_2(\Xi,K)}^2 := \mathbb{E} \left [ \int_0^T e^{ 2\alpha t} \| V(t) \|^2 _{L_2(\Xi,K)} \mathrm{d}t \right ].
    \end{equation}
    For $U' \in L_{\mathcal{F}}^2(\Omega;C([0,T];K))$ and $V' \in L_{\mathcal{F}}^2(\Omega\times [0,T];L_2(\Xi,K))$, we consider equation \eqref{general_bsde_mild_solution_biba} with $U(s)$ and $V(s)$ in $F$ replaced by $U'(s)$ and $V'(s)$, respectively, that is
    \begin{equation}\label{general_bsde_1}
        U(t) = e^{(T-t)\mathcal{G}} \xi + \int_t^T e^{(s-t)\mathcal{G}} F(s,U'(s),V'(s),\mathcal{L}(U'(s)),\mathcal{L}(V'(s))) \mathrm{d}s - \int_t^T e^{(s-t) \mathcal{G}} V(s) \mathrm{d}W(s).
    \end{equation}
    By \cite[Theorem 3.1]{hu_peng_1991} or \cite[Theorem 4.4]{guatteri_tessitore_2005}, this equation has a unique mild solution
    \begin{equation}
        (U^{U',V'},V^{U',V'}) \in L_{\mathcal{F}}^2(\Omega;C([0,T];K)) \times L_{\mathcal{F}}^2 (\Omega \times [0,T]; L_2(\Xi,K)).
    \end{equation}

    {\bf Step 2}: We introduce the mapping
    \begin{equation}\label{phi-contraction}
    \begin{split}
        \Phi :L_{\mathcal{F}}^2(\Omega;C([0,T];K)) \times L_{\mathcal{F}}^2 (\Omega \times [0,T]; L_2(\Xi,K))&\to L_{\mathcal{F}}^2(\Omega;C([0,T];K)) \times L_{\mathcal{F}}^2 (\Omega \times [0,T]; L_2(\Xi,K))\\
        ({U}',V') &\mapsto \left ( U^{U',V'},V^{U',V'} \right ),
    \end{split}
    \end{equation}
    where $( U^{U',V'},V^{U',V'} )$ is the solution of equation \eqref{general_bsde_1}. Next, let us show that $\Phi$ is a contraction on the space $L_{\mathcal{F}}^2(\Omega;C([0,T];K)) \times L_{\mathcal{F}}^2 (\Omega \times [0,T]; L_2(\Xi,K))$ endowed with the norm $\|(U,V)\|_{\alpha} = (\|U\|_{\alpha,K}^2 + \|V\|^2_{\alpha,L_2(\Xi,K)})^{1/2}$. To this end, let $(U_1,V_1) = \Phi(U'_1,V'_1)$ and $(U_2,V_2) = \Phi(U'_2,V'_2)$. Following \cite[Lemma 2.1]{hu_peng_1991} or \cite[Theorem 4.4, equations (4.13)-(4.14)]{guatteri_tessitore_2005}, we have
    \begin{equation}
    \begin{split}
    & \mathbb{E} \left [ \sup_{t\in [0,T]} e^{2 \alpha t} \| U_1(t) -U_2(t) \|_K^2 + \int_0^T  e^{2 \alpha t}\| V_1(t) -V_2(t)\|_{L_2(\Xi,K)}^2 \mathrm{d}t \right ] \\
    & \leq \frac{C}{\alpha} \mathbb{E}\left [ \int_0^T e^{2 \alpha t} \left \| F(t,U'_1(t),V'_1(t), \mathcal{L}(U'_1(t)), \mathcal{L}(V'_1(t)) )- F(t,U'_2(t),V'_2(t), \mathcal{L}(U'_2(t)), \mathcal{L}(V'_2(t)) ) \right \|_K^2 \mathrm{d}t \right ],
    \end{split}
    \end{equation}
    where the constant $C\geq 0$ depends only on $\mathcal{G}$. Using the Lipschitz continuity of $F$ to estimate the right-hand side, we obtain
    \begin{equation}
    \begin{split}
        & \mathbb{E} \left [ \sup_{t\in [0,T]} e^{2 \alpha t} \| U_1(t) -U_2(t) \|_K^2 + \int_0^T  e^{2 \alpha t}\| V_1(t) -V_2(t)\|_{L_2(\Xi,K)}^2 \mathrm{d}t \right ] \\ 
        & \leq \frac{C}{\alpha} \mathbb{E} \left [ \int_0^T e^{2 \alpha t} \left ( \| U'_1(t) - U'_2(t) \|_K^2 + \| V'_1(t) - V'_2(t) \|_{L_2(\Xi,K)}^2 \right ) \mathrm{d}t \right ] \\
        & \quad + \frac{C}{\alpha} \int_0^T e^{2 \alpha t} \left ( \mathbf{d}^2_2( \mathcal{L}(U'_1(t)), \mathcal{L}(U'_2(t)) ) + \mathbf{d}^2_2( \mathcal{L}(V'_1(t)), \mathcal{L}(V'_2(t)) ) \right ) \mathrm{d}t \\
        &\leq \frac{C}{\alpha} \mathbb{E} \left [ \sup_{t\in [0,T]} e^{2 \alpha t} \| U'_1(t) -U'_2(t) \|_K^2 + \int_0^T  e^{2 \alpha t}\| V'_1(t) -V'_2(t)\|_{L_2(\Xi,K)}^2 \mathrm{d}t \right ]
    \end{split}
    \end{equation}
    for some constant $C\geq 0$ that depends only on $\mathcal{G}$, the Lipschitz constant of $F$, and $T$. Choosing $\alpha = 4C$ shows that $\Phi$ is a contraction with constant $1/2$ on the space $L_{\mathcal{F}}^2(\Omega;C([0,T];K)) \times L_{\mathcal{F}}^2 (\Omega \times [0,T]; L_2(\Xi,K))$ endowed with the norm $\|(U,V)\|_{\alpha}$. Thus, by Banach's fixed point theorem $\Phi$ has a unique fixed point which is the mild solution of equation \eqref{general_bsde}.
    
    Now, let us turn to the proof of the a priori estimate \eqref{a_priori_estimate_bsde} for $(U,V)$. Let $(U_0,V_0)$ be the solution of 
    \begin{equation}
    \begin{cases}
        -\mathrm{d}U_0(t) = \left [ \mathcal{G} U_0(t) + F(t,0,0,\delta_0,\delta_0) \right ] \mathrm{d}t - V_0(t) \mathrm{d}W(t),\quad t\in [0,T]\\
        U_0(T) = \xi.
    \end{cases}
    \end{equation}
    As in \cite[Equations (4.13), (4.14)]{guatteri_tessitore_2005}, we obtain the estimate
    \begin{equation}\label{a_priori_zero}
        \mathbb{E} \left [ \sup_{t\in [0,T]} e^{2 \alpha t} \| U_0(t) \|_K^2 + \int_0^T e^{2\alpha t} \| V_0(t) \|_{L_2(\Xi,K)}^2 \mathrm{d}t  \right ] \leq C e^{2 \alpha T} \mathbb{E} \left [ \| \xi \|_K^2 \right ] + \frac{C}{\alpha} \mathbb{E} \left [ \int_0^T e^{2\alpha s} \| F(s,0,0,\delta_0,\delta_0) \|_K^2 \mathrm{d}s \right ],
    \end{equation}
    for some constant $C$ that depends only on $\mathcal{G}$. Noting that $(U,V) = \lim_{n\to\infty} \Phi^n(0,0)$, we obtain
    \begin{equation}
        \| (U,V) \|_{\alpha} \leq \lim_{n\to \infty} \sum_{k=1}^n \| \Phi^k(0,0) - \Phi^{k-1}(0,0) \|_{\alpha} \leq \lim_{n\to \infty} \sum_{k=1}^n \left (\frac12 \right )^{k-1} \| \Phi(0,0) \|_{\alpha} = 2 \| (U_0,V_0) \|_{\alpha},
    \end{equation}
    where $\alpha$ is chosen as above. Since the norms $\| \cdot \|_{\alpha}$ are equivalent with the standard norm on $L_{\mathcal{F}}^2(\Omega;C([0,T];K)) \times L_{\mathcal{F}}^2 (\Omega \times [0,T]; L_2(\Xi,K))$, this together with \eqref{a_priori_zero} concludes the proof.
\end{proof}

Now, we are interested in equation \eqref{general_bsde} when the unbounded operator $\mathcal{G}$ is replaced by its Yosida approximation, i.e.,
\begin{equation}\label{general_bsde_yosida}
\begin{cases}
    -\mathrm{d}U^{\lambda}(t) = \left [ \mathcal{G}_{\lambda} U^{\lambda}(t) + F(t,U^{\lambda}(t),V^{\lambda}(t),\mathcal{L}(U^{\lambda}(t)),\mathcal{L}(V^{\lambda}(t))) \right ] \mathrm{d}t - V^{\lambda}(t) \mathrm{d}W(t),\quad t\in [0,T]\\
    U^{\lambda}(T) =  \xi.
\end{cases}
\end{equation}

Note that for this equation, since $\mathcal{G}_{\lambda}$ is bounded, the notion of mild solution and classical solution coincide.

\begin{theorem}\label{theorem_general_bsde_yosida}
    Let Assumption \ref{assumption_general_bsde} be satisfied. Then, it holds
    \begin{equation}
        \lim_{\lambda\to\infty} \mathbb{E} \left [ \sup_{t\in [0,T]} \| U^{\lambda}(t) - U(t) \|_K^2 \right ] = 0,\quad \lim_{\lambda\to\infty} \mathbb{E} \left [ \int_0^T \| V^{\lambda}(t) - V(t) \|_{L_2(\Xi,K)}^2 \mathrm{d}t \right ] = 0.
    \end{equation}
\end{theorem}

\begin{proof}
    The result follows from the parameter-dependent contraction principle \cite[Theorem 10.1]{zabczyk_1999}, exactly as in \cite{guatteri_tessitore_2005}. Indeed, the maps $\Phi_{\lambda}$ that we obtain in place of $\Phi$ in the second step of the proof of Theorem \ref{BSDE_existence_of_solution} when $\mathcal{G}$ is replaced by $\mathcal{G}_{\lambda}$, turn out to be contractions for the same $\alpha$.
\end{proof}

\section{Operator-Valued Stochastic Differential Equations}\label{appendix_Operator_Valued_BSDE}

In this appendix, let us collect a few results about operator-valued (backward) SDEs, some of which are taken from \cite{guatteri_tessitore_2005}. We work in the same, more general framework, as in Appendix \ref{appendix_McKean_Vlasov_BSDE}. Moreover, let $L_2^{\text{sym}}(K) \subset L_2(K)$ be the subset of symmetric Hilbert--Schmidt operators.

First, we consider the forward equation
\begin{equation}\label{forward_stochastic_riccati}
\begin{cases}
    \mathrm{d}\mathcal{Y}(t) = [ \mathcal{G} \mathcal{Y}(t) + \mathcal{Y}(t) \mathcal{G}^* + A_{\texttt{\#}}(t) \mathcal{Y}(t) + \mathcal{Y}(t) A^*_{\texttt{\#}}(t) + \text{Tr} ( C(t) \mathcal{Y}(t) C^*(t) ) + S(t) ] \mathrm{d}t\\
    \qquad\qquad\qquad + [ C(t) \mathcal{Y}(t) + \mathcal{Y}(t) C^*(t) + T(t) ] \mathrm{d}W(t)\\
    \mathcal{Y}(0) = N \in L_2^{\text{sym}}(K),
\end{cases}
\end{equation}
where we impose the following assumptions on the coefficients.

\begin{assumption}\label{assumption_forward_operator_equation}
    \begin{enumerate}[label=(\roman*)]
        \item $\mathcal{G}: \mathcal{D}(\mathcal{G}) \subset K \to K$ is an unbounded operator that generates a $C_0$-semigroup $(e^{t\mathcal{G}})_{t\geq 0}$.
        \item $A_{\texttt{\#}} : [0,T] \times \Omega \to L(K)$ and $C : [0,T]\times \Omega \to L(K,L_2(\Xi,K))$ are essentially bounded in $(t,\omega)\in [0,T]\times \Omega$, strongly measurable and predictable.
        \item $S\in L_{\mathcal{F}}^2([0,T]\times \Omega; L_2^{\text{sym}}(K))$ and $T\in L_{\mathcal{F}}^2([0,T]\times\Omega; L_2(\Xi,L_2^{\text{sym}}(K)))$ are predictable.
    \end{enumerate}
\end{assumption}

\begin{remark}\label{remark_trace_term_appendix}
    Note that for $C: [0,T]\times\Omega\to L(K,L_2(\Xi,K))$ and an orthonormal basis $(\xi_i)_{i\in\mathbb{N}}$ of $\Xi$, we have for $k\in K$
    \begin{equation}
        C(t)(k) = \sum_{i=1}^{\infty} C_i(t)(k) \langle \xi_i,\cdot \rangle_{\Xi}
    \end{equation}
    where $C_i : [0,T]\times\Omega \to L(K)$ is given by
    \begin{equation}
        C_i(t)(k) := C(t)(k) \xi_i,
    \end{equation}
    which draws the connection with \cite[Hypothesis 2.1 (A3)]{guatteri_tessitore_2005}. Moreover, the trace term in equation \eqref{forward_stochastic_riccati} is understood as
    \begin{equation}
        \text{Tr}(C(t) \mathcal{Y}(t) C^*(t)) = \sum_{i=1}^{\infty} C_i(t) \mathcal{Y}(t) C_i^*(t).
    \end{equation}
\end{remark}

\begin{definition}\label{definition_mild_solution_forward_equation}
    A process $\mathcal{Y} \in L^2_{\mathcal{F}}(\Omega; C([0,T];L_2^{\text{sym}}(K)))$ is a mild solution of equation \eqref{forward_stochastic_riccati} if
    \begin{equation}
    \begin{split}
        \mathcal{Y}(t) &= e^{t\mathcal{G}}N e^{t\mathcal{G}^*} + \int_0^t e^{(t-s)\mathcal{G}} S(s) e^{(t-s)\mathcal{G}^*} \mathrm{d}s + \int_0^t e^{(t-s)\mathcal{G}} ( A_{\texttt{\#}}(s) \mathcal{Y}(s) + \mathcal{Y}(s) A^*_{\texttt{\#}}(s)) e^{(t-s)\mathcal{G}^*} \mathrm{d}s\\
        &\quad + \int_0^t e^{(t-s)\mathcal{G}} \text{Tr} \left ( C(s) \mathcal{Y}(s) C^*(s) \right ) e^{(t-s)\mathcal{G}^*} \mathrm{d}s + \int_0^t e^{(t-s)\mathcal{G}} ( C(s) \mathcal{Y}(s) + \mathcal{Y}(s) C^*(s) + T(s) ) e^{(t-s)\mathcal{G}^*} \mathrm{d}W(s)
    \end{split}
    \end{equation}
    for every $t\in [0,T]$, $\mathbb{P}$-almost surely.
\end{definition}
We have the following result.
\begin{theorem}\label{theorem_yosida_approximation_forward_equation}
    Let Assumption \ref{assumption_forward_operator_equation} be satisfied. Then, for every $N\in L_2^{\text{sym}}(K)$, equation \eqref{forward_stochastic_riccati} has a unique mild solution in the sense of Definition \ref{definition_mild_solution_forward_equation}. Moreover, if $\mathcal{G}$ is replaced by its Yosida approximation $\mathcal{G}_{\lambda}$, equation \eqref{forward_stochastic_riccati} has a unique classical solution $\mathcal{Y}^{\lambda} \in L^2_{\mathcal{F}}(\Omega; C([0,T];L_2^{\text{sym}}(K)))$ and
    \begin{equation}
        \lim_{\lambda\to\infty} \mathbb{E} \left [ \sup_{t\in [0,T]} \| \mathcal{Y}^{\lambda}(t) - \mathcal{Y}(t) \|_{L_2(K)}^2 \right ] = 0.
    \end{equation}
\end{theorem}

\begin{proof}
    As in \cite[Section 5.1]{guatteri_tessitore_2005}, we introduce the family $(e^{t\hat{\mathcal{G}}})_{t\geq 0}$ of linear operators $e^{t\hat{\mathcal{G}}} : L(K) \to L(K)$ by setting
    \begin{equation}
        e^{t\hat{\mathcal{G}}} X := e^{t\mathcal{G}} X e^{t\mathcal{G}^*}, \quad t\geq 0, \quad X\in L(K).
    \end{equation}
    By \cite[Lemma 5.1]{guatteri_tessitore_2005}, if we restrict it to $L_2^{\text{sym}}(K)$, this family is a $C_0$-semigroup. Let us denote by $\hat{\mathcal{G}} : \mathcal{D}(\hat{\mathcal{G}}) \subset L_2^{\text{sym}}(K) \to L_2^{\text{sym}}(K)$ its infinitesimal generator in $L_2^{\text{sym}}(K)$. Then, we have
    \begin{equation}
        \langle \hat{\mathcal{G}} Xx,y \rangle_K = \langle Xx,\mathcal{G}y \rangle_K + \langle X\mathcal{G}x,y \rangle_K, \quad X\in \mathcal{D}(\hat{\mathcal{G}}), \quad x,y\in \mathcal{D}(\mathcal{G}).
    \end{equation}
    Using this setup, the result follows from standard results regarding Yosida approximations and the convergence of the solutions of the corresponding SDEs, see \cite[Theorem 3.3]{guatteri_tessitore_2005}.
\end{proof}

Now, let us consider the BSDE
\begin{equation}\label{backward_stochatic_riccati}
\begin{cases}
    -\mathrm{d}P(t) = [ \mathcal{G}^* P(t) + P(t) \mathcal{G} + A^*_{\texttt{\#}}(t)P(t) + P(t) A_{\texttt{\#}}(t)\\
    \qquad\qquad\qquad + \text{Tr} \left ( C^*(t) Q(t) + Q(t) C(t) + C^*(t) P(t) C(t) \right ) + S(t) ] \mathrm{d}t - Q(t) \mathrm{d}W(t)\\
    P(T) = M\in L_2^{\text{sym}}(K).
\end{cases}
\end{equation}

\begin{definition}\label{definition_mild_solution_backward_riccati}
    A pair of processes $(P,Q) \in L^2_{\mathcal{F}}(\Omega; C([0,T];L_2^{\text{sym}}(K))) \times L^2_{\mathcal{F}}(\Omega\times [0,T]; L_2(\Xi,L_2^{\text{sym}}(K)))$ is a mild solution of equation \eqref{backward_stochatic_riccati} if
    \begin{equation}
    \begin{split}
        P(t) &= e^{(T-t)\mathcal{G}^*} M e^{(T-t)\mathcal{G}} + \int_t^T e^{(s-t)\mathcal{G}^*} S(s) e^{(s-t)\mathcal{G}} \mathrm{d}s + \int_t^T e^{(s-t)\mathcal{G}^*} \left ( A^*_{\texttt{\#}}(s) P(s) + P(s) A_{\texttt{\#}}(s) \right ) e^{(s-t)\mathcal{G}} \mathrm{d}s\\
        &\quad + \int_t^T e^{(s-t)\mathcal{G}^*} \text{Tr} \left ( C^*(s) P(s) C(s) + C^*(s) Q(s) + Q(s) C(s) \right ) e^{(s-t)\mathcal{G}} \mathrm{d}s\\
        &\quad - \int_t^T e^{(s-t)\mathcal{G}^*} Q(s) e^{(s-t)\mathcal{G}} \mathrm{d}W(s)
    \end{split}
    \end{equation}
    for all $t\in [0,T]$, $\mathbb{P}$-almost surely.
\end{definition}

We recall the following result from \cite[Theorem 5.4]{guatteri_tessitore_2005}.

\begin{theorem}\label{theorem_yosida_approximation_backward_Riccati}
    Let Assumption \ref{assumption_forward_operator_equation} be satisfied. Then, for every $\mathcal{F}_T$-measurable $M\in L^2(\Omega;L_2^{\text{sym}}(K))$, equation \eqref{backward_stochatic_riccati} has a unique mild solution in the sense of Definition \ref{definition_mild_solution_backward_riccati}. Moreover, if $\mathcal{G}$ is replaced by its Yosida approximation $\mathcal{G}_{\lambda}$, equation \eqref{backward_stochatic_riccati} has a unique classical solution $(P^{\lambda},Q^{\lambda}) \in L_{\mathcal{F}}^2(\Omega; C([0,T];L_2^{\text{sym}}(K))) \times L_{\mathcal{F}}^2(\Omega \times [0,T]; L_2(\Xi, L_2^{\text{sym}}(K)))$, and
    \begin{equation}
        \lim_{\lambda\to\infty} \mathbb{E} \left [ \sup_{t\in [0,T]} \| P^{\lambda}(t) - P(t) \|^2_{L_2(K)} \right ] = 0, \quad \lim_{\lambda\to\infty} \mathbb{E} \left [ \int_0^T \| Q^{\lambda}(t) - Q(t) \|^2_{L_2(\Xi , L_2(K))} \mathrm{d}t \right ] = 0.
    \end{equation}
\end{theorem}

\section*{Declarations}

\textbf{Competing Interests}\; The authors have no relevant financial or non-financial interests to disclose.

\section*{Acknowledgements}

The author would like to thank the anonymous referees for their careful reading of the manuscript and for their valuable comments.

\end{document}